\documentclass[a4paper,11pt,oneside,final]{amsart}
\usepackage{fullpage}
\usepackage[english]{babel}
\usepackage[T1]{fontenc}
\usepackage{mathtools}
\usepackage{enumerate}

\usepackage{etex}
\usepackage[utf8x]{inputenc}
\usepackage{amsmath,amssymb,amsfonts,amsthm}
\usepackage{mathrsfs}
\usepackage{hyperref}
\usepackage{tikz-cd}
\usepackage{ctable}

\theoremstyle{plain}
\newtheorem{thm}{Theorem}[section]
\newtheorem{cor}[thm]{Corollary}
\newtheorem{lem}[thm]{Lemma}
\newtheorem{prop}[thm]{Proposition}
\newtheorem{statement}[thm]{Statement}

\theoremstyle{definition}
\newtheorem{rem}[thm]{Remark}

\newtheorem*{condition*}{Condition}

\newtheorem{defin}[thm]{Definition}

\newtheorem*{example*}{Example}
\numberwithin{equation}{section}
\numberwithin{thm}{section}

\DeclareMathOperator{\Gal}{Gal}

\DeclareMathOperator{\bfk}{\mathbf{k}}
\DeclareMathOperator{\bfj}{\mathbf{j}}

\DeclareMathOperator{\bfa}{\mathbf{a}}

\DeclareMathOperator{\bfX}{\mathbf{X}}

\DeclareMathOperator{\bfA}{\mathbf{A}}
\DeclareMathOperator{\bfx}{\mathbf{x}}

\DeclareMathOperator{\Aut}{Aut}

\DeclareMathOperator{\disc}{disc}
\DeclareMathOperator{\G}{Gal}

\DeclareMathOperator{\GL}{GL}

\newcommand{\cov}{\mathrm{cov}}
\newcommand{\Abb}{\mathbb{A}}
\newcommand{\Pbb}{\mathbb{P}}

\newcommand{\Del}{\Delta}
\newcommand{\pfrak}{\mathfrak{p}}
\newcommand{\Ecal}{\mathcal{E}}
\newcommand{\Lcal}{\mathcal{L}}
\newcommand{\Ocal}{\mathcal{O}}
\newcommand{\Kcal}{\mathcal{K}}

\newcommand{\Q}{\mathbb{Q}}
\newcommand{\Z}{\mathbb{Z}}
\newcommand{\F}{\mathbb{F}}

\newcommand{\C}{\mathbb{C}}
\newcommand{\A}{\mathbb{A}}

\renewcommand{\GL}{\text{GL}}

\newcommand{\bSpec}{\textbf{Spec}}

\newcommand{\calO}{\mathcal{O}}

\newcommand{\bx}{\mathbf{x}}

\newcommand{\bz}{\mathbf{z}}

\newcommand{\ba}{\mathbf{a}}

\renewcommand{\bx}{\mathbf{x}}
\newcommand{\bX}{\mathbf{X}}

\newcommand{\ep}{\varepsilon}
\newcommand{\al}{\alpha}
\newcommand{\del}{\delta}

\newcommand{\ga}{\gamma}

\newcommand{\sig}{\sigma}
\newcommand{\beq}{\begin{equation}}
\newcommand{\eeq}{\end{equation}}
\newcommand{\ndiv}{\nmid}
\newcommand{\con}{\equiv}
\newcommand{\modd}[1]{\; ( \text{mod} \; #1)}

\newcommand{\red}{\mathrm{red}}
\newcommand{\lin}{\mathrm{lin}}

\definecolor{pink}{rgb}{1,.2,.6}
\definecolor{orange}{rgb}{0.7,0.3,0}
\definecolor{blue}{rgb}{.2,.6,.75}
\definecolor{green}{rgb}{.4,.7,.4}
\definecolor{purple}{RGB}{127,0,255}
 \definecolor{gray}{RGB}{211,211,211}

\begin{document}
\title{Genuine and strongly genuine polynomials:\\ With an application to the persistence of Galois groups under specialization}

\author[Bonolis]{Dante Bonolis}
\address{Duke University, 120 Science Drive, Durham NC 27708}
\email{dante.bonolis@duke.edu}

\author[Pierce]{Lillian B. Pierce}
\address{Duke University, 120 Science Drive, Durham NC 27708}
\email{lillian.pierce@duke.edu}

\author[Woo]{Katharine Woo}
\address{450 Serra Mall, Building 380, Stanford CA 94305}
\email{khwoo98@stanford.edu }

\begin{abstract}
 We develop the theory of strongly $n$-genuine polynomials $F(Y,X_1,\ldots,X_n)$, which have the property that the number of specializations $F(Y,X_1,\bx')$ with $\bx'=(x_2,\ldots,x_n) \in \Z^{n-1}$ (respectively $\bx' \in \F_p^{n-1}$) such that $F(Y,X_1,\bx')$ is reducible over $\overline{\Q}$ (respectively over $\overline{\F}_p$) can be well-controlled quantitatively.  We also develop the theory of a larger class of $n$-genuine polynomials $F(Y,X_1,\ldots,X_n)$, which have the property that the number of specializations $F(Y,X_1,\bx')$ with $\bx' \in \Z^{n-1}$  (respectively $\bx' \in \F_p^{n-1}$) such that $F(Y,X_1,\bx')$ splits completely  over $\overline{\Q}$ (respectively over $\overline{\F}_p$) into factors that are linear in $Y$ can be well-controlled quantitatively. For each of these classes, we prove that there are four equivalent characterizations. As an application, we demonstrate that  $n$-genuine and strongly $n$-genuine polynomials can be used to prove, for any polynomial $F(Y,X_1,\ldots,X_n)$, an upper bound for the number of specializations $F(Y,\bx)$ with $\bx=(x_1,\ldots,x_n) \in \Z^n$    such that the Galois group of the splitting field of $F(Y,\bx)$ over $\Q$ is not isomorphic to the Galois group of the splitting field of $F(Y,X_1,\ldots,X_n)$ over $\Q(X_1,\ldots,X_n)$. We simultaneously prove analogous results   over any  number field.
\end{abstract}
\maketitle
 \vspace{-1cm}

\section{Introduction}\label{sec_introduction}

 Consider a polynomial $F(Y,\bX)=F(Y,X_1,\ldots,X_n) \in \Z[Y,X_1,\ldots,X_n]$. If $F(Y,\bX)$ is irreducible over $\Q(X_1,\ldots,X_n)$, then Hilbert's Irreducibility Theorem    supplies a dense set $V \subset \Q^{n}$ such that for all $\bx \in V$, $F(Y,\bx)$ is irreducible over $\Q$. A quantitative version  would ask for an upper bound on the number of points of bounded height for which $F(Y,\bx)$ is reducible. Analogously, if $G$ is the Galois group of the splitting field of $F(Y,\bX)$ over $\Q(X_1,\ldots,X_n)$,  one can ask a qualitative question---is it generically true that for $\bx \in \Q^n$ the Galois group $G(\bx)$ of the splitting field of $F(Y,\bx)$ over $\Q$ is isomorphic to $G$?---or a quantitative version: an upper bound on the number of  $\bx$ of bounded height such that the Galois group $G(\bx)$ is not isomorphic to $G$.
 
 Both of these  questions can be formulated  in terms of counting points of bounded height in a thin set, a notion described  in a lecture course of J.-P. Serre  at the Coll\`ege de France in 1980-1981, which we cite in the form \cite{Ser97}. 
 To be precise, a subset $M \subset \Abb^n(\Q)$  is thin if there is an algebraic variety $X$ defined over $\Q$ and a morphism $\pi : X \rightarrow \Abb^n$ such that

(i) $M \subset \pi(X(\Q))$, and 

(ii) the fibre of $\pi$ over the generic point is finite and $\pi$ has no rational section over $\Q$.  

\noindent
Serre noted a qualitative statement (see \S \ref{sec_prove_kemma_irred_Galois_lie_in_thin_set} or \cite[\S 9.2-9.3]{Ser97}): 
 
\begin{lem}\label{lemma_irred_Galois_kie_in_thin_set}
Let $K$ be a number field. Let $F(Y,X_1,\ldots,X_n) \in \mathcal{O}_K[Y,X_1,\ldots,X_n]$ be  irreducible   over $K(X_1,\ldots,X_n)$. 

\begin{enumerate}[(i)]
	
\item  There is a thin set $M \subset \mathbb{A}^n$ such that for $(x_1,\ldots,x_n) \not\in M$, $F(Y,x_1,\ldots,x_n)$ is irreducible over $K$. 

\item  There is a thin set $M' \subset \mathbb{A}^n$ such that for $(x_1,\ldots,x_n) \not\in M'$, the Galois group of the splitting field of $F(Y,x_1,\ldots,x_n)$ over $K$  is isomorphic to the Galois group of the splitting field of $F(Y,X_1,\ldots,X_n)$   over $K(X_1,\ldots,X_n)$. 
 \end{enumerate}
\end{lem}

These facts motivate a general quantitative question: how many 
  points of bounded height can lie in a thin set? This was famously asked by Serre, and has stimulated  research ever since (see \S \ref{sec_relation_to_thin_sets}). 
In the present paper, we focus on   quantitative questions related to the persistence of Galois groups under specialization, working over an arbitrary number field. To set notation, let $K$ be a number field  of degree $h$ with ring of integers $\calO_K$.  Fix an integral basis $w_1,\dots,w_h\in \calO_K$. For $x\in \calO_K$ with the unique representation $x=a_1w_1+\dots+a_hw_h$, let $H_K(x) = (\max_i |a_i|)^{h}$. For a vector $\bx\in \calO_K^n$, let $\|\bx\|= \max_{i=1,\dots,n} H_K(x_i).$ Given a polynomial $F(Y,X_1,...,X_n)$   with coefficients in $\calO_K$, let $\|F\|$  denote the maximum absolute norm of any coefficient of $F$, namely for $F(Y,\bX) = \sum_{\Vec{e}} b_{\Vec{e}} Y^{e_0} X_1^{e_1}\dots X_n^{e_n}$, set $\|F\|:= \max_{\Vec{e}} |N_{K/\Q}(b_{\Vec{e}})|.$
One of our aims is to prove the main claim of Cohen
 \cite[Theorem 2.1]{Coh81}, for any polynomial (not required to be irreducible):

\begin{thm}\label{thm_Cohen}
Let   $F(Y,X_1,...,X_n) \in \calO_K[Y,X_1,\ldots,X_n]$   have total degree at most $D$ and Galois group $G$ over $K(X_1,\ldots,X_n)$. For each $\bx \in \calO_K^n$, let $G(\bx)$ denote the Galois group of the splitting field of $F(Y,\bx)$ over $K$. 
There exists a constant $c$   depending only on $n,D,K$ such that  for all $N\geq 3$, 
 \[ \# \{ \bx \in \calO_K^n,  \|\bx\|\leq N: \text{$G(\bx)\not\simeq G$}\} \ll_{n,D,K}  \|F\|^{c} N^{n-\frac{1}{2}}\log N.\]
\end{thm}

This bound has long stood as a baseline upper bound for the number of points in any thin set, and is also frequently cited in the context of persistence of Galois groups under specialization.  The proof we describe here fills a subtle gap in the original proof   when $n \geq 2$ by applying the theory of  $n$-genuine and strongly $n$-genuine polynomials, which are generic classes of polynomials  recently introduced in our work \cite{BPW25x}. 
Stronger  quantitative results than Theorem \ref{thm_Cohen} are now known (see \S \ref{sec_HIT}), but we believe it is valuable to clarify the proof method reported in  \cite[Theorem 2.1]{Coh81}, and we expect the natural features of $n$-genuine and strongly $n$-genuine polynomials characterized in the present manuscript will be broadly useful.  
These features also pertain to persistence of a given  factorization property under specialization  of some of the variables, but now  factoring over $\overline{\Q}$, or over $\overline{\F}_p$ for varying primes $p$. (The latter type of consideration plays an important role in applications of a sieve method, such as the large sieve or the polynomial sieve.)

 More precisely, the class of $n$-genuine polynomials $F(Y,X_1,\ldots,X_n)$ is designed to provide good quantitative control for the number of $\bx' =(x_2,\ldots,x_n) \in \Q^{n-1}$ such that $F(Y,X_1,\bx')$ splits completely (into linear factors with respect to $Y$) over $\overline{\Q}$, as well as for the number of $\bx'   \in \F_p^{n-1}$ such that $F(Y,X_1,\bx')$ splits completely over $\overline{\F}_p$ (for all but finitely many $p$).  The class of strongly $n$-genuine polynomials $F(Y,X_1,\ldots,X_n)$ is designed to provide good quantitative control for the number of $\bx'  \in \Q^{n-1}$ such that $F(Y,X_1,\bx')$ is reducible over $\overline{\Q}$, as well as for the number of $\bx'  \in \F_p^{n-1}$ such that $F(Y,X_1,\bx')$ is reducible over $\overline{\F}_p$ (for all but finitely many $p$). We will in fact develop this theory more generally over any number field.

\subsection{Introducing $n$-genuine and strongly $n$-genuine polynomials}
 Let $K$ be a number field.
  
\begin{defin}\label{dfn_strongly_field}
We say that a finite (nontrivial) extension $M$ of $K(\bX)=K(X_1,\ldots,X_n)$  is an {\bf $n$-genuine extension} if for every $G(Y,\bX)\in K[Y,\bX] = K[Y,X_1,\ldots,X_n]$ such that $$M = K(\bX)[Y]/(G(Y,\bX)),$$ 
$G(Y,\bX)$ has nonzero degree in each of $X_1,\ldots,X_n$. We say that $M$ is a {\bf strongly $n$-genuine extension} of $K(\bX)$ if for all subextensions $M'$ satisfying  $$K(\bX) \subsetneq M' \subset M,$$
    $M'$ is an $n$-genuine extension of $K(\bX)$.
\end{defin}

\begin{defin}\label{dfn_strongly_poly}
   A polynomial $G(Y,X_1,\ldots,X_n)\in K[Y,X_1,\ldots,X_n]$ that is monic in $Y$ and irreducible over $K(X_1,\ldots,X_n)$ is an {\bf $n$-genuine polynomial}   if $K(\bX)[Y]/(G(Y,\bX))$ is an $n$-genuine  extension of $K(\bX).$
A polynomial $G(Y,\bX)\in K[Y,\bX]$ that is monic in $Y$ and irreducible over $K(X_1,\ldots,X_n)$   is a {\bf strongly $n$-genuine polynomial}   if $K(\bX)[Y]/(G(Y,\bX))$ is a strongly $n$-genuine   extension of $K(\bX).$
\end{defin}

 We will later also define a natural generalization of these classes, which specifies when a polynomial $G(Y,X_1,\ldots,X_n)$ or an extension of $K(X_1,\ldots,X_n)$ is $\ell$-genuine, for some $1 \leq \ell \leq n$; see Definition \ref{dfn_ell_genuine_generalized}.
The definitions above remain valid over an arbitrary field; see  Remark \ref{remark_fields}.

 Different characterizations of these classes can be more convenient, based on the application of interest, so  we provide four equivalent characterizations in Theorem \ref{thm_strongly_gen_ppties} (strongly $n$-genuine case) and Theorem \ref{thm_gen_ppties} ($n$-genuine case). Here, we summarize the natural consequences for each class; these consequences played an essential role in \cite{BPW25x} and will play an essential role in the recovery of Theorem \ref{thm_Cohen} in this paper.

To set terminology, for an arbitrary field $\Kcal$, and a given algebraic closure  $\overline{\Kcal}$,  we say a monic polynomial 
 $f(Y,Z) \in \Kcal[Y,Z]$ is reducible over $\overline{\Kcal}$ if we can write
 \[ f(Y,Z) = f_1(Y,Z)f_2(Y,Z) \]
 with $f_1, f_2 \in \overline{\Kcal}[Y,Z]$ and $1 \leq \deg f_1 < \deg f$.
We say $f(Y,Z)$ splits completely over $\overline{\Kcal}$ if  we can write
\[f(Y,Z) = \prod_{j} (Y-Q_j(Z))\]
   with $Q_j(Z)\in \overline{\Kcal}[Z]$ for all $j$. 
   
  For example, the property of being $n$-genuine over $\Q(\bX)$ allows us to control quantitatively how many specializations of a polynomial   split completely over $\overline{\Q}.$ (We also prove versions of the following results over any number field; see Theorem \ref{thm_gen_linear_factor_consequence_k} and Theorem  \ref{thm_strongly_gen_reducible_consequence_k}, respectively.)

\begin{thm}[Genuine]\label{thm_gen_linear_factor_consequence}
Let $n \geq 2$. Let $F(Y,\bX) \in \Z[Y,X_1,\ldots,X_n]$ be an $n$-genuine polynomial of total degree $D$. Then for all $B \geq 1$, 
\[ \# \{\bx' \in \Z^{n-1} \cap [-B,B]^{n-1} : \text{$F(Y,X_1,\bx')$ splits completely over $\overline{\Q}$}\}\ll_{n,D} B^{n-2}.\]
Also, there exists a finite set $\mathcal{E}$ of exceptional primes, with $|\Ecal| \ll_{n,D} \log \|F\|/ \log \log \|F\|$, such that for all $p \not\in \Ecal$,  
\[ \# \{\bx' \in \F_p^{n-1}: \text{$F(Y,X_1,\bx')$ splits completely over $\overline{\F}_p$}\}\ll_{n,D} p^{n-2}.\]
\end{thm}
The property of being strongly $n$-genuine allows us to control, just as effectively, how many specializations are reducible, even though this is (in general) potentially a much larger class than those specializations that split completely:
\begin{thm}[Strongly genuine]\label{thm_strongly_gen_reducible_consequence}
Let $n \geq 2$. Let $F(Y,\bX) \in \Z[Y,X_1,\ldots,X_n]$ be a strongly $n$-genuine polynomial of total degree $D$. Then for all $B \geq 1$, 
\[ \# \{\bx' \in \Z^{n-1} \cap [-B,B]^{n-1} : \text{$F(Y,X_1,\bx')$ is reducible over $\overline{\Q}$}\}\ll_{n,D} B^{n-2}.\]
Also, there exists a finite set $\mathcal{E}$ of exceptional primes, with $|\Ecal| \ll_{n,D} \log \|F\|/ \log \log \|F\|$, such that for all $p \not\in \Ecal$,  
\[ \# \{\bx' \in \F_p^{n-1}: \text{$F(Y,X_1,\bx')$ is reducible over $\overline{\F}_p$}\}\ll_{n,D} p^{n-2}.\]
\end{thm}

\begin{rem}[Generic]\label{remark_generic} The classes of $n$-genuine and strongly $n$-genuine polynomials  are generic, which enhances their utility. Any strongly $n$-genuine extension is $n$-genuine; thus any strongly $n$-genuine polynomial is $n$-genuine. Let $\mathcal{M}(D,k_1,\ldots,k_n)$ denote the moduli space of polynomials in  $\Q[Y,X_1,\ldots,X_n]$ comprised of polynomials that are monic in $Y$ and satisfy $\deg_Y(F) = D$ and $\deg_{X_i}(F) \leq k_i$; this is a closed irreducible subset of the moduli space of polynomials. Strongly $n$-genuine polynomials   are generic in $\mathcal{M}(D,k_1,\ldots,k_n)$; hence  $n$-genuine polynomials are also generic in this sense. This is a consequence of
\cite[Cor. 3.2, 3.3]{BPW25x}; in fact those results confirm that an even smaller family, the strongly $(1,n)$-allowable polynomials, a subset of the strongly $n$-genuine polynomials, are generic in this sense.  
\end{rem}

\subsection{The strategy}
To see how these classes of polynomials assist in proving Theorem \ref{thm_Cohen}, we set the following notation.
 Let us denote by $\Omega_F$ the splitting field of $F(Y,\bX)$ over $K(\bfX)$. 
Define the number field $L_F=\Omega_F\cap\overline{K}$, so there is an intermediate extension $K(\bfX)\subset L_F(\bfX)\subset\Omega_F$. We will  let $M_F(Y,\bfX) \in L_F(\bfX)[Y]$ be the minimal polynomial of $\Omega_F$ over $L_F(\bfX)$; we may assume that $M_F(Y,\bX)$ is monic in $Y$ and irreducible over $L_F(\bX)$. 
Using the properties of strongly $n$-genuine polynomials, we will prove:

\begin{thm}[Special case of Theorem \ref{thm_Cohen}]\label{thm_Cohen_special}
Let  $K/\Q$ be a number field with ring of integers $\Ocal_K$. Let $F(Y,X_1,...,X_n) \in \calO_K[Y,X_1,\ldots,X_n]$   have total degree at most $D$ and Galois group $G$ over $K(X_1,\ldots,X_n)$.   Suppose that
\beq\label{extra_hyp}
\text{the associated minimal polynomial $M_F(Y,\bX)$ is strongly $n$-genuine.}
\eeq
For each $\bx \in \calO_K^n$, let $G(\bx)$ denote the Galois group of the splitting field of $F(Y,\bx)$ over $K$. 
There exists a constant $c$   depending only on $n,D,K$ such that  for all $N \geq 3$, 
 \[ \# \{ \bx \in \calO_K^n,  \|\bx\|\leq N: \text{$G(\bx) \not\simeq G$}\} \ll_{n,D,K}  \|F\|^{c} N^{n-\frac{1}{2}}\log N.\]
\end{thm}
Indeed, with the additional hypothesis (\ref{extra_hyp}), the original proof given in \cite[Thm. 2.1]{Coh81} can proceed. Moreover, we will show that a  weaker alternative hypothesis can replace the extra hypothesis (\ref{extra_hyp}) and still yield the conclusion. The weaker alternative hypothesis is ultimately easier to work with, although more complicated to state; let us call it (*) for the moment. (Precisely, (*) is (\ref{L_X1_special_n*}), and  Theorem \ref{thm_Cohen_special_weaker} proves the analogue of Theorem \ref{thm_Cohen_special} under (*)).
 To recover Theorem \ref{thm_Cohen} in full, 
 we will use the properties of $n$-genuine polynomials in order to reduce any instance of Theorem \ref{thm_Cohen} to a special case where (*) holds and the original proof of \cite[Thm. 2.1]{Coh81} can proceed.
 More precisely, given a linear transformation $\sig \in \GL_n(\Q)$, let $F_\sig(Y,\bX):=F(Y,\sig(\bX))$. We will show that for any $F$ considered by Theorem \ref{thm_Cohen} there is a linear transformation $\sig \in \GL_n(\Q)$ (with small norm) such that the minimal polynomial $M_{F_\sig}(Y,\bX)$ (of the splitting field of $F_\sig$) acquires property (*), so that $F_\sig$ lies in a special case for which the outcome of Theorem \ref{thm_Cohen} is already known (by Theorem \ref{thm_Cohen_special_weaker}). We will then bound the number of $\|\bx\|$ counted by Theorem \ref{thm_Cohen} for $F$, by a related count for $F_\sig$. The fact that the linear transformation $\sig$
has small norm will allow Theorem \ref{thm_Cohen} to inherit the (at most) polynomial dependence on $\|F\|$ from the special case we apply to $F_\sig$.  (For further remarks on dependence on $\|F\|$, see \S \ref{sec_relation_to_thin_sets}
and in particular Remark \ref{remark_dependence}.)

\subsection{Outline of the paper}
In \S \ref{sec_prelim} we briefly describe some relevant previous literature, and record several standard lemmas.
In \S \ref{sec_Cohen_original} we outline the gap for $n \geq 2$ in the original proof of Theorem \ref{thm_Cohen},  illustrate its relation to  Noether's lemma, and show how to fill the gap if we assume an additional hypothesis in a  key lemma. In \S \ref{sec_strongly_gen} we describe the general theory of strongly $n$-genuine polynomials, characterize their essential properties, and prove Theorem \ref{thm_strongly_gen_reducible_consequence} (also over any number field). In \S \ref{sec_apply_strongly_genuine} we briefly show  how this theory allows the original method of \cite{Coh81} to proceed (under an additional hypothesis), resulting in the special case of Theorem \ref{thm_Cohen_special}, as well as the special case Theorem \ref{thm_Cohen_special_weaker} with condition (*). To complete the recovery of Theorem \ref{thm_Cohen} in full, we also require the larger class of $n$-genuine polynomials. Thus in \S \ref{sec_genuine} we describe the general theory of $n$-genuine polynomials, characterize their essential properties, and prove Theorem \ref{thm_gen_linear_factor_consequence} (also over any number field). Finally, in \S \ref{sec_shift_poly} we employ $n$-genuine and strongly $n$-genuine polynomials to complete the proof of Theorem \ref{thm_Cohen}.

\subsection{Notation and terminology}

  For clarity, we reserve $K$ to denote the number field considered in Theorem \ref{thm_Cohen}. We will also let $k$ denote any fixed number field. When we consider an arbitrary field (finite or infinite) our generic notation will be $\Kcal$.

 For a given subset $I \subseteq \{1,\ldots,n\}$, denote $\bX_I = (X_i)_{i \in I}$. For example, if $I=\{2,\ldots,n\}$ then $\bX_I = (X_2,\ldots,X_n)$, so that as a function field $k(\bX_I)=k(X_2,\ldots,X_n)$, and a polynomial $G(Y,\bX_I)$ can depend on $Y$ and $X_2,\ldots,X_n$ but not $X_1$. Denote $I^c=\{1,\ldots,n\} \setminus I$. With a slight abuse of notation with respect to the ordering of variables, we define $G_{\bx_I}(Y,\bX_{I^c}):=G(Y,\bX_{I^c},\bx_{I})$ so that $X_i$ is specialized to $x_i$ for each $i \in I$ for a given $\bx_I \in k^{|I|}$, while $Y$ and $X_i$ for $i \in \{1,\ldots,n\} \setminus I$ are indeterminates.  Given a polynomial $G(Y,X_1,\ldots, X_n)$, we let $\deg G$ denote the total degree of $G$ while $\deg_Y G$ denotes the degree of $G$ as a polynomial in $Y$. Additionally we may specify the total degree of $G$ as a polynomial in a subset of the variables; for example, $\deg_{Y,X_1} G(Y,X_1,\ldots,X_n)$ denotes the total degree  of $G$ as a polynomial in $Y, X_1$. We say $G(Y,X_1,\ldots,X_n)$ is monic in $Y$ if 
 \[G(Y,\bX) = Y^{D_Y} + Y^{D_Y-1}f_{D_Y-1}(\bx) +\cdots + Y f_1(\bX) + f_0(\bX)
 \] 
 for polynomials $f_j(\bX)$.

Given an extension of fields $\Lcal/\Kcal$, $\Lcal$ is said to be a regular extension of $\Kcal$ if the extension is separable, and moreover $\Kcal$ is integrally closed in $\Lcal$, that is, $\overline{\Kcal} \cap \Lcal=\Kcal$. One of the outcomes of the present work is that the notion of being strongly $n$-genuine is equivalent to certain extensions being regular. Precisely, let $k$ be a number field, let $H(Y,X_1,\ldots,X_n) \in \Ocal_k[Y,X_1,\ldots,X_n]$ be irreducible over $k(\bX)=k(X_1,\ldots,X_n)$, and define $\Lcal_H=k(\bX)[Y]/(H(Y,\bX))$. In the case $n=1$ for example, $H(Y,X_1)$ is a strongly $1$-genuine polynomial if and only if $\Lcal_H$ is a regular extension of $k$, namely $\overline{k} \cap \Lcal_H = k$. (This is proved in Theorem \ref{thm_strongly_gen_ppties_n1}; see also Lemma \ref{lemma_at_least_strongly_1_gen}.) More generally, $H(Y,X_1,\ldots,X_n)$ is a strongly $n$-genuine polynomial if and only if  for each subset $I' \subset \{1,\ldots,n\}$ that omits one index, $\Lcal_H$ is a regular extension of $L(\bX_{I'})$ (Theorem \ref{thm_strongly_gen_ppties}).

\subsubsection{Comparison to original notation} For clarity, we distinguish once and for all between the notation of Cohen \cite{Coh81} and Serre \cite{Ser97}, and the notation we employ.  
 In Cohen's work, for a polynomial $F$ with coefficients in $\calO_K$, recall that Theorem \ref{thm_Cohen}  uses $\|F\|$ to denote the maximum absolute norm of any coefficient of $F$. This agrees  conceptually with Cohen's work, but the notation is different.
In the original paper, Cohen uses the notation $|F|$ to denote the maximum absolute norm of the coefficients of $F$, and the notation $\|F\| = \log |F|$.  (Cohen's statement of Theorem \ref{thm_Cohen} also allows the more general setting  that $K/k$ is a finite extension of number fields, and  $\bx \in \calO_k^n$ varies in the ring of integers of the smaller field, but we do not pursue that variant here.)

 The height $\|\bx\|$ we have defined for $\bx \in \Ocal_K^n$ agrees with Cohen's notation. In Serre's original discussion \cite[\S13]{Ser97}, a slightly different height is considered: for $x\in \calO_K$, temporarily define
$$|x| := \max_{\sigma:K\hookrightarrow \C} |\sigma(x)|.$$
For a vector $\bx=(x_1,\dots,x_n) \in \calO_K^n$, then Serre takes $\|x\|_{\mathrm{Serre}} = \max_{1\leq i\leq n} |x_i|.$ In comparison with Cohen's choice of height, note that for our fixed integral basis $w_1,\ldots,w_h \in \Ocal_K$ and any embedding $\sigma:K\hookrightarrow \C$, we have that 
$$|\sigma(x)| \leq |a_1||\sigma(w_1)|+\dots+|a_h||\sigma(w_h)|.$$
Hence, $|x|\ll_{w_1,\dots,w_h,K} H_K(x)^{1/h}.$ On the other hand, note that for  any embedding $\sigma:K\hookrightarrow \C$ we have that 
$$\sigma(x) = a_1 \sigma(w_1)+\dots +a_h\sigma(w_h),$$
and hence
\begin{equation*}
    \begin{pmatrix} 
\sigma_1(x) \\ 
\sigma_2(x) \\ 
\vdots \\ 
\sigma_h(x) 
\end{pmatrix} = 
\begin{pmatrix} 
\sigma_1(w_1) & \sigma_1(w_2) & \dots & \sigma_1(w_h) \\ 
\sigma_2(w_1) & \sigma_2(w_2) & \dots & \sigma_2(w_n) \\ 
\vdots & \vdots & \ddots & \vdots \\ 
\sigma_h(w_1) & \sigma_h(w_2) & \dots & \sigma_h(w_h) 
\end{pmatrix} 
\begin{pmatrix} 
a_1 \\ 
a_2 \\ 
\vdots \\ 
a_h 
\end{pmatrix}.
\end{equation*}
Since $w_1,\dots,w_h$ is an integral basis, the above matrix is invertible and thus we also have that 
$$H_K(x)^{1/h} = \max_{1\leq i\leq h} |a_i| \ll_{w_1,\dots,w_h,K} |x|.$$
Thus, Cohen's choice of height (and hence our choice) is comparable to Serre's height in the sense that $\|x\| \ll_{K,w_1,\dots,w_n} \|x\|_{\mathrm{Serre}}^h\ll_{K,w_1,\dots,w_n} \|x\|.$

 \begin{rem}[Arbitrary fields]\label{remark_fields}
   In this paper, we develop the theory of $n$-genuine and strongly $n$-genuine polynomials in the setting of number fields, but the definitions are valid over an arbitrary field $\Kcal$. To our knowledge, the only proofs in the characterization theorems of \S \ref{sec_strongly_gen} and \S \ref{sec_genuine} which would require modification when working over an arbitrary field rather than a number field would be Lemma \ref{lemma_F_strongly_n_geniune_open_V} (which is used to show (II) for $i_0$ $\Rightarrow$ (IV) for $i_0$ in Theorem \ref{thm_strongly_gen_ppties}), and   Lemma \ref{lemma_open_V_gen} (which is used to prove (II) for $i_0$ $\Rightarrow$ (III) for $i_0$ in Theorem \ref{thm_gen_ppties}).
\end{rem}

 \section{Previous literature and preliminaries}\label{sec_prelim}
 
\subsection{Recent literature beyond Theorem \ref{thm_Cohen}}\label{sec_HIT}

Let $F(Y,X_1,\ldots,X_n) \in \Q[Y,X_1,\ldots,X_n]$ be irreducible over $\Q$, and of degree $d$ in $Y$. Considering $F$ as a polynomial over the function field $\Q(X_1,\ldots,X_n)$, let $\al_1,\ldots,\al_d$ denote its roots in an algebraic closure $\overline{\Q(X_1,\ldots,X_n)}$; the roots are distinct, under the condition that $F(Y,X_1,\ldots,X_n)$ is irreducible. The Galois group $G$ of $F(Y,X_1,\ldots,X_n)$ is a subgroup of the symmetric group $S_d$; that is, it acts on the roots by permutations, and there is an injective homomorphism $\rho: G \hookrightarrow S_d$. If $\bx \in \Z^n$ is chosen such that $F(Y,\bx)$ is also irreducible, then the Galois group $G(\bx)$ of $F(Y,\bx)$ over $\Q$ is a subgroup of $G$, determined up to conjugation. Indeed, let $\Del(\bX)$ denote the discriminant of $F(Y,\bX)$ as a polynomial in $Y$; note that $\Del(\bX) \not\con 0$ since $F(Y,\bX)$ is irreducible over $\Q$. (Indeed, by the Hilbert Irreducibility Theorem, there is some $\bx_0 \in \Q^n$ such that $F(Y,\bx_0)$ is irreducible over $\Q$, hence separable, and thus square-free, so that $\Del(\bx_0) \neq 0.$) For all $\bx \in \Z^n$ such that $\deg F(Y,\bx)=d$ and $\Del(\bx) \neq 0$,  there exist injective homormorphisms $\rho_\bx: G(\bx) \hookrightarrow S_d$ and $\iota: G(\bx) \hookrightarrow G$ such that $\rho_\bx = \rho \circ \iota$; see \cite[Lemma 1]{CasDie17}.  (By applying the trivial bound in Lemma \ref{lemma_Schwartz_domain}, we can see that these two conditions on $\bx$ are verified for all but $\ll_{n,\deg F} B^{n-1}$ values of $\bx \in \Z^n \cap [-B,B]^n,$ for example.)
Thus for any subgroup $K \subseteq G$, one can ask for an upper bound on 
\[ M(K,G;B):=\#\{ \bx \in \Z^n \cap [-B,B]^n: \text{splitting field of $F(Y,\bx)$ over $\Q$ has Galois group $K$}\}.\]
In this setting, Castillo and Dietmann have proved that for any subgroup $K$ of $G$, 
\beq\label{CasDie}
M(K,G;B)\ll_{F,\ep} B^{n-1+\del_K+\ep} \qquad \text{for any $\ep>0$,}
\eeq
in which $\del_K : = [G: K]^{-1}$, where $[G:K]$ denotes the index of $K$ in $G$, so that $\del_K \leq 1/2$ for any proper subgroup.
For $F(Y,\bX)$ irreducible, this implies a sharper result (for certain groups $G$) than Theorem \ref{thm_Cohen} (over $\Q$), since it implies that  
\[ \#\{ \bx \in \Z^n \cap [-B,B]^n: G(\bx) \not\simeq G\}\ll_{F,\ep} B^{n-1+\ga_G + \ep} \qquad \text{for any $\ep>0$},\]
in which 
\[ \ga_G = \max \{ [G:K]^{-1}: \text{$K$ is a proper subgroup of $G$}\}.\]
For $G=S_d$, note that $\ga_G=1/2$, but $\ga_G$ can be smaller for certain other groups. For example, if $G=A_d$ with $d \geq 5$ then $\del_G=1/d$; see \cite{CasDie17} for details. 
Castillo and Dietmann's method  employs Galois resolvents and results from the determinant method on bounding the number of integral points on curves; this strategy generalizes earlier  work of Dietmann \cite{Die12}.  Zywina has also obtained (\ref{CasDie}) over any number field, as long as $K$ is any subset of $G$ that is stable under conjugation (for example, when $K$ is a normal subgroup of $G$); that method employed the larger sieve in place of the large sieve \cite[unpublished]{Zyw10x}. Castillo and Dietmann remarked in \cite{CasDie17} that in principle their method should allow the dependence on $F$ in the implicit constant to be quantified.

Our focus in this paper centers on $n \geq 2$, but we remark that certain improvements have been achieved in the case $n=1$; most recently,  Parades and Sasyk have achieved  a quantification of the implicit constant in (\ref{CasDie}) for $n=1$, and removed the $B^\ep$ factor completely \cite[Thm. 1.4]{ParSas24}. This  employs a version of the $p$-adic determinant method on bounding the number of integral points on curves, and their work is also valid over any global field. If $K$ denotes the global field, the dependence on $F$ in the implicit constant is polylogarithmic in the so-called $K$-relative height of $F$, which in the case $K=\Q$ is simply $\|F\|$, the maximum absolute value of any coefficient of $F$. 

An upper bound (such as Theorem \ref{thm_Cohen}) that explicitly controls the dependence on $\|F\|$ is interesting, for example, because it allows one to deduce an explicit estimate for the ``smallest'' specialization $\bx$ such that $G(\bx) \simeq G$. (We return to the question of dependence on $\|F\|$ in the next section, and in particular in Remark \ref{remark_dependence}.)

\subsection{Relation to thin sets of type II}\label{sec_relation_to_thin_sets}
  Theorem \ref{thm_Cohen} has frequently been cited in the study of thin sets.
Any thin set in $\A^n$ is a finite union of two types of thin set, called Type I and Type II  (see \cite[Ch. 9]{Ser97},  \cite[\S 3.1]{Ser92} or \cite{BPW25x_sur}).

\emph{Type I:} A   thin set $M \subset \Abb^n(\Q)$   is of type I if  there is a Zariski-closed subvariety $V \subsetneq \Abb^n$   such that $M \subset V(\Q)$. 

\emph{Type II:} A  projective thin set $M \subset \Abb^n(\Q)$ 
  is of type II
 if there is an irreducible affine algebraic variety $Z$ over $\Q$ with $\dim  Z =  n$, and a   dominant morphism $\pi: Z \rightarrow \Abb^n$ with generically finite fibres,  of degree $d \geq 2$,  with $M \subset \pi (Z(\Q))$.
   
    The types can also be defined analogously over any number field $K$, and the discussion of this section applies in such generality, although we focus on $\Q$ for simplicity.
There is a useful interpretation   of the Type I/Type II dichotomy in terms of   polynomials, following Serre \cite[\S 9.1]{Ser97}. If $M\subset \mathbb{A}^n(\Q)$ is an affine thin set of type I, then there is a nonconstant    polynomial $G \in \Q[X_1,\ldots,X_n]$   such that 
\[ M \subset \{ \bx \in \Q^n: G(x_1,\ldots,x_n)=0\}.\]
On the other hand, given an irreducible  $F(Y,X_{1},...,X_{n})\in\mathbb{Q}(X_{1},...,X_{n})[Y]$,  a polynomial in $Y$   with $\deg_{Y} F\geq 2$,  then the following set   is an affine thin set of type II:
\beq\label{affine_typeII_polynomial_interp}
  \{\bfx\in  \Q^n:  \text{$\bfx$ not a pole of any coefficient of $F$, $F(Y,\bfx)=0$ is solvable over $\Q$}\} \subset \Abb^n(\Q).\eeq
By replacing  $F(Y,\bX)$ by a multiple $\tilde{F}(Y,\bX)$ of $F$ by an appropriate polynomial in $X_1,\ldots,X_n$ so that $\tilde{F}(Y,\bX) \in \Q[Y,X_1,\ldots,X_n]$, the set depicted above is contained in the set 
\[
\{\bx \in \Q^n: \text{$\tilde{F}(Y,\bx)=0$ is solvable over $\Q$}\}.\]
Thus it is no loss of generality   to assume that within (\ref{affine_typeII_polynomial_interp}), $F$ is   a polynomial $Y,X_1,\ldots,X_n$. Moreover, modulo a thin set of type I, every thin set of type II takes the form of (\ref{affine_typeII_polynomial_interp}); see \cite[Lemma 1.2]{BPW25x_sur}.

For a quantitative statement, we now suppose  $M \subset \mathbb{A}^n(\Z)  $ is a thin set, and for each integral point $\bx=(x_1,\ldots,x_n) \in \Abb^n(\Z)$ define 
$\|\bx\|=\max_{1 \leq i \leq n}|x_i|$. 
Then define the counting function
\beq\label{thin_counting_function_affine_dfn}
N_{\Abb^{n}}(M,B) := \#\{\bx \in M \subset \mathbb{A}^n(\Z): \|\bx\| \leq B\}.
\eeq
Certainly $N_{\Abb^{n}}(M,B) \ll_n B^n$ is trivially true for all $B \geq 1$.  Serre established a baseline upper bound:
for any thin set $M \subset \mathbb{A}^n(\Z)$, for some unspecified $C(M) \geq 1$ and $0<\ga(M)<1$,
\beq\label{thin_set_bound_affine_intro}
N_{\Abb^{n}}(M,B) \leq C(M) B^{n-1/2} (\log (B+2))^{\ga(M)}, \qquad \text{for all $B \geq 1$.}
\eeq
 Serre's   motivation in \cite[Ch. 13 Thm. 1]{Ser97} included   questions of Hilbert irreducibility, and specialization of Galois groups, as well as   hilbertian fields. (A field $k$ is hilbertian precisely when for all $n \geq 1$, $\Pbb^{n}(k)$ is not a thin set;  $\Q$ and all number fields are hilbertian \cite[\S 9.5-9.6]{Ser97}.) Serre observed that (\ref{thin_set_bound_affine_intro}) is in fact sharp (up to the power of log), although he predicted a stronger upper bound in an analogous projective setting. We leave that broader discussion to the recent survey \cite{BPW25x_sur}, but remark here that if we define 
\beq\label{NFB_intro}
N^{\mathrm{cov}}_{\Abb^{n}}(F,B):=\#\{\bfx\in[-B,B]^n \cap \Z^n:  \text{ $F(Y,\bfx)=0$ is solvable over $\Z$}\},
\eeq
  it can be shown that to prove (\ref{thin_set_bound_affine_intro}) it suffices to prove  that for all polynomials  $F$ with $\deg_Y F \geq 2$ that are absolutely irreducible (that is, irreducible over $\overline{\Q}$),
 \beq\label{NFB_nonuniform_intro}
N^\cov_{\Abb^n}(F,B) \leq C(F) B^{n-\frac{1}{2}}  (\log(B+2))^{\ga(F)} \qquad \text{for all $B \geq 1$,}
\eeq
in which $C(F)$ and $\ga(F)$ are positive constants that may depend on $F$. 
Note that Theorem \ref{thm_Cohen}  also implies (\ref{thin_set_bound_affine_intro}). Indeed, if $F(Y,\bX)$ is irreducible and $F(Y,\bx)=0$ is solvable over $\Z$, then $F(Y,\bx)$ is certainly reducible over $\Q$, so that $G(\bx)$ cannot be isomorphic to $G$.   Consequently Theorem \ref{thm_Cohen} implies  (\ref{NFB_nonuniform_intro}) (and hence also a version of (\ref{thin_set_bound_affine_intro})) in which $C(F)$ has at most polynomial dependence on $\|F\|$, and $\ga(F)=1$.   

As remarked above, (\ref{thin_set_bound_affine_intro}) is sharp, as can be demonstrated by considering the polynomial $F(Y,X_1,\ldots,X_n)=Y^2 - (X_1+ \cdots +X_n)$. For certain special shapes of polynomial $F$, better bounds than (\ref{thin_set_bound_affine_intro}) have been obtained, in works such as \cite{Mun09, HBPie12, Bon21, BonPie24withcor} and most recently by Buggenhaut, Cluckers, Salberger, Santens, and Vermeulen \cite{BCSSV25x} and by the present authors \cite{BPW25x}. The former paper \cite{BCSSV25x} improves the main exponent in (\ref{thin_set_bound_affine_intro}) to $n-1$ if $F$ has the special form $
F(Y,\bX) =F_{\mathrm{top}}(Y,\bX) + F_0(Y,\bX)
$
in which $F_0(Y^e,\bX)$ has total degree strictly smaller than $de$, and for some $d \geq 2$,
\[F_{\mathrm{top}}(Y,\bX) = Y^d +Y^{d-1}f_1(\bX) + \cdots + Yf_{d-1}(\bX) + f_d(\bX)\]
is an absolutely irreducible polynomial,
in which each $f_i$ is homogeneous of degree $e \cdot i$, so that $F_{\mathrm{top}}(Y^e,\bX)$ is homogeneous of degree $de$. The bound in \cite{BCSSV25x} exhibits at most polynomial dependence on $\|F\|$. The latter paper \cite{BPW25x}  motivated the initial construction of the classes of $n$-genuine and strongly $n$-genuine polynomials. Its main theorem improves 
the main exponent in (\ref{thin_set_bound_affine_intro}) to $n-1+1/(n+1)$ if for some integer $m \geq 2$, $F(Y,\bX)$ is a polynomial in $Y^m$, and $F$ has the property that for any linear transformation $L \in \GL_n(\Q)$, $F(Y,L(\bX))$ is strongly $n$-genuine. (If $m=1$, the same result is obtained, but conditional on GRH.) The bound in \cite{BPW25x} exhibits weaker dependence on $\|F\|$, in the sense that it has at most polylogarithmic dependence on $\|F\|$; this is qualitatively the same order of dependence on $\|F\|$ as found by Parades and Sasyk \cite{ParSas24} in their work on quantitative HIT, mentioned earlier.

\begin{rem}[Dependence on $\|F\|$]\label{remark_dependence}
Can an upper bound for the number of integral points in an affine thin set be made uniform, that is, independent of the  coefficients of the defining polynomial(s) of the thin set? For an affine thin set of type I, the answer is famously yes (the subject of the well-known Uniform Dimension Growth Conjecture, surveyed in \cite{BPW25x_sur}).  For an affine thin set of type II, the answer is not yet clear, and a subtlety has recently been raised. 
In particular, we have constructed examples in \cite[Thm. 1.6]{BPW25x_sur} that violate a putative uniform upper bound for integral points in affine thin sets of type II.  Precisely, let $n\geq 1$ be given. As $k$ varies over positive integers, there is a family of  thin sets $M_k \in \Abb^n(\Z)$  of type II, defined by polynomials $F_k(Y,X_1,\dots,X_n)$ with $\|F_k\|=k$,   for which the counting function defined in (\ref{thin_counting_function_affine_dfn}) (or in this case, equivalently (\ref{NFB_intro})) has the following property: there is no constant $c>0$ such that 
\beq\label{counterex} 
N_{\Abb^{n}} (M_k, B) \ll_{n,\deg F_k}B^{n-1}(\log B)^{c}
\eeq
can hold as $k \rightarrow \infty$, with an implicit constant independent of $k$. In particular for $n=1$, there is no universal constant $C(n,d)$ (dependent only on dimension and degree) such that $N_{\Abb^{1}} (M, B) \leq C(n,d)$ for all thin sets $M \subset \mathbb{A}^1$ and $B \geq 1$.   The examples do not violate, however, a putative uniform upper bound of the form 
\[N_{\Abb^{n}} (M_k, B) \ll_{n,\deg F_k,\ep}B^{n-1+\ep}\]
for a given $\ep>0$.

Such questions on uniformity are also relevant to studying quantitative Hilbert irreducibility theorems,   by the inclusion properties recorded in Lemma \ref{lemma_irred_Galois_kie_in_thin_set}. The question of uniformity, in this context, has been raised by Yasumoto \cite{Yas88}.  Let $K$ be a number field.  Yasumoto proved that if $F(Y,X,T) \in K[Y,X,T]$ is irreducible, then there is a constant $C(F)$ such that for each $t \in \mathcal{O}_K$, if $F(Y,X,t)$ is irreducible then $F(Y,x_t,t)$ is irreducible for some natural number $x_t < C(F)$ \cite[Thm. 2]{Yas88}.  This exhibits partial uniformity, in that $C(F)$ is independent of $t$.
  Yasumoto furthermore asked  Open Problem 1:  For each $d \geq 1$, is there a constant $C(d)$ such that for every irreducible polynomial $F(Y,T) \in K[Y,T]$ with $\deg F \leq d$, there is a natural number $t < C(d)$ with $F(Y,t)$ irreducible? The observations above, for thin sets, suggest the answer may be no. 

\end{rem}

\subsection{The qualitative result of Lemma \ref{lemma_irred_Galois_kie_in_thin_set}}\label{sec_prove_kemma_irred_Galois_lie_in_thin_set}
For clarity, we recapitulate the proof of \cite[Prop. 1, 2 \S 9.2]{Ser97} for Lemma \ref{lemma_irred_Galois_kie_in_thin_set}; see also \cite[Prop. 3.3.1 and 3.3.5]{Ser92}.
 \begin{proof}
We start by proving (ii): Let $\Omega_{F}$ be the Galois closure of $F(Y,X_{1},...,X_{n})$ over $K(\bfX)$, $L(Y,\bfX)\in \mathcal{O}_{K}[Y,\bfX]$ its minimal polynomial over $K(\bfX)$, and $G=\Gal (\Omega_{F}/K(\bfX))$. Consider $\mathcal{X}=\bSpec (K(\bfX)[Y]/L)=\bSpec(\Omega_{F})$. Then the inclusion $K(\bfX)\hookrightarrow \Omega_{F}$, induces a morphism $\pi: \mathcal{X} \rightarrow\mathbb{A}^{n}$. Then for every proper subgroup $H\leq G$ , one can consider $\mathcal{X}/H=\bSpec ((\Omega_{F})^{H})$. Since $(\Omega_{F})^{G}=K(\bfX)$, we get a dominant morphism 
\[
\pi_{H}:\mathcal{X}/H\rightarrow \mathcal{X}/G=\bSpec ((\Omega_{F})^{G})=\bSpec (K(\bfX))=   \mathbb{A}^{n},
\] 
of degree $[G:H]>1$. Now the set $T \subset \Abb^n(K)$ defined by 
\[
T:=\bigcup_{\substack{H\subset G\\ H\text{ proper subgroup}}}\pi_{H}((\mathcal{X}/H)(K)),
\]
is a thin set. 

It remains to show that if $\bfx\in K^n$ is such that $G(\bfx)$, the Galois group of $F(Y,\bfx)$ over $K$, is strictly smaller than $G$, then $\bfx\in T$. In what follows we are going to denote by $d=\deg L=[\Omega_{F}:K(\bfX)]$, and for every $\bfx$, $\Omega_{\bfx}$ will denote the Galois closure of $F(Y,\bfx)$ over $K$, and $G(\bfx)=\Gal(\Omega_{\bfx}/K)$. Let $\gamma$ be a root of $L(Y,\bfX)$. Then we claim that if $G(\bfx)\not\simeq G$, then 
\[
R(Y):=\prod_{\sigma\in G(\bfx)}(Y-\sigma(\gamma))\not\in K(\bfX)[Y].
\]
Indeed, since $G(\bfx)\not\simeq G$, then $\deg R = |G(\bx)| < |G| = \deg L$, yet $L$ is the minimal polynomial of $\gamma$ over $K(\bX)$; hence $R$   cannot lie in $K(\bX)[Y]$. On the other hand, if we denote $\gamma_{\bfx}=\gamma\mod\bfx$, one gets
\[
r(Y):=\prod_{\sigma\in G(\bfx)}(Y-\sigma(\gamma_{\bfx}))\in K[Y];
\]
this is an element of $K[Y]$ since  $r(Y)$ is certainly in $\Omega_{\bfx}[Y]$ and is invariant under the action of $G(\bx) =\Gal(\Omega_{\bfx}/K)$.

Let $ K(\bfX)\subset M\subset \Omega_{F}$ be the minimal extensions which contains all the coefficients of $R(Y)$, i.e. $R(Y)\in M [Y]$. Then since $M\supsetneq K(\bfX)$, we have $\Gal (\Omega_F/M )=H\subsetneq G$, and $M=(\Omega_{F})^{H}$. Let $S(Y,\bfX)$ be the minimal polynomial of $M$ over $K(\bfX)$, i.e. $K(\bfX)[Y]/S=M=(\Omega_{F})^{H}$. Hence $\mathcal{X}/H=\bSpec (K(\bfX)[Y]/S)$. On the other hand, since $r(Y)\in K[Y]$ (so that specializing to $\bx$ collapses $M$ to $K$), then $S(Y,\bfx)$ is solvable over $K$ and hence $\bfx\in \pi_H((\mathcal{X}/H)(K))$, as we wanted. Thus, if $\bx\in K^n$ satisfies that $G(\bx)\not\simeq G$, we must have that $\bx\in T$.

For part $(i)$, it is enough to observe that if $\bfx\in K$ is such that $F(Y,\bfx)$ is reducible, then $G(\bfx)$ is strictly smaller than $G$, hence one can apply part (ii).
\end{proof}

 \subsection{Useful lemmas}
 We  recall a version of Hilbert's Irreducibility Theorem from \cite[Ch. 13 and Ch. 14 \S 3]{FriJar23} and \cite[Ch. 9, pp. 233-235]{Lan83}.
\begin{lem}[Hilbert]\label{lemma_HIT}
Let $K/\mathbb{Q}$ be a number field and $F_{1},...,F_{s}\in K(X_{1},...,X_{n},T_{1},...,T_{r})$ be  polynomials in $n+r$ variables, irreducible over $K$. There exists a dense subset $U\subset\mathbb{A}^{r}_K$ such that for any $(t_{1},...,t_{r})\in U$, $F_{1}(X_{1},...,X_{n},t_{1},...,t_{r}),...,F_{s}(X_{1},...,X_{n},t_{1},...,t_{r})$ are irreducible over $K$.
\end{lem}
We require a trivial bound: 
 \begin{lem}[Trivial bound, Schwartz-Zippel] \label{lemma_Schwartz_domain}
 Let $A$ be a domain, such as $\Z$, or $\F_p$ for a prime $p$, or the ring of integers $\calO_K$ in a number field $K$, or the finite residue field $K_\pfrak=K/\pfrak$ for a prime ideal $\pfrak$ in $\Ocal_K$. Let $F \in A[X_1,\ldots,X_n]$ be a nonzero   polynomial of degree $e \geq 1$, and $S \subset A$ a finite subset. Then 
\[ \#\{(x_1,\ldots,x_n) \in S^n : F(x_1,\ldots,x_n)=0\} \leq e |S|^{n-1}.\]
\end{lem}
The proof is by induction on dimension, and can be found in many places, such as \cite[Thm. 1]{HB02} or \cite[Lemma 10.1]{BCLP23}. (While \cite[Thm. 1]{HB02} is stated in the setting where $F$ is absolutely irreducible, and  \cite[Lemma 10.1]{BCLP23} is stated for any domain $A$ in the case where $F$ is homogeneous, either proof applies in the present setting.)

\begin{lem}\label{lemma_count_in_ring}
Let $K/\Q $ be a number field with ring of integers $\calO_K$ and $h:=[K:\Q]$. Then for all $B \gg 1$,
\[  \#\{x\in  \calO_K: H_K(x)\leq B\} = (2\lfloor B^{1/h}\rfloor +1)^h\asymp_h B.\]
Given any nonzero rational integer $g$, the number of distinct prime ideals $\pfrak$ in $\Ocal_k$ that divide g is $\ll_{h} \log g/\log \log g$. 
\end{lem}
\begin{proof}
For the first claim, it suffices to observe that if $H_K(x)\leq B$ and we write $x = a_1w_1+\dots+a_hw_h$ for the fixed integral basis $w_1,\ldots,w_h$ of $\Ocal_K$, then $|a_i|\leq B^{1/h}$ for each $i$. Thus, 
\begin{equation*}
    \#\{x\in \calO_K: H_K(x)\leq B\} = \{(a_1,\dots,a_h)\in \Z^h: |a_i|\leq B^{1/h} \forall i\} = (2\lfloor B^{1/h}\rfloor +1)^h. 
\end{equation*}
The second claim uses the fact that given any rational prime $p$, at most $h$ distinct prime ideals $\pfrak$ in $\Ocal_K$ divide $p$ (the extremal situation occurs when $p$ splits completely in $\Ocal_K$)  \cite[Ch. 4 \S 2 Cor. 2]{Nar90}.  Thus the claim follows, from the standard fact that the number of distinct (rational) prime divisors of $g$ is $\ll \log g/ \log \log g$.     
\end{proof}

\section{Original strategy and Noether's lemma}\label{sec_Cohen_original}

A subtle  gap in the argument recorded for Theorem \ref{thm_Cohen} for $n \geq 2$ arises in \cite[Lemma 4.2]{Coh81}. To describe the gap, we temporarily restrict to the case $K=\Q$ for simplicity, and consider the following statement (which we amend with an additional hypothesis in Proposition \ref{prop_Cohen_corrected_appendix}).
 
\begin{statement}[over $\Q$]\label{statement_Cohen}
    
Let $F(Y,X_1,\ldots,X_n) \in \Z[Y,X_1,\ldots,X_n]$ be a squarefree polynomial of total degree $D$. Let $\Omega_F$ be the splitting field of $F(Y,\bX)$ over $\Q(\bX)$ and $L_F:=\Omega_F \cap \overline{\Q}$.  For any prime $p$, for any prime ideal $\mathfrak{P}$ in $\calO_{L_F}$ that divides $p$, let $L_p$ denote the finite residue field $L_F/\mathfrak{P}$.

\begin{enumerate}[(i)]
\item If $n = 1$, \cite[Lemma 4.2(i)]{Coh81} states there exists a finite   set $\mathcal{E}$ of rational primes with $|\mathcal{E}| \ll_{n,D} \log\|F\|/\log\log \|F\|$ such that  for each  prime $p \not\in \mathcal{E}$, the splitting field $\Omega_p$ of $F(Y,X_{1})$ over $L_{p}(X_{1})$ is a regular extension of $L_p$.
\item     If $n \geq 2$,    for each rational prime $p$ and $(x_2,\ldots,x_n) \in \F_p^{n-1}$, let $\Omega_{x_{2},...,x_{n},p}$ denote the splitting field of $F(Y,X_{1},x_{2},...,x_{n})$ over $L_{p}(X_{1})$.
Let 
\begin{align*} 
M(p) = \#\{ (x_{2},\ldots,x_{n})\in\mathbb{F}_{p}^{n-1}: \;&
\text{$
\G(\Omega_{x_{2},...,x_{n},p}, L_{p}(X_{1}))\neq \G(\Omega_F, \Q(\bfX)),$} \\ &\text{or $\Omega_{x_{2},...,x_{n},p}$ is not a regular extension of $L_p$}\}.
\end{align*}
Then \cite[Lemma 4.2(ii)]{Coh81} states there exists a finite   set $\mathcal{E}$ of exceptional primes with $|\mathcal{E}| \ll_{n,D} \log\|F\|/\log\log \|F\|$ such that  for each  prime $p \not\in \mathcal{E}$, 
\beq\label{Cohen_false_claim}
M(p) \ll_{n,D}p^{n-2}.
\eeq
\end{enumerate}

\end{statement}
  We provide a counterexample to (\ref{Cohen_false_claim}) for $n=2$.  Let 
\beq\label{Cohen_counterexample}
F(Y,X_{1},X_{2},X_{3})=(Y-X_{1})^{2}-X_{2}^{2}-X_{3}^{2}.
\eeq
Fix any prime $p \geq 3$. For all pairs $(x_2,x_3) \in \F_p^2$ such that
$x_{2}^{2}+x_{3}^{2}$ is a square modulo $p$, then 
\[
F(Y,X_{1},x_{2},x_{3})=(Y-X_{1})^{2}-x_{2}^{2}-x_{3}^{2}
\]
is reducible over $\F_p$. If true, (\ref{Cohen_false_claim}) would imply that for all $p$ sufficiently large,  $M(p) =O_n(p)$. However,  $M(p) \gg p^2$ for all sufficiently large $p$, since  $x_{2}^{2}+x_{3}^{2}$ is a square modulo $p$  for roughly half the pairs $(x_2,x_3) \in \F_p^2$. To see this, observe that the number of $x_2,x_3$ for which $x_2^2 + x_3^2$ is a square modulo $p$ is 
\[ \frac{1}{2} \sum_{x_2,x_3 \modd{p}} (\chi_p(x_2^2+x_3^2) + 1 ) =: \frac{1}{2}S(\chi_p)+ \frac{1}{2} p^2 \gg p^2,\]
 where $\chi_p (\cdot) $ is the Legendre symbol modulo $p$. The last inequality follows since $|S(\chi_p)|\ll p^{3/2}$, by summing over $x_2$ with square-root cancellation, then over $x_3$ trivially \cite[Thm. 2.2]{Kat02}. Consequently, (\ref{Cohen_false_claim}) cannot hold for this example. 
Next we describe the source of the difficulty, and amend it with an additional hypothesis.

\subsection{Noether's lemma}
Noether's Lemma is a classical result that is used to detect when a divisibility property of a polynomial (e.g. being reducible, or splitting completely) holds over the algebraic closure of a given field; it is particularly useful for studying this property for specializations, and  over $\F_p$ for primes $p$ varying outside a finite exceptional set. We need a refined form of Noether's Lemma, which we cite from \cite[Lemma 2.6]{BPW25x}, based on \cite[Ch. V Thm. 2A]{Sch76}.

\begin{defin}   Let $\Kcal$ be a field, and $\overline{\Kcal}$ a given algebraic closure. Let $F(Y,\bX) \in \Kcal[Y,X_1,\ldots,X_n]$. Let $e$ denote a multi-degree $e = (e_0,e_1,\ldots,e_n)$ with non-negative integral entries, and set $|e| = e_0+e_1+ \cdots +e_n $. For a given multi-degree $e$ with  $1 \leq |e| < \deg F$,  we say that $F(Y,\bX)$ satisfies  divisibility condition $\mathcal{D}(e)$ over $\overline{\Kcal}$ if there exists a factorization $$F(Y,\bX) = G(Y,\bX)H(Y,\bX)$$
    where $G$ and $H$ lie in $\overline{\Kcal}[Y,X_1,\ldots,X_n]$, $\deg H < \deg F$, and $\deg_Y G \leq e_0,$ $\deg_{X_j} G \leq e_j$ for $j = 1,\ldots, n$. 
\end{defin}
For example,   $F(Y,\bX) \in \Kcal [Y,X_1,\ldots,X_n]$ with total degree $D >1$ is absolutely irreducible (that is, irreducible over $\overline{\Kcal}$) precisely when $F$ does not satisfy condition $\mathcal{D}(e)$ for any multi-degree $e$ with $1 \leq |e| < \deg F$.

\begin{lem}[Variant of Noether's Lemma]\label{lemma_Noether}
    Let $\Kcal$ be a field. Fix $D \in \Z$ with $D\geq 2$.
    
(i) Let $\mathcal{D}(e)$ be a divisibility condition for a fixed multi-degree $e =(e_0,e_1,\ldots,e_n)$ with $1 \leq |e|<D$. Then there exist   forms $G_1,...,G_s$ in variables $(A_{i_0,...,i_n})_{i_0+\cdots+i_n\leq D}$ such that a polynomial $$F(Y,X_1,\ldots,X_n) = \sum_{i_0+\cdots+i_n\leq D} a_{i_0,...,i_n} Y^{i_0}X_1^{i_1}\cdots X_n^{i_n}$$
   \beq\label{Noether_condition}\text{satisfies $\mathcal{D}(e)$ over $\overline{\Kcal}$ or is of degree $<D$,}
   \eeq
   if and only if $$G_j((a_{i_0,...,i_n})) = 0, \forall j=1,...,s.$$
The forms $G_1,...,G_s$ depend only on $n, D$ and $e$, and are independent of the field $\Kcal$ in the sense that if $\mathrm{char}(\Kcal)=0$ they have rational integer coefficients and if $\mathrm{char}(\Kcal) = p\neq 0$, the polynomials are obtained by reducing the integral coefficients modulo $p$. 
Moreover, $s \ll_{n,D,e}1$, and $\deg G_j \ll_{n,D,e} 1$ for all $j=1,\ldots,s$. If $\mathrm{char}(\Kcal)=0$, 
\[ \|G_j\| \ll_{n,D,e} 1\] 
  for all $j=1,\ldots,s$. 

  (ii) The same result as (i) holds if (\ref{Noether_condition}) is replaced by: is reducible over $\overline{\Kcal}$ or is of degree $\deg F < D$.
\end{lem}

\begin{rem}\label{remark_Noether_form}
  To be   precise,  Lemma \ref{lemma_Noether}  provides a collection of forms $\{G_1,...,G_s\}$   rather than a single form. Over   any field $\mathcal{K}$ of characteristic zero, by setting $B = G_1^{2 \ell_1} + \cdots + G_s^{2\ell_s}$ for appropriate  $\ell_i \geq 1$ we obtain a single form that vanishes in $\mathcal{K}$  if and only if each form in   $\{G_1,\ldots,G_s\}$ does, and with degree $2r$, where $r$ is the $\mathrm{lcm}$ of the degrees of $G_1,\ldots,G_s$. We will denote such a form $B_\red$ when testing for reducibility and by $B_\lin$ when testing for having a linear factor.  We are not afforded this luxury over certain finite fields, as the sum $\sum_i G_i^{2\ell_i}$ might vanish in the finite field even if not all $G_i$ do; however this issue only occurs at finitely many primes (depending on $\mathcal{D}(e)$). Suppose $p$ is such a prime: in any argument where we consider the vanishing of $B$ (most commonly $B_\red$ or $B_\lin$), it is equivalent to check that  $G_{i}$ vanishes over $\F_p$ for every $i$.  For brevity, we will refer consistently to the output of Lemma \ref{lemma_Noether} as ``a form'', with the convention that in positive characteristic settings, this indicates each form in the collection is tested individually for vanishing.
\end{rem}

\subsection{Relation to Noether's lemma}
The method presented to prove Statement \ref{statement_Cohen} in \cite[Lemma 4.2(ii)]{Coh81} considers (in Cohen's notation) the minimal polynomial, say $g(Y,\bX)$, of the splitting field $\Omega_F$ of $F(Y,\bX)$ over $\overline{\Q}(\bX)$ and writes it as a polynomial in $Y$ and $X_1$ in an expansion of the form 
\[
    g(Y,X_{1},X_{2},...,X_{n})=\sum_{i,j}g_{i,j}(X_{2},...,X_{n})Y^{i}X_{1}^{j}.
    \]
    Then the argument considers $B_{\red}(\{g_{i,j}(X_{2},...,X_{n})\})$ where $B_\red$ is a form provided by Noether's  Lemma \ref{lemma_Noether}   with the property that 
    \[ \text{$B_{\red}(\{g_{i,j}(x_{2},...,x_{n})\}) \neq 0$ iff $g(Y,X_1,x_2,\ldots,x_n)$ is irred. over $\overline{\Q}$ and has degree $=\deg g(Y,\bX)$.}\]
    It is stated  that $B_{\red}(\{g_{i,j}(X_{2},...,X_{n})\})\not\equiv 0$, and this is important to the remainder of the argument. However, under the hypothesis of \cite[Lemma 4.2(ii)]{Coh81}, it can occur that the polynomial $B_{\red}(\{g_{i,j}(X_{2},...,X_{n})\})\equiv 0$; for example this occurs for $F(Y,\bX)$ defined in (\ref{Cohen_counterexample}).

    To see that $B_{\red}(\{g_{i,j}(X_2,...,X_n))\equiv 0$ for (\ref{Cohen_counterexample}), recall Noether's Lemma (ii) provides a form $B_\red$ to test irreducibility over $\overline{K}$ for polynomials of degree $\leq D$, and the construction of $B_\red$ is global in the sense that if  $K$ has characteristic zero then $B_\red$ has rational integer coefficients, and if $K$ has characteristic $p$ then it is obtained by reducing the integral coefficients mod $p$.
    
   Regarding the  example (\ref{Cohen_counterexample}), we can  take $g(Y,\bX) = F(Y,\bX)$. Let $p$ be a prime and consider for any pair $(x_2,x_3) \in \F_p^2$ the specialization $$F(Y,X_1,x_2,x_3) = (Y-X_1)^2 - x_2^2-x_3^2.$$
    This specialization is reducible over the algebraic closure $\overline{\F}_p$ for every choice of $(x_2,x_3) \in \F_p^2$.  Consequently, for every prime $p$ and every $(x_2,x_3)\in \F_p^2$, 
    $$B_{\red}(\{g_{i,j}(x_2,x_3)\})= 0 \bmod p,$$
    so $B_{\red}(\{g_{i,j}(X_2,X_3)) \con 0 \modd{p}$ for all prime $p$.
    Since $B_\red$ is constructed globally, we must have that $B_{\red}(\{g_{i,j}(X_2,X_3)\})\equiv 0$  as a polynomial in $\Z[X_2,X_3]$.

\subsection{Statement \ref{statement_Cohen}   is true with an additional hypothesis}\label{sec_Cohen_add_hyp}
Fortunately, a version of Statement \ref{statement_Cohen}   is true  if we include an additional hypothesis that $B_\red$   does not vanish everywhere. For clarity, we record this precisely.
We work in a general setting over a number field $K$ and with $F(Y,\bfX)\in\calO_{K}[Y,\bfX] $. We set the following notation, which we will use throughout our discussion of Theorem \ref{thm_Cohen}.
 Let us denote by $\Omega_F$ the splitting field of $F(Y,\bX)$ over $K(\bfX)$; for every $\bfx\in\calO_{K}^{n}$ we denote by $\Omega_{\bfx}$ the splitting field of $F(Y,\bfx)$ over $K$.
Let $L_F=\Omega_F\cap\overline{K}$. We have the intermediate extension $K(\bfX)\subset L_F(\bfX)\subset\Omega_F$. We will  let $M_F(Y,\bfX) \in L_F(\bfX)[Y]$ be the minimal polynomial of $\Omega_F$ over $L_F(\bfX)$; we may assume that $M_F(Y,\bX)$ is monic in $Y$ and irreducible over $L_F(\bX)$. For each prime ideal  $\mathfrak{p}$ in $\calO_K$, let $\mathfrak{P}$ be any prime ideal dividing $\mathfrak{p}$ in $\calO_{L_F}$;  let $K_\mathfrak{p}$ denote the finite field extension $K/\mathfrak{p}$ with order $|K_\mathfrak{p}|$, and let $L_\mathfrak{p}$ denote the finite field extension $L_F/\mathfrak{P}$.  

 For any $i_0 \in \{1,\ldots,n\}$ and $I' :=\{1,\ldots,n\}\setminus \{i_0\}$, we can  expand
$M_F(Y,\bX)$ in terms of $Y,X_{i_0}$ as 
\[ M_F(Y,X_{i_0},\bX_{I'}) = \sum_{\substack{\ell,m\\ \ell+m \leq D}} a_{\ell,m}(\bX_{I'})Y^\ell X_{i_0}^m.\]
     Let $B_{\red,i_0}^{M_F}(\bX_{I'}) \in \Z[\bX_{I'}]$  denote the form produced by Noether's Lemma  \ref{lemma_Noether} (ii) applied with respect to  the algebraic closure $\overline{L_F} = \overline{\Q}$,
          with the property 
     \[      \text{$B_{\red,i_0}^{M_F}( \bx_{I'})=0 \Leftrightarrow M_F(Y,X_{i_0},\bx_{I'})$ is reducible over $\overline{\Q}$ or $\deg M_F(Y,X_{i_0},\bx_{I'}) <  \deg_{Y,X_{i_0}} M_F$. }\]
If $n=1$, we take $i_0=1$ (and $I'$ is the empty set), so that $B_{\red,1}^{M_F}$ is an integer, and is nonzero if and only if $M_F(Y,X_1)$ is irreducible over $\overline{\Q}$. Following Remark \ref{remark_Noether_form}, we note that over $\overline{L_F}=\overline{\Q}$ we can take $B_{\red,i_0}^{M_F}(\bX_{I'})$ to be a single form.

\begin{prop}[Correction to Statement \ref{statement_Cohen} over $K$]\label{prop_Cohen_corrected_appendix}
Let $F(Y,\bX) \in \calO_K[Y,X_1,\ldots,X_n]$ be a squarefree polynomial of total degree $D$.  Let $\Omega_F$ be the splitting field of $F(Y,\bX)$ over $K(\bX)$ and  $L_F=\Omega_F \cap \overline{K}$. Let $M_F(Y,\bX) \in L_F(\bfX)[Y]$ denote the minimal polynomial of $\Omega_F$ over $L_F(\bX)$.
 \begin{enumerate}[(i)]
\item  Suppose $n = 1$. {\bf Assume that $B_{\red,1}^{M_F} \neq 0$.} There exists a finite   set $\mathcal{E}$ of exceptional prime ideals in $\mathcal{O}_K$ with $|\mathcal{E}| \ll_{n,D} \log\|F\|/\log\log \|F\|$ such that  for each  prime ideal $\mathfrak{p} \not\in \mathcal{E}$, the splitting field $\Omega_\pfrak$ of $F(Y,X_{1})$ over $L_{\mathfrak{p}}(X_{1})$ is a regular extension of $L_\pfrak$.  
   
 \item  Suppose $n\geq 2$.
For each index $i_0 \in \{1,\ldots,n\}$,  set $I'=\{1,\ldots,n\}\setminus \{i_0\}$. Consider $M_F(Y,\bX)$ as a polynomial in $Y,X_{i_0}$ with coefficients that are polynomials in $\bX_{I'}$.
          {\bf Assume that for some $i_0 \in \{1,\ldots,n\}$, $B_{\red,i_0}^{M_F}(\bX_{I'})$ is not identically zero} as a polynomial in $\bX_{I'}$.
       For each prime ideal  $\mathfrak{p}$ in $\calO_K$, and tuple $\bx_{I'} \in K_{\mathfrak{p}}^{n-1}$, let $\Omega_{\bx_{I'},\mathfrak{p}}$ denote the splitting field of $F(Y,X_{i_0},\bx_{I'})$ over $L_{\mathfrak{p}}(X_{i_0})$. Let
       \begin{align*} 
       M(\mathfrak{p}) = \# \{ \bx_{I'}\in K_{\mathfrak{p}}^{n-1} : & \G(\Omega_{\bx_{I'},\mathfrak{p}}, L_{\mathfrak{p}}(X_{i_0}))\neq \G(\Omega_F, L_F(\bfX)), 
       \\ & \text{or $\Omega_{\bx_{I'},\mathfrak{p}}$ is not a regular extension of $L_\pfrak$.}\} 
\end{align*}
Then there exists a finite   set $\mathcal{E}$ of exceptional prime ideals in $\calO_K$ with $|\mathcal{E}| \ll_{n,D} \log\|F\|/\log\log \|F\|$ such that  for each  prime ideal $\mathfrak{p} \not\in \mathcal{E}$, 
$$M(\mathfrak{p}) \ll_{n,D}|K_{\mathfrak{p}}|^{n-2}.$$
\end{enumerate}
    \end{prop}
 
\begin{proof}
Consider first the case $n \geq 2$. For simplicity, we suppose the hypothesis is true for $i_0=1$, and let $\bX' =(X_2,\ldots,X_n)$. Let  $B_\red^{M_F}(\bX')$ denote the form produced by Noether's lemma, as defined above (for $i_0=1$)

 In the proof of \cite[Lemma 4.2(ii)]{Coh81}, Cohen begins by studying, in the original notation, the minimal polynomial $g$ of $\Omega$ over $\overline{K}(\bX)$. In our notation, this is the polynomial $M_F(Y,\bX)$. If $B_\red^{M_F}(\bX')$ is not identically zero, then the proof of Proposition \ref{prop_Cohen_corrected_appendix} follows verbatim as in Cohen's original argument. Thus the additional hypothesis included in this lemma suffices for Cohen's argument to proceed.

For completeness, we now outline the argument. For simplicity we restrict to the case $K=\Q$ and to those polynomials $F$ for which $L_F := \Omega_F\cap \overline{\Q}$ satisfies $L_F = \Q$. The more general case is   analogous, and Cohen's paper presents the full setting.
 In the case we consider, $M_F(Y,\bX) \in \Z[Y,\bX]$ and for any rational prime $p$, $L_p=\F_p$, although we will continue to use the notation $L_F$ and $L_p$ several times below. 

 By hypothesis $B_{\red}^{M_F}(\bX')$ is not identically zero as a polynomial in $\bX' = (X_2,\ldots,X_n)$.   Consequently, for all rational primes $p$ that do not divide the gcd of the coefficients of $B_{\red}^{M_F}(\bX')$, the reduction of $B_{\red}^{M_F}(\bX')$ modulo $p$ is a   polynomial in $\F_p[\bX']$, not identically zero. Denote this gcd temporarily by $g$,  and denote by $\|B_\red^{M_F}\|$ the maximum absolute value of any coefficient   in the polynomial $B_{\red}^{M_F}(\bX')$. Note that $g \leq \|B_\red^{M_F}\|$, and by Lemma \ref{lemma_Noether}, $\log \|B_\red^{M_F}\| \ll_{n,D} \log \|F\|$. Upon defining the exceptional set $\mathcal{E}$ to be the set of all primes $p|g$, it follows that $|\mathcal{E}| \ll_{n,D} \log \|F\| / \log \log \|F\|$, as claimed.  Since $B_\red^{M_F}$ is a form (recall Remark \ref{remark_Noether_form}), observe that any prime $p$ where $M_F$ vanishes identically will satisfy that $p\in \mathcal{E}$.
 
 Now suppose $p \not\in \mathcal{E}$ is fixed. Again by Lemma \ref{lemma_Noether}, for a given tuple $\bx'=(x_2,\ldots,x_n) \in \F_p^{n-1}$, if   $B_{\red}^{M_F}(x_2,\ldots,x_n)\neq 0$ in $\F_p$ then $M_F(Y,X_1,x_2,\ldots,x_n)$ is irreducible over $\overline{\F}_p$. From this we claim  it follows that $\Omega_{\bx',p}$ is a regular extension of $L_p$ with Galois group $\Gal(\Omega_F,L_F(\bX))$. 

We first show that this implies that $\Omega_{\bx',p}$ is a regular extension of $L_p$, namely that $\overline{L_p} \cap \Omega_{\bx',p}=L_p$. Let us prove the contrapositive (alternatively, see Lemma \ref{lemma_integrally_closed_implies_absirred}). Suppose that $\Omega_{\bx',p}$ is not a regular extension of $L_p$, so there exists a field $L'$ with
  $L_p\subsetneq L' \subset \overline{L_p}\cap \Omega_{\bx',p}.$ Then $L_p(X_1) \subsetneq L'(X_1)$,  so that $[L'(X_1):L_p(X_1)]\geq 2$, which implies $[\Omega_{\bx',p}: L'(X_1)]<[\Omega_{\bx',p}: L_p(X_1)]$. Since $M_F(Y,X_1,\bx')$ is irreducible over $L_p(X_1)$ and generates the extension $\Omega_{\bx',p}$, and since $[\Omega_{\bx',p}: L'(X_1)]<\deg_Y M_F$, we must have that $M_F(Y,X_1,\bx')$ is reducible over $L'$, an extension of $\F_p$, and hence it is reducible over $\overline{\F}_p$.   This establishes the first claim.

 Next, we establish that $\Omega_{\bx',p}$ has Galois group $\Gal(\Omega_F,L_F(\bX))$. Set $z(\bX) \in \overline{L_F(\bX)}$ such that $\Omega_F = L_F(\bX)(z(\bX))$; that is to say $z(\bX)$ is a root of the minimal polynomial $M_F(Y,\bX)$.   Then $\Omega_{\bx',p} = L_p(X_1)(z(X_1,\bx')).$ If $\Gal(\Omega_{\bx',p},L_p(\bX_1))\subsetneq \Gal(\Omega_F,L_F(\bX))$, then this implies that $[\Omega_{\bx',p}:L_p(\bX_1)]<[\Omega_F:L_F(\bX)]$. So, $z(X_1,\bx')$ must satisfy a strictly lower degree minimal polynomial relation over $L_p(X_1)$, say $H(Y,X_1)$, where $H(Y,X_1) \in L_p(X_1)[Y]$ with $\deg_Y H < \deg_Y M_F(Y,X_1,\bX')$. However, since $M_F(z(X_1,\bx'),X_1,\bx')=0$, this would imply that $H(Y,X_1)\mid M_F(Y,X_1,\bx')$ which is a contradiction, since by construction $M_F(Y,X_1,\bx')$ is irreducible over $\overline{\F}_p=\overline{L_p}$. So, we must have that the Galois group is $\Gal(\Omega_F,L_F(\bX))$. 
 
 Thus to complete the proof, we need only bound from above the number of $\bx' \in \F_p^{n-1}$ with $B_{\red}^{M_F}(\bx') = 0$. Since in the present case $B_{\red}^{M_F}$  is not the zero polynomial in $\F_p$, the trivial bound (Lemma \ref{lemma_Schwartz_domain}) shows that 
 \[ \#\{\bx' \in \F_p^{n-1}: B_{\red}^{M_F}(\bx') = 0\} \leq (\deg B_\red^{M_F}) p^{n-2}.\] 
 By Lemma \ref{lemma_Noether}, $\deg B_\red^{M_F} \ll_{n,D} 1$, so that this is sufficient for the conclusion of the proposition.

Consider the case $n=1$. The proof sketched above continues to apply, with the   minor modification that $B_{\red}^{M_F}$ is an integer with $\log \|B_{\red}^{M_F}\| \ll_{n,D} \log \|F\|$. The exceptional set $\Ecal$ is the set of primes $p | B_{\red}^{M_F}$, so that again $|\mathcal{E}| \ll_{n,D} \log \|F\| / \log \log \|F\|$. If $p \not\in \Ecal$, then $M_F(Y,X_1)$ is irreducible over $\overline{\F}_p$, from which it follows that $\Omega_{p}$ is a regular extension of $L_p$, by arguing as above (or see Lemma \ref{lemma_integrally_closed_implies_absirred}).  

For either $n=1$ or $n \geq 2$, the general case over a number field follows exactly the same sequence of ideas.
 \end{proof}

\subsection{The role of Proposition \ref{prop_Cohen_corrected_appendix} in the proof of Theorem \ref{thm_Cohen} and Theorem \ref{thm_Cohen_special}}

 Once the \emph{conclusion} of Proposition \ref{prop_Cohen_corrected_appendix} is obtained for a given polynomial $F$, the rest of Cohen's proof of Theorem \ref{thm_Cohen} simply can proceed. In this section, we describe the rest of the argument. After the  discussion of this section, all that remains to prove 
 the special case Theorem \ref{thm_Cohen_special} is to show that for any polynomial $F$ such that $M_F$ is strongly $n$-genuine,  the additional hypothesis of Proposition \ref{prop_Cohen_corrected_appendix} holds. This will be accomplished in \S \ref{sec_strongly_gen} and \S \ref{sec_apply_strongly_genuine}. Then in \S \ref{sec_apply_strongly_genuine}, we will begin the process of reducing the proof of Theorem \ref{thm_Cohen} in full generality to a special case (related to Theorem \ref{thm_Cohen_special}) in which the additional hypothesis of Proposition \ref{prop_Cohen_corrected_appendix} holds.
 
In this section, rather than reciting the remainder of Cohen's proof in full detail, we sketch the main ideas, restricting attention initially to the case $K=\Q$ and those polynomials $F$ for which $L_F := \Omega_F\cap \overline{\Q}$ satisfies $L_F = \Q$. 
 In the proof below, given any $\bx \in \Q^n$, let $\Omega_{F,\bx}$ denote the splitting field of $F(Y,\bx)$ over $\Q$, so that $G(\bx)=\Gal(\Omega_{F,\bx},\Q)$. Recall also from the notation of Theorem \ref{thm_Cohen} that for any $\bx' \in \Q^{n-1}$ and any prime $p \in \Q = L_F$, $\Omega_{\bx',p}$ denotes the splitting field of $F(Y,X_1,\bx')$ over $L_p(X_1)=\F_p(X_1)$.

 To reach the conclusion of Theorem \ref{thm_Cohen} (or Theorem \ref{thm_Cohen_special}) we may assume that $F$ is squarefree, for repeated factors make no effect on its Galois group $G=\Gal(\Omega_F/\Q)$. Cohen's proof lies in an application of a large sieve inequality. (Specifically the large sieve inequality is given in \cite[Lemma 4.3]{Coh81} and the application in \cite[Lemma 5.2]{Coh81}.) Let $H$ denote a subgroup of $G$ and $\mathscr{C}(H)$ the union of its conjugates in $G$. For any $\bx \in \Q^n$ and prime $p$, let $[\textrm{Frob}_{\bx,p}]$ denote the conjugacy class of the Frobenius element of $p$ in $G(\bx)$. 
 Let $\Del_F(\bX)$ denote $\disc (F(Y,\bX))$, the discriminant of $F(Y,\bX)$ over $\Q(\bX)$; note that $\Del_F(\bX) \not\con 0$ since $F$ is squarefree. For each $Y \geq 1$, set
$$P_H(\bx,Y) := \#\{p\leq Y: p\nmid \Del_F(\bx), [\textrm{Frob}_{\bx,p}]\subset G\backslash \mathscr{C}(H)\}.$$
In particular if $H\subsetneq G$ and $G(\bx) = H$, then $P_H(\bx,Y)=0$.   Define for any $Y \geq 1$:
\beq\label{Cohen_PHY_dfn} 
P_H(Y) := \sum_{p\leq Y} p^{-n}\cdot \#\{\bx\bmod p: p\ndiv \Del_F(\bx), [\textrm{Frob}_{\bx,p}]\subset G\backslash \mathscr{C}(H)\}.
\eeq
For each $H \subsetneq G$, the large sieve inequality shows that for any $Y$,
\[\sum_{\|\bx \| \leq N}(P_H(\bx,Y)-P_H(Y))^2 \leq (N^n+Y^{2n}) P_H(Y),\]
for a universal constant $C$ (depending only on the base field). Since $P_H(\bx,Y) =0$ when $G(\bx)=H$, we can deduce from this that
$$\#\{\bx\in \Z^{n},\|\bx\|\leq N:G(\bx) \cong H\} \leq C(N^n+Y^{2n}) P_H(Y)^{-1}.$$
This suggests choosing $Y=N^{1/2}$, and reduces the problem to producing a lower bound for $P_H(Y)$ for each $H \subsetneq G$. In particular, if we can obtain $P_H(Y) \geq C_0 Y/ \log Y$ for each subgroup $H$ (with $C_0$ independent of $H$), then upon summing over $H \subsetneq G$ it follows that
$$\#\{\bx\in \Z^{n},\|\bx\|\leq N:G(\bx) \not\cong G\} \leq C (N^n+Y^{2n}) (C_0Y/\log Y)^{-1} \leq CC_0^{-1} N^{n-1/2}\log N,$$
a suitable bound for Theorem \ref{thm_Cohen}.

For a fixed $H \subsetneq G$, let us write
\[ P_H(Y) = \sum_{p \leq Y} p^{-n} (N_1(p)-N_2(p)),\]
in which 
\begin{align*}
    N_1(p)&=  \#\{\bx \modd{p}:   [\textrm{Frob}_{\bx,p}]\subset G\backslash \mathscr{C}(H)\}\\
      N_2(p)&=  \#\{\bx \modd{p}: p | \Del_F(\bx), [\textrm{Frob}_{\bx,p}]\subset G\backslash \mathscr{C}(H)  \} \leq \#\{\bx \modd{p}: p | \Del_F(\bx)  \}.
\end{align*}
Now $\Del_F(\bX)$ is a polynomial in $\bX$ with $\log \|\Del_F\| \ll_{n,\deg F} \|F\|$; denote the $\gcd$ of the coefficients of $\Del_F$ by $g$ so that $g \leq \|\Del_F\|$. Then $\Del_F(\bX) \con 0 \modd{p}$ precisely when $p | g$; let $\mathcal{E}_0=\{p: p|g\}$; then $|\Ecal_0| \ll \log g / \log \log g \ll_{n,D} \log \|F\|/ \log \log \|F\|$.
For each prime $p \not\in \mathcal{E}_0$, then $N_2(p)\ll_D p^{n-1}$, by the trivial bound (Lemma \ref{lemma_Schwartz_domain}),  so that in particular $N_2(p) \leq \tfrac{1}{4D!}p^n$ for all $p \not\in \mathcal{E}_0$ with $p \gg_{n,D} 1$.

Next, write $N_1(p)=N_1'(p) + N_1''(p)$ in which 
\begin{align*}
    N_1'(p)&=  \sum_{\substack{\bx' \in \F_p^{n-1}\\\Gal(\Omega_{\bx',p},L_p(X_1)) = \Gal(\Omega_F,L_F(\bX))}}\#\{x_1 \in \F_p:   [\textrm{Frob}_{\bx,p}]\subset G\backslash \mathscr{C}(H) \}\\
      N_1''(p)&= \sum_{\substack{\bx' \in \F_p^{n-1}\\\Gal(\Omega_{\bx',p},L_p(X_1)) \neq \Gal(\Omega_F,L_F(\bX))}}\#\{x_1 \in \F_p:   [\textrm{Frob}_{\bx,p}]\subset G\backslash \mathscr{C}(H) \}.
\end{align*}
By non-negativity, to bound $N_1(p)$ from below, it suffices to bound $N_1'(p)$ from below. Let $\mathcal{E}$ denote the exceptional set provided by Proposition \ref{prop_Cohen_corrected_appendix}.
We first observe that if the conclusion of Proposition \ref{prop_Cohen_corrected_appendix} holds (replacing the original application of \cite[Lemma 4.2]{Coh81}),  then for every prime $p \not\in \Ecal$, the sum over $\bx'$ in $N_1''(p)$ is over at most $\ll_{n,D}p^{n-2}$ terms, so that the sum over $\bx'$ in $N_1'(p)$ must be over $\gg_{n,D}p^{n-1}$ terms. 
To treat $N_1'(p)$, fix any $\bx'=(x_2,\dots,x_n) \in \F_p^{n-1}$ such  that $\Gal(\Omega_{\bx',p},L_p(X_1)) = \Gal(\Omega_F,L_F(\bX)).$ For such a tuple $\bx'$, a function field version of the Chebotarev density theorem \cite[Lemma 4.4]{Coh81} shows that 
\beq\label{Cohen_Chebotarev_conclusion} 
\#\{x_1 \in \F_p:   [\textrm{Frob}_{\bx,p}]\subset G\backslash \mathscr{C}(H) \} = \frac{|G\setminus \mathscr{C}(H)|}{|G|}p + O_{n,D}(p^{1/2}).\eeq
For each proper subgroup $H\subsetneq G$, $|G \setminus \mathscr{C}(H)| \geq 1$, so summing over the relevant tuples $\bx'$ yields
\[N_1(p) \geq N_1'(p) \gg \frac{|G\setminus\mathscr{C}(H)|}{|G|}p^{n} + O_{n,D}(p^{n-1/2}) \geq  \tfrac{1}{2}\cdot \tfrac{1}{D!} p^n,\] 
say, as long as $p \gg_{n,D}1$ and $p \not\in  \Ecal$.  
Hence in total, for each $p \not\in (\Ecal \cup \Ecal_0)$ with $p \gg_{n,D} 1$, $N_1(p) - N_2(p) \geq \tfrac{1}{4D!}p^n$. We may conclude that
\[ P_H(Y) \geq \sum_{\substack{p \leq Y\\p \not\in (\Ecal \cup \Ecal_0), p \gg_{n,D} 1}}p^{-n} \cdot \tfrac{1}{4D!}p^n  = \tfrac{1}{4D!}\pi(Y)+ O_{n,D}(1+|\Ecal \cup \Ecal_0|) \geq C_0 Y/\log Y\]
for some constant $C_0=C_0(n,D)$,
as long as $Y/\log Y \gg |\Ecal \cup \Ecal_0|$. In particular, by the upper bound provided for $|\mathcal{E}|$ in Proposition \ref{prop_Cohen_corrected_appendix}, it suffices to take $Y \gg_{n,D} \log \|F\|$. This leads to the conclusion of Theorem \ref{thm_Cohen} for all $N \gg_{n,D}(\log \|F\|)^2$. For all $N \ll_{n,D}(\log \|F\|)^2$, we apply the trivial bound 
\[\# \{\bx \in \Z^n: \|\bx \| \leq N: G(\bx) \not\cong G\} \ll_n N^n \ll_{n,D}(\log \|F\|)^{2n} \ll_{n,D}(\log \|F\|)^n N^{n-1/2}(\log N),\]
which also suffices for the theorem.

For $K=\Q$ and polynomials $F$ such that $L_F=\Q$, this in fact yields better dependence on $\|F\|$ than Theorem \ref{thm_Cohen} states; thus for clarity we briefly indicate how polynomial dependence on $\|F\|$ could arise in the more general case of a number field $K$ and $L_F=\Omega_F \cap \overline{K}$. In this case, let $K'$ denote the  Galois closure of $K$ over $\Q$, and consider the compositum $K'L_F$ over $K$, which Cohen shows has   absolute discriminant $\Del_F \ll \|F\|^{c}$ for some $c=c(K,n,D)$. 
The construction given above is then generalized to this setting, and in particular the definition of $P_H(Y)$ in (\ref{Cohen_PHY_dfn}) is generalized to  a certain sum over prime ideals $\pfrak$ in $\Ocal_K$ that split completely in $K'$. For each such $\pfrak$, the outcome of the Chebotarev density theorem in (\ref{Cohen_Chebotarev_conclusion}) is an asymptotic with $(|G \setminus \mathscr{C}(H)|/|G|)\cdot p$ now replaced by
$(|S_\pfrak \setminus \mathscr{C}(H)|/|S_\pfrak|)\cdot |K_\pfrak|$,
in which $S_\pfrak$ is a certain subset of $\Gal(F(Y,\bX),K_\pfrak(\bX))$. (Precisely, $S_\pfrak$ is the subset of $\Gal(F(Y,\bX),K_\pfrak(\bX))$ whose fixed field of constants is $K_\pfrak$; $S_\pfrak$ is defined up to conjugation by an element in $G$.) In particular, for each $\pfrak$ such that $S_\pfrak \not\subset \mathscr{C}(H)$, the main term in the asymptotic is nonvanishing, so that it is at least $\geq \tfrac{1}{D!}|K_\pfrak|$. In this setting, we must now ensure that for sufficiently many $\pfrak$ considered in the sum defining $P_H(Y)$, $S_\pfrak \not\subset \mathscr{C}(H)$. 
Cohen achieves this by counting   primes with an Artin symbol condition; namely   for $Y$ sufficiently large he shows that:
 \[\#\{\pfrak \in \Ocal_K, \mathrm{Nm}(\pfrak)=|K_\pfrak| \leq Y:  \text{$\pfrak$ splits completely in $K'$,} \left[\frac{L_F,K}{\pfrak}\right] \not\subset \mathscr{C}(H)\} \geq C_1 Y/\log Y.
 \]
In particular, this is verified (by another application of a Chebotarev density theorem) for some $C_1 \gg \Del_F^{-1}$ as long as $Y > (\Del_F)^{c'}$ for some $c'=c'(n,D,K)$ \cite[Lemma 6.2]{Coh81}.  This dependence on $\Del_F$, which is only controlled polynomially by $\|F\|$, limits the large sieve argument to the case in which $N \gg \|F\|^{c''}$ for some $c''$. Then, the remaining cases with $N \ll \|F\|^{c''}$ are handled by a trivial bound, which enlarges the implicit constant in the upper bound in Theorem \ref{thm_Cohen}. This is the source of the possible polynomial dependence on $\|F\|$ in the general case of Theorem \ref{thm_Cohen} over a number field.

\section{The theory of strongly $n$-genuine polynomials }\label{sec_strongly_gen}

In this section, we prove the   key properties of strongly $n$-genuine polynomials over a  number field $k$. (See also Remark \ref{remark_fields} about arbitrary fields.) 
In preparation, suppose $f\in \calO_k[Y,Z]$ has total degree at most $D$ and  expand
    \[
    f(Y,Z)=\sum_{\substack{\ell,m\\\ell+m\leq D}}a_{\ell,m}Y^{\ell}Z^{m}.
    \]
    Then Noether's Lemma \ref{lemma_Noether} produces a   form $B_\red=B_{\red}( (a_{\ell,m})_{\ell,m})$, with coefficients in $\Z$, such that
    \beq\label{define_utility_of_B}
    \text{$B_{\red}( (a_{\ell,m})_{\ell,m})=0 \Leftrightarrow f(Y,Z)$ is reducible over $\overline{\Q}=\overline{k}$ or $\deg f < D$. }
    \eeq
 For clarity, note that if we apply this criterion to a polynomial $F(Y,X_{i_0},\bX_{I'})$ as a polynomial in $Y,X_{i_0}$ (with $I' = I \setminus \{i_0\}$), 
then for a given specialization $\bx_{I'} \in k^{|I'|}$,  $B_\red((a_{\ell,m}(\bx_{I'}))_{\ell,m})=0$   if and only if $F(Y,X_{i_0},\bx_{I'})$ is reducible over $\overline{\Q}$ or the total degree of $F(Y,X_{i_0},\bx_{I'})$ is strictly less than $\deg_{Y,X_{i_0}}F(Y,X_{i_0},\bX_{I'})$, the total degree of $F(Y,X_{i_0},\bX_{I'})$ \emph{as a polynomial in $Y,X_{i_0}$.}

There are four equivalent properties that characterize strongly $n$-genuine polynomials. For notational clarity, we record a version specifically for $n=1$, followed by a version for all $n \geq 2$.
\begin{thm}[Strongly $1$-genuine]\label{thm_strongly_gen_ppties_n1}
Let $k/\Q$ be a finite extension, and let  $F \in \calO_k[Y,X_1]$  be irreducible over $k(X_1)$ and of total degree $D$.  
Define $\mathcal{L}_F = k(X_1)[Y]/(F(Y,X_1))$. 
The following are equivalent: 

(I) $F$ is a strongly $1$-genuine polynomial, that is to say, $k(X_1)[Y]/(F(Y,X_1))$ is a strongly $1$-genuine extension of $k(X_1)$. 

(II) $ \overline{k} \cap \mathcal{L}_F = k$, 
that is to say, $k$ is integrally closed in $\Lcal_F$ (or equivalently $\Lcal_F$ is a regular extension of $k$).

(III)   $F(Y,X_1)$ is irreducible over $\overline{k} = \overline{\Q}$. 

(IV)  Expanding the polynomial  as
   $
    F(Y,X_1)=\sum_{\ell+m\leq D}a_{\ell,m} Y^{\ell}X_{1}^{m},
$
   the form $B_{\red}((a_{\ell,m})_{\ell,m})$  defined as in (\ref{define_utility_of_B}), when evaluated at the coefficients $a_{\ell,m}$, is a nonzero integer.  
\end{thm}

\begin{thm}[Strongly $n$-genuine]\label{thm_strongly_gen_ppties}
Let $k/\Q$ be a finite extension. Let $I$ be an index set of cardinality $n\geq 2$. Let  $F \in \calO_k[Y,\bX_I]$ be irreducible over $k(\bX_I)$ and   of total degree $D$. Define $\mathcal{L}_F = k(\bX_I)[Y]/(F(Y,\bX_I))$, and consider the following conditions.

(I) $F$ is a strongly $n$-genuine polynomial, that is to say, $k(\bX_I)[Y]/(F(Y,\bX_I))$ is a strongly $n$-genuine extension of $k(\bX_I)$.

 (II) For any $i_0 \in I$, upon defining $I':=I \setminus \{i_0\}$,  $k(\bX_{I'})$ is integrally closed in $\Lcal_F$: 
\[ \overline{k(\bX_{I'})} \cap \mathcal{L}_F = k(\bX_{I'}).\] 

(III) For any $i_0 \in I$, upon defining $I':=I \setminus \{i_0\}$,  there exists a choice of $\bx_{I'} \in k^{n-1}$ such that $F(Y,X_{i_0},\bx_{I'})$ is irreducible over $\overline{k} = \overline{\Q}$ and has  total degree  equal to $\deg_{Y,X_{i_0}}F(Y,\bX_I)$.

 (IV) For any $i_0 \in I$, upon defining $I':=I \setminus \{i_0\}$,   expand the polynomial $F(Y,\bX_I)$ in $Y$ and $X_{i_0}$ as
    \beq\label{app_F_expansion_for_B}
    F(Y,\bX_I)=\sum_{\substack{\ell,m\\\ell+m\leq D}}a_{\ell,m}(\bX_{I'})Y^{\ell}X_{i_{0}}^{m}.
    \eeq
   The form $B_{\red}(a_{\ell,m}(\bX_{I'}))$ defined as in (\ref{define_utility_of_B}) is not identically zero as a polynomial in $\bX_{I'}$. 

The following are equivalent:
\[ \text{(I) $\Leftrightarrow$ (II) $\Leftrightarrow$(III) $\Leftrightarrow$(IV).}\]
 Furthermore, for any fixed index $i_0 \in I$, the following are equivalent:
\[   \text{(II) for $i_0$ $\Leftrightarrow$  (III) for $i_0$ $\Leftrightarrow$  (IV) for $i_0$}.\]

\end{thm}

We note that for a given index $i_0$, condition (II) for $i_0$ is the statement that $\Lcal_F$ is a regular extension of $k(\bX_{I'})$. 
\subsection{Absolute irreducibility and regularity}
In preparation, we encapsulate a useful property in the following lemma.  Its proof applies equally well over any field, and since we refer to it within Proposition \ref{prop_Cohen_corrected_appendix} in the setting of a finite field, we provide a general statement over an arbitrary field $\Kcal$.
 
 \begin{lem}\label{lemma_integrally_closed_implies_absirred} 
 Let $\Kcal$ be a field.
    Let $F\in \Kcal[Y,X_1,\ldots,X_n]$  be monic in $Y$ and irreducible over $\Kcal(X_1,\ldots,X_n)$, and denote $\mathcal{L}_F=  \Kcal(X_1,\ldots,X_n)[Y]/F$. Then $\Lcal_F$ is a regular extension of $\Kcal$, namely:
       \[
       \overline{\Kcal}\cap \mathcal{L}_F=\Kcal,
       \]
     if and only if $F(Y,\bX)$ is absolutely irreducible, that is, irreducible over $\overline{\Kcal}$.      
   \end{lem}
 We will apply this directly   in the proof of Theorem \ref{thm_strongly_gen_ppties_n1}. We also note that in the context of Theorem \ref{thm_strongly_gen_ppties} over a number field $k$, if for some $i_0 \in  I$ the condition 
     $\overline{k(\bX_{I'})}\cap \mathcal{L}_F=k(\bX_{I'})$  holds,   then $\overline{k}\cap \mathcal{L}_F=k$, so that $\Lcal_F$ is a regular extension of $k$, and then this lemma shows $F$ must be absolutely irreducible.  

 \begin{proof} 
Let $F[Y,\bfX]$ be a polynomial such that $\overline{\Kcal}\cap \mathcal{L}_F=\Kcal$ and suppose for a contradiction that it is not absolutely irreducible (so it is reducible over $\overline{\Kcal}$). Denote by $\Omega_{F}$ the Galois closure of $\mathcal{L}_{F}$ over $\Kcal(\bX)$ and set $L:=\overline{\Kcal}\cap \Omega_{F}$.  Notice that since $\Omega_F$ is the Galois closure, $L/\Kcal$ must be Galois and hence $L(\bfX)/\Kcal(\bfX)$ is a Galois extension.  Since by assumption $F$ is not absolutely irreducible, we can write
\[
F=\prod_{i=1}^{\ell}G_{i}(Y,\bfX)\qquad\text{with each $G_i(Y,\bX) \in L[Y,\bX]$,}
\]
for some $\ell \geq 2$,
with each $G_i$ irreducible over $L(\bX)$ and monic in $Y$.
Now take $W$ to be one of the roots of $F$ so that $\mathcal{L}_{F}\cong \Kcal(\bfX)[W]$; then there exists $i$ such that $G_{i}(W)=0$. Without loss of generality, we may assume $i=1$, i.e. the polynomial
\[
G_{1}(Y,\bX) =Y^{m}+a_{m-1}Y^{m-1}+\cdots +a_{1}Y+a_{0},
\]
with $a_i \in L[X_1,\ldots,X_n]$,
is the minimal polynomial of $W$ over $L(\bfX)$  (and additionally observe that $m<\deg_Y F$). By construction, we have that  $G_{1}\not\in \Kcal(\bX)[Y]$, since $F$ is the minimal polynomial of $W$ over $\Kcal(\bfX)$. 

We claim that for every $\sigma\in\Gal(\Omega_{F}/\mathcal{L}_F)$,  
we have $\sigma(a_{i})=a_{i}$ for every $i=1,...,m-1$. We can see that 
\[
\begin{split}
0 &=\sigma(G_{1}(W,\bX)) =\sigma(W^{m}+a_{m-1}W^{m-1}+\cdots +a_{1}W+a_{0})\\&=\sigma(W^{m})+\sigma(a_{m-1})\sigma(W^{m-1})+\cdots +\sigma(a_{1})\sigma(W)+\sigma (a_{0})\\&=W^{m}+\sigma(a_{m-1})W^{m-1}+\cdots +\sigma(a_{1})W+\sigma (a_{0}),
\end{split}
\]
 where in the third step we have used the fact that $\sigma(W)=W$ since $\sigma\in\Gal(\Omega_{F}/\mathcal{L}_F)$. That is to say, we conclude that $W$ is a root of the polynomial 
 \[\widetilde{G}_{1}:=Y^{m}+\sigma(a_{m-1})Y^{m-1}+\cdots \sigma(a_{1})Y+\sigma (a_{0}).\]
 On the other hand, since $G_{1}$ is the minimal polynomial of $W$ over $L(\bfX)$, $\deg(\widetilde{G}_{1})=\deg(G_{1})=m$, and $G_{1},\widetilde{G}_{1}$ are both monic, we must have $G_{1}=\widetilde{G}_{1}$. Hence, we can conclude $\sigma(a_{i})=a_{i}$ for every $\sigma\in\Gal(\Omega_{F}/\mathcal{L}_F)$ and every $i\in\{1,...,m-1\}$. 
 
 Since $G_{1}\in L[Y,\bX]$ but not in $\Kcal(\bfX)[Y]$, we can find $i_{0}\in\{1,...,m-1\}$ such that $a_{i_{0}}\in L[\bfX]$ but not in $\Kcal(\bfX)$.
 Now write
 \[
 a_{i_{0}}(\bfX)=\sum_{\mathbf{i}}b_{\mathbf{i}}\bfX^{\mathbf{i}},
 \]
 where $b_{\mathbf{i}}\in L$; furthermore there must exist a multi-index $\mathbf{j}$ such that $b_{\mathbf{j}}\in L\setminus \Kcal$. On the other hand, from the fact that $\sigma(a_{i_{0}})=a_{i_{0}}$ for every $\sigma\in\Gal(\Omega_{F}/\mathcal{L}_F)$, it follows that $\sigma(b_{\mathbf{i}})=b_{\mathbf{i}}$ for every $\sigma\in\Gal(\Omega_{F}/\mathcal{L}_F)$ and for every $\mathbf{i}$; that is $b_{\mathbf{i}}\in \mathcal{L}_F$ for every $\mathbf{i}$. In particular, this implies that $b_{\mathbf{j}}\in L\cap \mathcal{L}_F \subset \overline{\Kcal}\cap \mathcal{L}_F$. Under the hypothesis of the lemma, this says that $b_{\mathbf{j}} \in \Kcal$, which leads us to our contradiction, as we already established that $b_{\mathbf{j}} \in L \setminus \Kcal$.  This establishes that if $\mathcal{L}_F$ is a regular extension of $\Kcal$ then $F(Y,\bX)$ is absolutely irreducible.

 To prove the other direction, we will prove the contrapositive: supposing that $\Kcal\subsetneq \overline{\Kcal}\cap \Lcal_F$, we will show that $F(Y,\bX)$ is not absolutely irreducible. Let $\Kcal\subsetneq \Kcal'\subset \overline{\Kcal} \cap \Lcal_F$ be a finite extension of $\Kcal$; it follows that $\Kcal(\bX)\subsetneq \Kcal'(\bX) \subset \Lcal_F$ and $[\Lcal_F:\Kcal'(\bX)] < [\Lcal_F:\Kcal(\bX)]=\deg_Y F.$ Hence, $F(Y,\bX)$ cannot be irreducible over $\Kcal'(\bX)$ (since otherwise we would have that $[\Lcal_F:\Kcal'(\bX)]=\deg_Y F$). Hence, we can find $G,H\in \Kcal'(\bX)[Y]$ such that $F(Y,\bX) = G(Y,\bX)H(Y,\bX)$; this gives us a factorization of $F(Y,\bX)$ over $\Kcal'$  and hence $F(Y,\bX)$ is not irreducible over $\overline{\Kcal}$.
 
\end{proof}

\subsection{Proof of Theorem \ref{thm_strongly_gen_ppties_n1} for $n=1$}
(I) is the statement that if  an extension $M'$ satisfies  $k(X_1) \subseteq  M' \subseteq k(X_1)[Y]/(F(Y,X_1)) = \Lcal_F$ and $M'=k(X_1)[Y]/R(Y)$ for some polynomial $R$ independent of $X_1$, then $M'=k(X_1)$. Any extension of the form $M'=k(X_1)[Y]/R(Y)$ is equivalent to $L'(X_1)$ for some extension  $L' \supset k$. Thus (I) is equivalent to the statement that if $L'$ is some extension with $k \subset L' \subset \overline{k}\cap \Lcal_F$, then $L'=k$; this confirms (I) $\Leftrightarrow$ (II).
  By Lemma \ref{lemma_integrally_closed_implies_absirred},   (II) $\Leftrightarrow$ (III).  Finally,   (III) $\Leftrightarrow$ (IV) by Noether's Lemma \ref{lemma_Noether} (ii).

\subsection{Proof of Theorem \ref{thm_strongly_gen_ppties} for $n=2$}
We first   prove all the relations that have brief, elementary arguments. Then we turn to the final more intricate relation (II) $\Leftrightarrow$ (IV), which we extract as a lemma. 

((I) $\Rightarrow$ (II))
By definition, $F$ is a strongly $n$-genuine polynomial if and only if $\mathcal{L}_F$ is a strongly $n$-genuine extension of $k(\bX_I)$. For each ${i_0}\in I$, denote $N_{{i_0}}:= (\overline{k(\bX_{I'})}\cap \mathcal{L}_F)(X_{{i_0}})$, and observe $k(\bX_I) \subseteq N_{i_0} \subseteq \mathcal{L}_F$. By construction, $N_{i_0}$ is an extension of $k(\bX_I)$ that is not  $n$-genuine.  Now under the hypothesis that $\mathcal{L}_F$ is strongly $n$-genuine, if for some  ${i_0}$ it were true that  $k(\bX_I)\subsetneq N_{{i_0}}\subseteq \mathcal{L}_F$, this would be a contradiction to the definition of a strongly $n$-genuine extension. Thus for each ${i_0} \in I$, $k(\bX_I) = N_{i_0}$.  This implies that for each ${i_0} \in I$, $ \overline{k(\bX_{I'})}\cap \mathcal{L}_F=k(\bX_{I'}).$ 

((II) $\Rightarrow$ (I))
 We will prove the contrapositive. Suppose $\mathcal{L}_F$ is not strongly $n$-genuine: then  there exists some intermediate extension $M'$ with $k(\bX_I)\subsetneq M'\subset \mathcal{L}_F$, which is not $n$-genuine. In particular, we can write $M'=k(\bX_I)[Y]/R(Y,\bX_J)$ for some polynomial $R$ depending only on $Y$ and $\bX_J$ for some subset $J \subsetneq I$.  For each ${i_0}\in I$, denote $N_{{i_0}}:= (\overline{k(\bX_{I'})}\cap \mathcal{L_F})(X_{{i_0}})$, as in the previous step. Observe that $M'\subset N_{{i_0}}$ for some ${i_0}$, and this implies $N_{{i_0}}\supsetneq k(\bX_I)$. In turn, this implies $ \overline{k(\bX_{I'})}\cap \mathcal{L}_F \supsetneq k(\bX_{I'}),$  in which $I' = I \setminus \{i_0\}$.

 The next elementary arguments assume we have fixed an index $i_0$.

 ((III) for $i_0$ $\Leftrightarrow$ (IV) for $i_0$) The condition (III) for a given index $i_0$ is equivalent to the condition (IV)  for the index $i_0$ by (\ref{define_utility_of_B}). Indeed, for each   $\bx_{I'} \in k^{|I'|}$,  the property
\[ \text{$F(Y,X_{i_0},\bx_{I'})$  is irreducible over $\overline{k}$ and has $\deg F(Y,X_{i_0},\bx_{I'})=\deg_{Y,X_{i_0}}F(Y,X_{i_0},\bX_{I'})$}\]
occurs
if and only if  $B_{\red}(a_{\ell,m}(\bx_{I'})) \neq 0.$

((III) for $i_0$ $\Rightarrow$ (II) for $i_0$) We will proceed by showing the contrapositive. Assume that $k(\bX_{I'})\subsetneq (\overline{k(\bX_{I'})} \cap \Lcal_F) =: M_{i_0}$, so that in particular $[M_{i_0}:k(\bX_{I'})] \geq 2$, which implies $[M_{i_0}(X_{i_0}):k(\bX_{I})] \geq 2$. We want to establish that for all $\bx_{I'}\in k^{n-1}$, we have that $F(Y,X_{i_0},\bx_{I'})$ is reducible over $\overline{k}.$ First, we observe that $F(Y,X_{i_0},\bX_{I'})$ is reducible over $M_{i_0}(X_{i_0})$. This follows from the fact that $F(Y,X_{i_0},X_{I'})$ is irreducible over $k(\bX_{I})$ and so $\deg_Y F=[\Lcal_F: k(\bX_I)]$, while on the other hand $[\Lcal_F: k(\bX_I)]=[\Lcal_F : M_{i_0}(X_{i_0})][M_{i_0}(X_{i_0}): k(\bX_I)]$, in which the last factor is $\geq 2$.  From this we conclude that $[\Lcal_F : M_{i_0}(X_{i_0})]< \deg_Y F$, and hence $F(Y,X_{i_0},\bX_{I'})$ is reducible over $M_{i_0}(X_{i_0})$ as claimed.

Finally, for any choice of $\bx_{I'}\in k^{n-1}$, let $M_{i_0,\bx_{I'}}$ denote the specialization of the field $M_{i_0}$ at the choice $\bx_{I'}\in k^{n-1}$. Consequently, we have that for any choice of $\bx_{I'}\in k^{n-1}$, $F(Y,X_{i_0},\bx_{I'})$ is reducible over $M_{i_0,\bx_{I'}}$. Since $M_{i_0,\bx_{I'}} \subset \overline{k}$, this completes the claim.

To complete the proof of the theorem, it suffices to show that (II) for $i_0$ $\Rightarrow$  (IV) for $i_0$. This we will derive in Lemma \ref{lemma_Cohen_strongly_genuine_reducible}, via the following lemma.

 \begin{lem}\label{lemma_F_strongly_n_geniune_open_V}
 Let $k/\Q$ be a finite extension. Let $I$ be an index set of cardinality $n \geq 2$.
  Let $F(Y,\bX_I) \in \calO_k[Y,\bX_I]$ be irreducible over $k(\bX_I)$ and define $\mathcal{L}_F=k(\bX_I)[Y]/F$. Suppose that for some $i_0 \in I$ and $I':=I \setminus \{i_0\}$,
  \[
       \overline{k(\bX_{I'})}\cap \mathcal{L}_F=k(\bX_{I'}),
       \]
       that is to say, $k(\bX_{I'})$ is integrally closed in $\mathcal{L}_F$.
       Then, there exists a nonempty open set $V\subset k^{n-1}$ such that for all $\bx_{I'}\in V$, 
       \[ 
       \overline{k} \cap (k(X_{i_0})[Y]/F(Y,X_{i_0},\bfx_{I'})) = k,
       \]
       that is to say, 
       $k$ is integrally closed in $k(X_{i_0})[Y]/F(Y,X_{i_0},\bfx_{I'})$.  
    \end{lem}
    \begin{proof}
    Let $i_0$ be fixed as in the hypothesis.
    Consider $W=\text{Spec}\left(k
    [Y,\bX_I]/F\right)$ and the morphism $f$ induced by the inclusion $k[\bX_{I'}]\subset k[Y,\bX_I]/F$, namely:
    \[
    f:W\longrightarrow\mathbb{A}_{k}^{n-1}.
    \]
    Let $V'\subset\mathbb{A}_{k}^{n-1}$ be the set of the $u\in\mathbb{A}_{k}^{n-1}$ such that the fiber
    \[
    f_{u}:W_{u}\longrightarrow\mathbb{A}_{k(u)}^{n-1}
    \]
    is geometrically integral (i.e. $k(u)$ is integrally closed in $K(W_{u})$). (Here we use the standard notation   that $k(u)$ is the residue field at $u$ and $K(W_{u})$ is the field of rational functions on $W_{u}$, following e.g. \cite[Chapter II, section $3$, page 89]{Har77}.) By  \cite[EGA IV part 3, Theorem $9.7.7$]{EGAIV_part3},   the set $V'$ is locally constructible. On the other hand, the fiber on the generic point $\eta$, say $f_{\eta}$, is geometrically integral; this follows from the hypothesis that $k(\bX_{I'})$ is integrally closed in $k(\bX_I)[Y]/F=K(W)$.   Consequently, the generic point $\eta$ lies in  $V'$. Hence, since it contains the generic point, $V'$ contains an open subset $U'$: indeed, since $V'$ is locally constructible we can find an open covering $\mathbb{A}_{k}^{n-1}=\bigcup_{i}V_{i}$ such that for each $i$, $V_{i}\cap V'$ is a constructible set, i.e. it is a finite union of sets of the type $S\cap T^{c}$ for $T,S$ open sets. Since $\eta\in V'$, we can find $i$ such that $\eta\in V_{i}\cap V'$, but then we can find $S,T$ open sets such that $\eta\in S\cap T^{c}$. On the other hand, since $\{\eta\}$ is dense in $\mathbb{A}_{k}^{n-1}$, it follows that $T=\emptyset$. Then it suffices to set $U' = S \subset (V_i \cap V') \subset V'$.  Finally, we take $V$ to be the open set $V=U'\subseteq k^{n-1}$; for every $u\in V$,  which we will denote in terms of coordinates by $u=\bx_{I'}$, we have that $k(u)=k$ is integrally closed in $K(W_{u})=\text{Frac}\left(k[Y,X_{i_0},\bx_{I'}]/F(Y,X_{i_0},\bx_{I'})\right)$. 
\end{proof}

\begin{lem}[(II) for $i_0$ $\Rightarrow$ (IV) for $i_0$]\label{lemma_Cohen_strongly_genuine_reducible}
  Let $k/\Q$ be a finite extension. Let $I$ be an index set of cardinality $n \geq 2$.
  Let $F(Y,\bX_I) \in \calO_k[Y,\bX_I]$ be irreducible over $k(\bX_I)$ and define $\mathcal{L}_F=k(\bX_I)[Y]/F$. Suppose that for some $i_0 \in I$ and $I':=I \setminus \{i_0\}$,
  \[
       \overline{k(\bX_{I'})}\cap \mathcal{L}_F=k(\bX_{I'}),
       \]
       that is to say, $k(\bX_{I'})$ is integrally closed in $\mathcal{L}_F$.
   For this index $i_0$, define the polynomial 
         $  B_{\red}(a_{\ell,m}(\bX_{I'})) \in \Z[\bX_{I'}]$ as in (\ref{define_utility_of_B}), to detect whether $F(Y,X_{i_0},\bx_{I'})$ is reducible over $\overline{\Q}$. Then $B_{\red}(a_{\ell,m}(\bX_{I'})) $ is not identically zero as a polynomial in $\bX_{I'}$.

\end{lem}

\begin{proof} 
     By Hilbert's irreducibility theorem (Lemma \ref{lemma_HIT}), since $F(Y,\bX_I)$ is irreducible over $k$, there exists a dense set $U \subset k^{n-1}$    such that 
    \beq\label{conclusion_F_irred_on_U}
    F(Y,X_{i_0},\bx_{I'})\text{ is irreducible over }k, \forall \bx_{I'}\in U.
    \eeq
On the other hand, expand $F$ as a polynomial in $Y,X_{i_0}$ as in (\ref{app_F_expansion_for_B}), and consider the polynomial  $B_{\red}(a_{\ell,m}(\bX_{I'})) $ as in (\ref{define_utility_of_B}).
    We will use Lemma \ref{lemma_F_strongly_n_geniune_open_V} to show that in particular there exists some $\bx_{I'}\in U$ such that $B_{\red}((a_{\ell,m}(\bx_{I'}))) \neq 0$.  
  Assume for contradiction that $B_\red = 0$ on $U$.  Then $F(Y,X_{i_0},\bx_{I'})$ is reducible over $\overline{\Q}$ (or $\deg F(Y,X_{i_0},\bx_{I'})<\deg_{Y,X_{i_0}} F$) for all $\bx_{I'}\in U$.  We first argue that there exists a dense subset $U'\subset U$ such that $F(Y,X_{i_0},\bx_{I'})$ is reducible over $\overline{\Q}$ for all $\bx_{I'} \in U'$. Indeed, if we let $W$ denote the subset of $U$ where $\deg F(Y,X_{i_0},\bx_{I'})< \deg_{Y,X_{i_0}} F$ for $\bx_{I'} \in W$, we claim $W$ is nowhere dense, so that $U' = U \setminus W$ then is the desired set. Now $W$ is contained in a finite union of sets, each of which is defined as the vanishing set of a polynomial (a nonzero polynomial in $\bx_{I'}$, that defines the coefficient of a certain monomial in $Y,X_{i_0}$ in $F(Y,X_{i_0},\bx_{I'})$), and so $W$ is a proper closed subset of lower dimension, and is nowhere dense.
  
  Now recall that by the construction in (\ref{conclusion_F_irred_on_U}), for each $\bx_{I'} \in U' \subset U$, $F(Y,X_{i_0},\bx_{I'})$ is irreducible over $k$ yet reducible over $\overline{k} = \overline{\Q}$; hence $F(Y,X_{i_0},\bx_{I'})$ must be reducible over some extension $k'$ with $k\subsetneq k'\subset \overline{\Q}$ (with $k'$ depending on $\bx_{I'}$). So, for all $\bx_{I'}  \in U'$, $k$ cannot be integrally closed in $k(X_{i_0})[Y]/F(Y,X_{i_0},\bfx_{I'})$   since its integral closure will contain $k'$.  (Alternatively, apply Lemma \ref{lemma_integrally_closed_implies_absirred} to see that $F(Y,X_{i_0},\bx')$ reducible over $\overline{k}$ implies $k \subsetneq \overline{k} \cap k(X_{i_0})[Y]/F(Y,X_{i_0},\bfx_{I'})$.)
  On the other hand, by Lemma \ref{lemma_F_strongly_n_geniune_open_V}, there exists a nonempty open set $V \subset k^{n-1}$ such that for all $\bx_{I'} \in V$, $k$ is integrally closed in $k(X_{i_0})[Y]/F(Y,X_{i_0},\bfx_{I'})$. Since $U'$ is dense and $V$ is open, $U'\cap V \neq \emptyset.$ Thus, we get a contradiction: if an element $\bx_{I'}$  exists in $ U'\cap V$, $k$ would be both integrally closed and not integrally closed in $k(X_{i_0})[Y]/F(Y,X_{i_0},\bfx_{I'})$. 
  
  Thus the supposition that $B_\red = 0$ on $U$ must be false.
  In particular, there exists a choice of $\bx_{I'}\in U$ such that $B_{\red}((a_{\ell,m}(\bx_{I'}))) \neq 0$, and finally $B_{\red}((a_{\ell,m}(\bX_{I'})))$ is a polynomial in $\Z[\bX_{I'}]$ that is not identically zero.

 \end{proof}

 This completes the proof of Theorem \ref{thm_strongly_gen_ppties} for $n \geq 2$.

\begin{rem}\label{remark_dfn_agree}
 Recall from Definition \ref{dfn_strongly_field} that $F$ is a strongly $n$-genuine polynomial if $F(Y,\bX_I)$ is irreducible over $k(\bX_I)$ and monic in $Y$, and $k[Y](\bX_I)/F(Y,\bX_I)$ is a strongly $n$-genuine extension.
 In \cite{BPW25x} we defined a strongly $n$-genuine polynomial $F$ to be an  absolutely irreducible polynomial (irreducible over $\overline{k}$) that is monic in $Y$ and such that $k[Y](\bX_I)/F(Y,\bX_I)$ is a strongly $n$-genuine extension. Notice that  if Definition \ref{dfn_strongly_field} holds, then property (II) in Theorem \ref{thm_strongly_gen_ppties} holds for any index $i_0\in I$. This implies that $\overline{k} \cap \Lcal_F = k$, so that by Lemma \ref{lemma_integrally_closed_implies_absirred},  $F(Y,\bX_I)$ is irreducible over $\overline{k}$, so that any polynomial that is strongly $n$-genuine in the sense of Definition \ref{dfn_strongly_field} (and monic in $Y$) is strongly $n$-genuine in the sense of \cite{BPW25x}.

 \end{rem}

\subsection{Natural consequence for strongly $n$-genuine polynomials}
Here we prove the quantitative property that follows from being strongly $n$-genuine.

\begin{thm}\label{thm_strongly_gen_reducible_consequence_k}
Let $n \geq 2$. Let $k/\Q$ be a finite extension, with ring of integers $\mathcal{O}_k$ and $m=[k:\Q]$. Let $F(Y,\bX) \in \Ocal_k[Y,X_1,\ldots,X_n]$ be a strongly $n$-genuine polynomial of total degree $D$. Then for all $B \gg 1$, 
\[ \# \{\bx' \in \Ocal_k^{n-1}, \|\bx'\|\leq B : \text{$F(Y,X_1,\bx')$ is reducible over $\overline{\Q}$}\}\ll_{m,n,\deg F} B^{n-2} (\log B)^{(n-2)(m-1)}.\]
Also, there exists a finite set $\mathcal{E}$ of exceptional prime ideals $\pfrak \in \Ocal_k$, with $|\Ecal| \ll_{m,n,D} \log \|F\|/ \log \log \|F\|$, such that for all $\pfrak \not\in \Ecal$,  
\[ \# \{\bx' \in k_\pfrak^{n-1}: \text{$F(Y,X_1,\bx')$ is reducible over $\overline{k_\pfrak}$}\}\ll_{n,\deg F} |k_\pfrak|^{n-2}.\]
\end{thm}
\begin{proof}[Proof of Theorem \ref{thm_strongly_gen_reducible_consequence}]
Theorem \ref{thm_strongly_gen_reducible_consequence} is simply the special case when $k=\Q$, $k_\pfrak=\F_p$, so that $|k_\pfrak| = |\F_p|=p$.
\end{proof}
\begin{proof}[Proof of Theorem \ref{thm_strongly_gen_reducible_consequence_k}]
Since $F$ is strongly $n$-genuine, by Theorem \ref{thm_strongly_gen_ppties} (IV), for each $i_0 \in I$ and $I'=I\setminus\{i_0\}$, the form $B_\red(a_{\ell,j}(\bX_{I'})) \in \Z[\bX_{I'}]$ that detects whether $F(Y,X_{i_0},\bx_{I'})$ is reducible over $\overline{k} = \overline{\Q}$ is not identically zero. In particular this is true for $i_0=1$. Thus to prove the first claim, it suffices to observe that by Noether's Lemma \ref{lemma_Noether} (ii) followed by the trivial bound in Lemma \ref{lemma_Schwartz_domain} and then Lemma \ref{lemma_count_in_ring},
\begin{align*}
    \# \{\bx' & \in \Ocal_k^{n-1}, \|\bx'\|\leq B : \text{$F(Y,X_1,\bx')$ is reducible over $\overline{\Q}$}\}
   \\ 
   &\leq \#\{\bx' \in \Ocal_k^{n-1}, \|\bx'\|\leq B : B_\red(a_{\ell,j}(\bx_{I'}))=0\} \ll_{m,n,D} B^{n-2} (\log B)^{(n-2)(m-1)}.
\end{align*}
For the second claim, we define $\Ecal$ to be the set of all prime ideals that divide the gcd, call it $g$, of the coefficients of $B_\red(a_{\ell,j}(\bX_{I'})) \in \Z[\bX_{I'}]$. Note that $g \leq \|B_\red\|$ and by Lemma \ref{lemma_Noether},  $\log \|B_\red\| \ll_{n,D} \log \|F\|$. Consequently by Lemma \ref{lemma_count_in_ring}, $|\Ecal| \ll_{m,n,D} \log \|F\|/ \log \log \|F\|$. Then for $\pfrak \not\in \Ecal$, $B_\red(a_{\ell,j}(\bX_{I'}))$ is not identically zero over $k_\pfrak$. Thus to prove the second claim, we apply Noether's Lemma \ref{lemma_Noether} (ii) followed by the trivial bound in Lemma \ref{lemma_Schwartz_domain}:
\begin{align*}
    \# \{\bx' & \in k_\pfrak^{n-1} : \text{$F(Y,X_1,\bx')$ is reducible over $\overline{k_\pfrak}$}\}
   \\ 
   &\leq \#\{\bx' \in k_{\pfrak}^{n-1} : B_\red(a_{\ell,j}(\bx_{I'}))=0\} \leq \deg B_\red |k_{\pfrak}|^{n-2} .
\end{align*}
Since $B_\red \ll_{n,D}1$, this suffices.
\end{proof}
    
 \begin{rem}\label{remark_strongly_gen_consequence_weaker_hyp}
 As the method of proof showed, Theorem \ref{thm_strongly_gen_reducible_consequence_k} (and analogously Theorem \ref{thm_strongly_gen_reducible_consequence}) is still true under the following weaker hypothesis: that $F(Y,\bX)$ is irreducible over $k(\bX)$ and that for $i_0=1$ and $I'=\{2,\ldots,n\}$,
 \[   \overline{k(\bX_{I'})}\cap \mathcal{L}_F=k(\bX_{I'}).\]
 For then by Lemma \ref{lemma_Cohen_strongly_genuine_reducible}, the polynomial 
         $  B_{\red}(a_{\ell,j}(\bX_{I'})) \in \Z[\bX_{I'}]$ that detects whether $F(Y,X_1,\bx_{I'})$ is reducible over $\overline{\Q}$  is not identically zero, and the above proof can proceed.
 \end{rem}

\section{Special case of Theorem \ref{thm_Cohen}, and reduction to the special case}\label{sec_apply_strongly_genuine}

With the theory of strongly $n$-genuine polynomials in hand, we return to the verification  of Theorem \ref{thm_Cohen} in a special case (Theorem \ref{thm_Cohen_special}).  
 Recall the notation associated to Theorem \ref{thm_Cohen}, as defined in \S \ref{sec_Cohen_add_hyp}. Thus $\Omega_F$ is the splitting field of $F(Y,\bX)$ over $K(\bfX)$ and
  $L_F=\Omega_F\cap\overline{K}$, so that $K(\bfX)\subset L_F(\bfX)\subset\Omega_F$. We define $M_F(Y,\bfX) \in L_F(\bfX)[Y]$ to be the minimal polynomial of $\Omega_F$ over $L_F(\bfX)$; so we may assume that $M_F(Y,\bX)$ is monic in $Y$ and irreducible over $L_F(\bX)$, and $\Omega_F = L_F(X_1,\ldots,X_n)[Y]/(M_F(Y,\bX)).$ 
  
\subsection{Verification of Theorem \ref{thm_Cohen_special} for $n =1$}
For $n=1$, if $M_F$ is strongly $1$-genuine then 
  (in the notation of Proposition \ref{prop_Cohen_corrected_appendix} (i))  $B_{\red,1}^{M_F} \in \Z$ is nonzero, by Theorem \ref{thm_strongly_gen_ppties_n1} (I) $\Rightarrow$ (IV). Consequently the hypotheses of Proposition \ref{prop_Cohen_corrected_appendix} are all met, so the remainder of the proof of Theorem \ref{thm_Cohen} as in \cite[Thm. 2.1]{Coh81} can proceed, verifying the case $n=1$ of Theorem \ref{thm_Cohen_special}.

\begin{rem}[Original proof is correct for $n=1$]\label{remark_Cohen_n1_correct}
 Cohen's original proof requires no modification when $n=1$:
in the set-up where $M_F(Y,X_1)$ is the minimal polynomial of  $\Omega_F$ over $L_F$, $M_F$ must be strongly $1$-genuine. Indeed, it is strongly $1$-genuine by Theorem \ref{thm_strongly_gen_ppties_n1} (II) $\Rightarrow$ (I) since by definition $L_F:=\overline{K} \cap \Omega_F=\overline{\Q} \cap \Omega_F=\overline{L_F} \cap \Omega_F$.   Thus for $n=1$, Theorem \ref{thm_Cohen} is equivalent to Theorem \ref{thm_Cohen_special}. 
\end{rem}

   \subsection{Verification of Theorem \ref{thm_Cohen_special} for $n \geq 2$}
  For $n \geq 2$, suppose that for some $i_0 \in \{1,\ldots,n\}$, upon defining $I' = \{1,\ldots,n\}\setminus \{i_0\}$, 
     \beq\label{L_X1_special_n}
    \overline{L_F(\bX_{I'})} \cap \Omega_F = L_F(\bX_{I'}),
    \eeq
    that is to say $L_F(\bX_{I'})$ is integrally closed in $\Omega_F$.
 Then  (in the notation of Proposition \ref{prop_Cohen_corrected_appendix} (ii))  $B_{\red,i_0}^{M_F}(\bX_{I'}) \in \Z[\bX_{I'}]$ is not identically zero by Theorem \ref{thm_strongly_gen_ppties} (II) for $i_0$ $\Rightarrow$ (IV) for $i_0$. 
Consequently the hypotheses of Proposition \ref{prop_Cohen_corrected_appendix} are all met,   so that the outcome of Proposition \ref{prop_Cohen_corrected_appendix} replaces the desired outcome of Statement \ref{statement_Cohen} (ii). Thus the remainder of the proof of Theorem \ref{thm_Cohen} as in \cite[Thm. 2.1]{Coh81} can proceed, as already discussed.
In particular, if $M_F$ is strongly $n$-genuine then  (\ref{L_X1_special_n}) holds (in fact for all indices $i_0$)  by Theorem \ref{thm_strongly_gen_ppties} (I) $\Rightarrow$ (II).  This discussion has verified the proof of Theorem \ref{thm_Cohen_special} for $n\geq 2$.
Moreover, this discussion has proved another special case of Theorem \ref{thm_Cohen}, in which (\ref{L_X1_special_n*}) is the condition called (*) in the discussion below Theorem \ref{thm_Cohen_special}.

 \begin{thm}\label{thm_Cohen_special_weaker}
  Let  $K/\Q$ be a number field with ring of integers $\Ocal_K$. Let $F(Y,X_1,...,X_n) \in \calO_K[Y,X_1,\ldots,X_n]$   have total degree at most $D$ and Galois group $G$ over $K(X_1,\ldots,X_n)$. 
 If $n\geq 2$ suppose the additional condition: for some $i_0 \in \{1,\ldots,n\}$, upon defining $I' = \{1,\ldots,n\}\setminus \{i_0\}$, 
     \beq\label{L_X1_special_n*}
    \overline{L_F(\bX_{I'})} \cap \Omega_F = L_F(\bX_{I'}),
    \eeq
If $n=1$, no additional condition is required.
Then there exists a   constant $c$ depending only on $n,D,K$ such that  for all $N \geq 3$, 
 \[ \# \{ \bx \in \calO_K^n,  \|\bx\|\leq N: \text{$G(\bx)\not\simeq G$}\} \ll_{n,D,K}  \|F\|^{c} N^{n-\frac{1}{2}}\log N.\]
 \end{thm}

\subsection{A general notion of $\ell$-genuine and strongly $\ell$-genuine polynomials}
 
 Before beginning the recovery of Theorem \ref{thm_Cohen} in full generality, we need to introduce a generalized notion of $\ell$-genuine and strongly $\ell$-genuine polynomials in $Y,X_1,\ldots,X_n$, now allowing $1 \leq \ell \leq n$ to be considered. (We continue to state the definitions over a number field $k$, although the elementary considerations of this section apply without change in the case of an arbitrary field.)

\begin{defin}[$\ell$-genuine extension]\label{dfn_ell_genuine_generalized}
        We say that a finite (nontrivial) extension $M$ of $k(\bX) = k(X_1,\dots,X_n)$ is an \textbf{$\ell$-genuine extension} if for every $G(Y,\bX)\in k[Y,\bX] = k[Y,X_1,\dots,X_n]$ such that 
        $$M = k(\bX)[Y]/(G(Y,\bX)),$$
        there is some index set $|I|\geq \ell$ such that $G(Y,\bX)$ has nonzero degree in $X_i$ for all $i\in I$. We say that $M$ is a \textbf{strongly $\ell$-genuine extension} of $k(\bX)$ if for all subextensions $M'$ satisfying $$k(\bX)\subsetneq M' \subset M,$$
        $M'$ is an $\ell$-genuine extension of $k(\bX)$.
    \end{defin}

   We will require the following simple observation about strongly $1$-genuine extensions; this merely applies the reasoning of Remark \ref{remark_Cohen_n1_correct} in an $n$-variable setting.

    \begin{lem}\label{lemma_at_least_strongly_1_gen}
    Let $k$ be a number field and let $\Lcal$ be a nontrivial extension of $k(X_1,\ldots,X_n)$.
 The following are equivalent:

(I)  $k=\Lcal \cap \overline{k}$, or equivalently there is no number field $k'$ with $k \subsetneq k' \subset (\Lcal \cap \overline{k})$. 

(II) $\Lcal$ is  a strongly $1$-genuine extension of $k(X_1,\ldots,X_n)$. 

(III) For any extension $\mathcal{M}$ with $k(X_1,\ldots,X_n) \subsetneq \mathcal{M} \subset \Lcal$, $\mathcal{M}$ is an $\ell$-genuine extension of $k(X_1,\ldots,X_n)$ for some $\ell \geq 1$. 
\end{lem}

\begin{proof}
    An extension $k'$ with   $k \subsetneq k' \subset (\overline{k} \cap \Lcal)$ exists if and only if $k(\bX) \subsetneq k'(\bX) \subset \Lcal$, if and only if there  exists a polynomial $G(Y)$  independent of $X_1,\ldots,X_n$ and with   $\deg_Y G  \geq 2$, such that $k(\bX) \subsetneq k(X_1,\ldots,X_n)[Y]/G(Y) \subset \Lcal$, which occurs if and only if $\Lcal$ is not strongly $\ell$-genuine for any $\ell \geq 1$. Thus (I) $\Leftrightarrow$ (II). Similarly, (II) $\Leftrightarrow$ (III) since $\Lcal$ is  strongly $1$-genuine if and only if for every polynomial $G(Y,\bX)$ such that   $k(\bX) \subsetneq k(X_1,\ldots,X_n)[Y]/G(Y, \bX)) \subset \Lcal$, $\deg_{X_i}G \geq 1$ for some $i$, which is equivalent to (III).  
\end{proof}

\begin{cor} \label{cor_min_poly_at_least_strongly_1_gen}
Let $\Omega_F$ be the splitting field of $F(Y,\bX)$ over $K(\bfX)$,
  $L_F:=\Omega_F\cap\overline{K}$, and $M_F(Y,\bfX) \in L_F(\bfX)[Y]$ be the minimal polynomial of $\Omega_F$ over $L_F(\bfX)$. Then there is no number field extension $L' \supsetneq L_F$ such that 
  \[ L_F(\bX) \subsetneq L'(\bX) \subset \Omega_F = L_F(X_1,\ldots,X_n)[Y]/(M_F(Y,\bX)).\]
 Equivalently $M_F(Y,\bX)$ is  strongly $1$-genuine over $L_F$, and $\Omega_F$ is a strongly $1$-genuine extension of $L_F(\bX)$, in the sense that every extension $\mathcal{M}$ with $L_F(X_1,\ldots,X_n) \subsetneq \mathcal{M} \subseteq \Omega_F$ is an $\ell$-genuine extension of $L_F(X_1,\ldots,X_n)$ for some $\ell \geq 1$.
    \end{cor}
    \begin{proof}
For the first claim, by construction, since $L_F :=\overline{K} \cap \Omega_F$, no number field $L'$ with $L_F \subsetneq L' \subset (\overline{K} \cap \Omega_F)=(\overline{L_F} \cap \Omega_F)$ exists. For the equivalence, we may apply Lemma \ref{lemma_at_least_strongly_1_gen}.
    \end{proof}

\subsection{Reduction of the general case for $n \geq 2$}

 With Corollary \ref{cor_min_poly_at_least_strongly_1_gen} in hand, we begin discussing the strategy to recover Theorem \ref{thm_Cohen} in full generality by reducing to an application of Theorem \ref{thm_Cohen_special_weaker}.
From now on we need only consider the case $n\geq 2$. 
 We introduce a notation for a shifted polynomial: 
 for any  $\mathbf{a}\in \calO_K^{n-1}$ define
    \beq\label{app_shift_poly_dfn}
    F_{\bfa}(Y,X_{1},...,X_{n})=F(Y,X_{1},X_{2}+a_{2}X_{1},...,X_{n}+a_{n}X_{1}).
    \eeq
This is a linear transformation, say $F_\ba (Y,\bX) = F(Y,\sig_\ba(\bX))$  where $\sig_\ba \in \GL_n(K)$ has  associated matrix 
\[\sig_\ba=
\left( \begin{array}{cccc}
    1 &0 &\hdots &0\\
    a_2 & 1 &\hdots &0 \\
    \vdots &\vdots&\vdots &\vdots \\
    a_n &0&\hdots & 1 \end{array} \right).
\]
    \begin{lem}\label{lemma_shift_before_or_later}
    In the notation of Corollary \ref{cor_min_poly_at_least_strongly_1_gen}, for each $\ba \in \calO_K^{n-1}$, $L_{(F_\ba)} = L_F$. Moreover, 
    $(M_{F})_{\bfa}(Y,\bX) = M_{F_\ba}(Y,\bX)$,
        that is to say, the shift of the minimal polynomial is the minimal polynomial of the shift.
        \end{lem}
        \begin{proof}
 We first prove that for $\ba\in \calO_K^{n-1}$, $L_{(F_\ba)} = L_F$. Since by definition $L_F := \overline{K}\cap \Omega_F$, note that $L_F = \cap_{\bx\in K^n} \overline{K}\cap \Omega_{F,\bx}$, where $\Omega_{F,\bx}$ is the splitting field of $F(Y,\bx)$. On the other hand, $L_{(F_\ba)} = \cap_{\bx \in K^n} \overline{K}\cap \Omega_{F,\sig_\ba(\bx)} = \cap_{\bx\in \sig_\ba(K^n)} \overline{K} \cap \Omega_{F,\bx}$, in which $\sig_\ba$ is the linear transformation associated to $\ba$, as defined above. Now, $\sig_\ba\in \GL_n(K)$ is invertible, so indeed $L_{(F_\ba)} =  \cap_{\bx\in K^n} \overline{K} \cap \Omega_{F,\bx}=L_F,$ and the claim is proved.

For the second claim, we will show that $\Omega_{F_\ba} = L_F(\bX)[Y]/(M_F)_\ba(Y,\bX).$ Let us write $k:=L_F = L_{F_\ba}$ and $\Omega_F = k(\bX)(\alpha_1(\bX),\hdots, \alpha_D(\bX))$, where $\alpha_i(\bX)$ are algebraic expressions and the roots of $F(Y,\bX).$ Then since $\Omega_F = k(\bX)[Y]/M_F(Y,\bX)$, we know that $\alpha_i(X_1,X_2+a_2X_1,...,X_n+a_nX_1)\in k(\bX)[Y]/(M_{F})_{\bfa}(Y,\bX)$ for each $i$. Thus, $\Omega_{F_\ba}\subset k(\bX)[Y]/(M_{F})_{\bfa}(Y,\bX).$ Furthermore, we know that \[[\Omega_{F_\ba}:k(\bX)] = [\Omega_F:k(\bX)] = [k(\bX)[Y]/(M_{F})_{\bfa}(Y,\bX):k(\bX)].\]
 Hence  $\Omega_{F_{\bfa}} = k(\bX)[Y]/(M_{F})_{\bfa}(Y,\bX)$ as desired.
 
    \end{proof}

  We will transform $F(Y,\bX)$ to produce a shifted polynomial $F_\ba$  whose corresponding minimal  polynomial $M_{F_\ba}(Y,\bX) = (M_F)_\ba(Y,\bX)$  satisfies the hypothesis (\ref{L_X1_special_n*}) for $i_0=1$, and moreover $\log\|F_\ba\|\ll_{n,D,m}\log\|F\|$, via the following theorem. 
\begin{thm}\label{thm_appendix_shift_to_strongly_genuine}
Let $K/\Q$ be a finite extension and denote $m = \deg (K/\Q)$.
 Suppose $n\geq 2$ and let $F(Y,\bX) \in \calO_K[Y,X_1,\ldots,X_n]$ be a squarefree polynomial of total  degree $D$ and $\deg_Y F \geq 2$.  For any $\bfa \in \calO_K^{n-1}$,
    let $\Omega_{\bfa}$ be the splitting field of $F_{\bfa}(Y,\bX)$ over $K(\bX)$ and define the number field  
    \[ L:=L_{(F_\ba)}=\Omega_{\bfa} \cap \overline{K} = \Omega_F \cap \overline{K} =L_F,\]
    which is independent of $\ba$ by Lemma \ref{lemma_shift_before_or_later}. 
    Then there exists some $\bfa \in \calO_K^{n-1}$ with $\|\bfa\|\ll_{n,D,m} 1$ such that
      \beq\label{strongly_genuine_replacement}
     \overline{L(X_{2},...,X_{n})} \cap  \Omega_\ba = L(X_{2},...,X_{n}).
      \eeq  
\end{thm}

We will prove this theorem in \S \ref{sec_genuine} and  \S \ref{sec_shift_poly}. For now, we show why it suffices to complete the proof of Theorem \ref{thm_Cohen}.
 
\begin{proof}[Deduction of Theorem \ref{thm_Cohen}]
In  Theorem \ref{thm_Cohen}, consider the given polynomial $F(Y,X_1,\ldots,X_n)$ in $\calO_K[Y,X_1,\ldots,X_n]$  of total degree $ \leq D$ and Galois group $G$ over $K(\bX)$.  Fix once and for all a choice of  $\bfa \in \calO_K^{n-1}$ provided by Theorem \ref{thm_appendix_shift_to_strongly_genuine} so that (\ref{strongly_genuine_replacement}) holds, and define the shifted polynomial $F_{\bfa}$ as in (\ref{app_shift_poly_dfn}).
Since $\|\ba\| \ll_{n,D,m} 1$, then  $\|F_\ba\|\ll_{n,D,m} \|F\|$;  this follows from the relation that for any $a\in \calO_K$, $|N_{K/\Q}(a)|\ll_{m} H_K(a).$
Similarly, the property  $\|\bfa\|\ll_{n,D,m} 1$ implies that if   $\|\bx\|\leq N$  then $\bx_{\bfa} := (x_1,x_2-a_1x_1,\ldots,x_n-a_nx_1)$ also satisfies $\|\bx_{\bfa}\|\ll_{n,D,m} N$; this follows from the triangle inequality after expressing $x_1,\dots,x_n$ in terms of the chosen integral basis for $\calO_K$ that defines the height $H_K(\cdot)$. (Note the complementary signs of this shift, so that tautologically $F_{\bfa}(Y,x_{\bfa})=F(Y,\bx)$ for every $\bfa$.)

For each $\bx \in \calO_K^n$, denote the Galois group of $F(Y,\bx)$ over $K$ by $G(\bx)$. Our aim is to show that 
\[ \#\{\|\bx\|\leq N: G(\bx)\not\simeq G\} \ll_{n,D} \|F\|^{c/3} N^{n-1/2}\log N,\] 
for some $c$ depending only on $n,D,K$. Given $\bz \in \calO_K^n$, let $G_\ba(\bz)$ denote the Galois group of $F_\ba(Y,\bz)$ over $K$. If 
  $G(\bx)\not\simeq G$, then tautologically $G_\ba(\bx_{\bfa}) \not\simeq G$. Thus, 
\begin{equation*}
    \#\{\|\bx\|\leq N: G(\bx)\not\simeq G\} \leq \#\{\|\bx_{\bfa}\|\ll_{n,D,m} N: G_\ba(\bx_{\bfa})\not\simeq G\}.
\end{equation*}
Because (\ref{strongly_genuine_replacement}) holds,  Theorem \ref{thm_Cohen_special_weaker} applies to $F_\ba$, so that   for some $c$ (depending on $n,D,K$), for all $N\geq 3$,
\begin{equation*}
    \#\{\|\bz\|\ll_{n,D,m} N: G_\ba(\bz)\neq G\} \ll_{n,D,K} \|F_\ba\|^{c} N^{n-1/2}\log N\ll_{n,D,K} \|F\|^{c} N^{n-1/2}\log N,
\end{equation*}
where we have applied $\|F_{\bfa}\| \ll_{n,D,m} \|F\|$.

\end{proof}
Consequently, all that remains to recover Theorem \ref{thm_Cohen} is to prove Theorem \ref{thm_appendix_shift_to_strongly_genuine}. To do so, we first need to characterize the properties of $n$-genuine polynomials (analogous to Theorem \ref{thm_strongly_gen_ppties}); we turn to this in the next section.

\begin{rem}
Other theorems in \cite{Coh81} also apply \cite[Lemma 4.2(ii)]{Coh81} in their proof; it would be interesting to pursue whether the ideas of the present paper can be adapted to show that the other theorems can also be recovered by replacing Statement \ref{statement_Cohen} (ii) by Proposition \ref{prop_Cohen_corrected_appendix} and Theorem \ref{thm_strongly_gen_ppties}, possibly after passage to an appropriate strongly $n$-genuine polynomial, using the techniques developed in the present paper.

\end{rem}

\section{The theory of $n$-genuine polynomials}\label{sec_genuine}

 In this section, we characterize several useful properties of $n$-genuine polynomials, and prove a theorem analogous to Theorem \ref{thm_strongly_gen_ppties}. Now, for the larger class of $n$-genuine polynomials, property (II) of Theorem \ref{thm_strongly_gen_ppties} is replaced by a weaker condition. If $H(Y,\bX_I)$ is an $n$-genuine polynomial and $I' = I \setminus \{i_0\}$, we will see that  the integral closure of $k(\bX_{I'})$ in $\mathcal{L}_H:=k(\bX_{I})[Y]/H(Y,\bX_I)$ is  strictly smaller than $\mathcal{L}_H$, rather than being forced to be as small as $k(\bX_{I'})$. Additionally, the role of `reducibility' in (III) and (IV) of Theorem \ref{thm_strongly_gen_ppties} is replaced by the property of `having a linear factor' in $Y$ (or equivalently, `splitting completely' into linear factors in $Y$).  We briefly work over an arbitrary field, to prepare for an application in Theorem \ref{thm_gen_linear_factor_consequence_k} (which implies Theorem \ref{thm_gen_linear_factor_consequence}).

\begin{defin}   Let $\Kcal$ be a field. We say that a polynomial $f(Y,Z) \in \Kcal [Y,Z]$ of total degree $\deg f \geq 2$ {\emph{has a linear factor in $Y$ over $ \overline{\Kcal}$}} if 
    \beq\label{dfn_Y_linear_factor} 
    f(Y,Z) = (Y - Q(Z)) \Tilde{H}(Y,Z),
    \eeq
    where $Q(Z)\in \overline{\Kcal}[Z]$ and $\Tilde{H}(Y,Z)\in \overline{\Kcal}[Y,Z].$
    \end{defin} 
    Suppose $f$ belongs to the family of polynomials of total degree at most $D$ and we expand
    \beq\label{fYZ_expand}
    f(Y,Z)=\sum_{\substack{\ell,m\\\ell+m\leq D}}a_{\ell,m}Y^{\ell}Z^{m}.
    \eeq
     In the notation of Noether's Lemma  \ref{lemma_Noether},   $f$ satisfies  divisibility condition $\mathcal{D}((1,e_1))$ over $\overline{\Kcal}$ for $1+e_1 < \deg f$ precisely when $f$ factors over $\overline{\Kcal}$ either as $f(Y,Z)=G_0(Z)\tilde{H}(Y,Z)$ with $\deg_Y G_0=0$ and $1 \leq \deg_Z G_0 < \deg f$, or $G_1(Y,Z)\tilde{H}(Y,Z)$ with $\deg_Y G_1=1$.   Noether's Lemma produces a   form $B_\lin=B_{\lin}( (a_{\ell,m})_{\ell,m})$ with coefficients in $\Z$ such that
    \beq\label{define_utility_of_P_prepare}
    \text{$B_{\lin}( (a_{\ell,m})_{\ell,m})=0 \Leftrightarrow f(Y,Z)$ satisfies $\mathcal{D}((1,e_1))$ over $\overline{\Kcal}$ or $\deg f < D$. }
    \eeq

   We may guarantee the ``linear factor in $Y$'' option under a mild condition:
\begin{lem}\label{lemma_actual_linear_factor}
    Let $\Kcal$ be a field. If a polynomial $f(Y,Z) \in \Kcal[Y,Z]$ of total degree at most $D$ is monic in $Y$, then 
       \beq\label{define_utility_of_P}
    \text{$B_{\lin}( (a_{\ell,m})_{\ell,m})=0 \Leftrightarrow f(Y,Z)$ has a linear factor in $Y$ over $\overline{\Kcal}$ or $\deg f < D$. }
    \eeq
\end{lem}
\begin{proof}
Certainly if $f(Y,Z)$ has a linear factor in $Y$ over $\overline{\Kcal}$ or $\deg f<D$, then $B_{\lin}( (a_{\ell,m})_{\ell,m})=0$, by (\ref{define_utility_of_P_prepare}). Suppose on the other hand that $B_{\lin}( (a_{\ell,m})_{\ell,m})=0$, so that $f$ satisfies $\mathcal{D}((1,e_1))$: then it suffices to note that a factorization $f(Y,Z)=G_0(Z)\tilde{H}(Y,Z)$ cannot hold over $\overline{\Kcal}$, under the assumption that $f$ is monic in $Y$. (Indeed, suppose $f(Y,Z)=G_0(Z)\tilde{H}(Y,Z)$ over $\overline{\Kcal}$, so that $\deg_Y f =\deg_Y \tilde{H}=: D_Y$, and let $Y^{D_Y}P(Z)$ denote the part of $\tilde{H}$ of highest order in $Y$. Then we must have $1 \con G_0(Z)P(Z)$, which is impossible since $\deg_Z G_0(Z) \geq 1$.)
\end{proof}

      For clarity, note that if we apply this criterion over a number field $k$ to a polynomial $F(Y,X_{i_0},\bX_{I'})$ as a polynomial in $Y,X_{i_0}$ (monic in $Y$, with $I' = I \setminus \{i_0\}$), 
then for a given specialization $\bx_{I'} \in k^{|I'|}$,  $B_\lin((a_{\ell,m}(\bx_{I'}))_{\ell,m})=0$   if and only if $F(Y,X_{i_0},\bx_{I'})$ has a linear factor in $Y$ over $\overline{\Q}$ or the total degree of $F(Y,X_{i_0},\bx_{I'})$ is strictly less than $\deg_{Y,X_{i_0}}F(Y,X_{i_0},\bX_{I'})$, the total degree of $F(Y,X_{i_0},\bX_{I'})$ \emph{as a polynomial in $Y,X_{i_0}$.}

  It is useful to observe that if $f(Y,Z)$ has a linear factor in $Y$ over $\overline{\Kcal}$, then it splits completely.

  \begin{lem}\label{lemma_linear_factor_splits_completely}
 Let $\Kcal$ be a field.   If $f(Y,Z) \in \Kcal[Y,Z]$  has a linear factor in $Y$ over $\overline{\Kcal}$ and is irreducible over $\Kcal(Z)$, then $$f(Y,Z) = \prod_{j} (Y-Q_j(Z))$$
    for $Q_j(Z)\in \overline{\Kcal}[Z]$ for all $j$, so that $f(Y,Z)$ splits completely over $\overline{\Kcal}$.
\end{lem}
\begin{proof}
    Since $f(Y,Z)$ is irreducible over $\Kcal(Z)$,  the automorphisms of $\Kcal(Z)[Y]/(f(Y,Z))$  act transitively on the roots of $f(Y,Z)$. Since $f(Y,Z)$ has a linear factor in $Y$ over $\overline{\Kcal}$,   $$f(Y,Z) = \prod_{\sigma_j} (Y-\sigma_j(Q(Z)))$$
    for some polynomial $Q(Z)\in \overline{\Kcal}[Z],$ and $\sigma_j$ varying over 
    the group of embeddings of the extension $\Kcal(Z)[Y]/(f(Y,Z))$ in its Galois closure.  We observe that $\sigma_j(Q(Z))\in \overline{\Kcal}[Z]$ to complete the proof.
\end{proof}
We will later call upon the following equivalence:
\begin{lem}\label{lemma_splits_field_identity}
Let $\Kcal$ be a field. If $f(Y,Z) \in \Kcal[Y,Z]$ is irreducible over $\Kcal(Z)$, then $f(Y,Z)$ has a linear factor in $Y$ over $\overline{\Kcal}$ if and only if 
\beq\label{splits_field_identity}
(\Kcal(Z)[Y]/f(Y,Z) \cap \overline{\Kcal})(Z) = \Kcal(Z)[Y]/f(Y,Z).
\eeq
\end{lem}
\begin{proof}
    If  $f$ has a linear factor in $Y$ over $\overline{\Kcal}$ then by Lemma \ref{lemma_linear_factor_splits_completely}, all the roots of $f$ lie in $\overline{\Kcal}[Z]$, so that (\ref{splits_field_identity}) holds. In the other direction, since $f$ is irreducible over $\Kcal(Z)$ then for any root $W$ of $f$, $\Kcal(Z)[Y]/f(Y,Z) \simeq \Kcal(Z)[W]$. Thus if (\ref{splits_field_identity}) holds then $(\Kcal(Z)[Y]/f(Y,Z) \cap \overline{\Kcal})(Z)  \simeq \Kcal(Z)[W]$, so that $W$ must lie in $\Kcal'(Z)$ for some finite extension $\Kcal'/\Kcal$. That is to say, $W \in \overline{\Kcal}(Z)$, and then this implies that a factorization of the form (\ref{dfn_Y_linear_factor}) holds, so that $f$ has a linear factor in $Y$ over $\overline{\Kcal}$.
\end{proof}
 
\begin{rem}
    Lemma \ref{lemma_splits_field_identity} plays a similar role in the proof of our classification theorem for $n$-genuine polynomials to the role that Lemma \ref{lemma_integrally_closed_implies_absirred} plays in Theorem \ref{thm_strongly_gen_ppties} for the classification of strongly $n$-genuine polynomials. They are related in the following way. Over a number field $k$, Lemma \ref{lemma_integrally_closed_implies_absirred} gives that if $f(Y,Z)$ is absolutely irreducible, then $M_f :=\overline{k}\cap (k(Z)[Y]/f(Y,Z)) $ satisfies $M_f= k$. Lemma \ref{lemma_splits_field_identity} states that if $f(Y,Z)$ does not have a linear factor over $\overline{k}$ (a strictly weaker condition than being absolutely irreducible), then $M_f(Z)\subsetneq k(Z)[Y]/f(Y,Z)$. The latter is a strictly weaker requirement than $M_f = k$. 
\end{rem}

Next we characterize the key properties of $n$-genuine polynomials. For clarity, we separate the cases $n=1$ and $n\geq 2$. 
\begin{thm}[$1$-Genuine] \label{thm_gen_ppties_n1}
Let $k/\Q$ be a finite extension.   Let $H \in \calO_k(X_1)[Y]$ be irreducible over $k(X_1)$, monic in $Y$, and   of total degree $D$. Define $\mathcal{L}_H = k(X_1)[Y]/(H(Y,X_1))$. 
The following are equivalent: 
  
  (I) $H$ is a $1$-genuine polynomial, that is to say,  $k(X_1)[Y]/H(Y,X_1)$ is a $1$-genuine extension of $k(X_1)$. 

  (II) $ (\overline{k} \cap \mathcal{L}_H)(X_1) \subsetneq \mathcal{L}_H.$

  (III)    $H(Y,X_{1})$  does not have a linear factor in $Y$ over $\overline{\Q}$.

   (IV) Expanding the polynomial   as
   $
    H(Y,X_1)=\sum_{\ell+m\leq D}a_{\ell,m}Y^{\ell}X_{1}^{m},
  $
   the form $B_{\lin}(a_{\ell,m})_{\ell,m})$ defined as in (\ref{define_utility_of_P}), when evaluated at the coefficients $a_{\ell,m}$, is a nonzero integer.  
\end{thm}

\begin{thm}[$n$-Genuine] \label{thm_gen_ppties}
Let $k/\Q$ be a finite extension. Let $I$ be an index set of cardinality $n\geq 1$. Let $H \in \calO_k(\bX_I)[Y]$ be irreducible over $k(\bX_I)$, monic in $Y$, and   of total degree $D$. Define $\mathcal{L}_H = k(\bX_I)[Y]/(H(Y,\bX_I))$. 
Consider the following conditions.
  
  (I) $H$ is an $n$-genuine polynomial, that is to say,  $k(\bX_I)[Y]/H(Y,\bX_I)$ is an $n$-genuine extension of $k(\bX_I)$.

(II) Given any index $i_0$, upon defining $I':=I\setminus\{i_0\}$, the  following inclusion of fields is strict:
\[ (  \overline{k(\bX_{I'})} \cap \mathcal{L}_H)(X_{i_0}) \subsetneq \mathcal{L}_H .\]

  (III) Given any index $i_0$, upon defining $I':=I\setminus\{i_0\}$, there exists a point $\bx_{I'} \in k^{n-1}$ such that  $H(Y,X_{i_0},\bx_{I'})$, as a polynomial in $Y,X_{i_0}$, does not have a linear factor in $Y$ over $\overline{\Q}$,  and $\deg H(Y,X_{i_0},\bx_{I'}) = \deg_{Y,X_{i_0}} H(Y,\bX_I)$.

   (IV)  Given any index $i_0$, upon defining $I':=I\setminus\{i_0\}$, upon expanding the polynomial $H(Y,\bX_I)$ in $Y$ and $X_{i_0}$ as
    \beq\label{app_H_expansion}
    H(Y,\bX_I)=\sum_{\substack{\ell,m\\\ell+m\leq D}}a_{\ell,m}(\bX_{I'})Y^{\ell}X_{i_{0}}^{m},
    \eeq
   the form $B_{\lin}(a_{\ell,m}(\bX_{I'}))$ defined as in (\ref{define_utility_of_P}) is not identically zero as a polynomial in $\bX_{I'}$. 

The following are equivalent:
\[ \text{(I) $\Leftrightarrow$ (II) $\Leftrightarrow$(III) $\Leftrightarrow$(IV).}\]
 Furthermore, for any fixed index $i_0 \in I$, the following are equivalent:
\[   \text{(II) for $i_0$ $\Leftrightarrow$  (III) for $i_0$ $\Leftrightarrow$  (IV) for $i_0$}.\]

\end{thm}

  \subsection{Proof of Theorem \ref{thm_gen_ppties_n1} for $n =1$}
  (I) is the statement that we cannot write $\Lcal_H = k(X_1)[Y]/R(Y)$ for any polynomial $R$ independent of $X_1$; equivalently, if $L'$ is some extension with $k \subset L' \subset (\overline{k} \cap \Lcal_H)$, then $L'(X_1) \subsetneq \Lcal_H$; this confirms (I) $\Leftrightarrow$ (II). 
 The equivalence  (II) $\Leftrightarrow$ (III) holds by Lemma \ref{lemma_splits_field_identity}, which shows $H(Y,X_1)$ has a linear factor in $Y$ over $\overline{k}$ if and only if $(\Lcal_H \cap \overline{k})(X_1) =\Lcal_H$.  
  By Noether's Lemma \ref{lemma_Noether} and Lemma \ref{lemma_actual_linear_factor}, (III) $\Leftrightarrow$ (IV), since $H$ is monic in $Y$.

 \subsection{Proof of Theorem \ref{thm_gen_ppties} for $n \geq 2$}
 We  prove all of the more elementary relations, before extracting a more subtle relation (II) $\Rightarrow$ (III) as a lemma. Throughout, for any index set $I$, we say that a polynomial $f(Y,\bX_I) \in k[Y,\bX_I]$ has  a linear factor in $Y$ over $\overline{k}$ if 
 \[f(Y,\bX_I) = (Y - Q(\bX_I)) \tilde{H}(Y,\bX_I),\]
 in which $Q(\bX_I)\in \overline{k}[\bX_I]$ and $\Tilde{H}(Y,\bX_I)\in \overline{k}[Y,\bX_I].$
  \begin{lem}\label{lemma_H_not_n_gen_linear_factor}
 Given $H(Y,\bX_I)\in \Ocal_k[Y,\bX_I]$ that  is irreducible over $k[Y,\bX_I]$ and monic in $Y$, define $\mathcal{L}_H = k(\bX_I)[Y]/(H(Y,\bX_I))$. Suppose  that for some set $J \subsetneq I$, there is a  polynomial $G(Y,\bX_J) \in k[Y,\bX_J]$, irreducible over $k[Y,\bX_J]$, such that 
\[ (k(\bX_J)[Y]/G(Y,\bX_J))(\bX_{I\setminus J})=\mathcal{L}_H .\]
  Then   for all $\bx_J \in k^{|J|}$, 
  $H(Y,\bX_{I\setminus J},\bx_J)$ has a linear factor in $Y$ over $\overline{k}$. 
\end{lem}
In particular, the lemma applies if $H(Y,\bX_I)$ is not $|I|$-genuine.

\begin{proof}
Let $H(Y,\bX_I)$ and $G(Y,\bX_J)$ be given as in the lemma, so that $J \subsetneq I$ and 
\[ (k(\bX_J)[Y]/G(Y,\bX_J))(\bX_{I\setminus J})=\mathcal{L}_H .\]
Let $W_G$ be a root of $G$ so that $k(\bX_J)[Y]/G(Y,\bX_J) = k(\bX_J)(W_G)$; hence $\Lcal_H = k(\bX_I)(W_G)$. Let $W_H$ denote a root of $H(Y,\bX_I)$ such that $\Lcal_H = k(\bX_I)(W_H)$; then we can decompose 
\[H(Y,\bX_I) = (Y-W_H) \tilde{H}(Y)\]
over $\Lcal_H = k(\bX_I)(W_H) = k(\bX_{I})(W_G)$, that is to say $\tilde{H}(Y) \in \Lcal_H[Y]$. Note that as a polynomial in $Y$, $H(Y,\bX_I)$ is monic and has coefficients that are elements in $k[\bX_I]$ (that is, are polynomials in $\bX_I$); thus its root $W_H$ in $\Lcal_H$ must lie in the ring of integers of that field. Since this ring of integers is contained in $k(W_G)[\bX_I]$, this implies that  the root $W_H$ lies in $k(W_G)[\bX_{I}]$, and since  $W_G$ only depends on $\bX_J$, we finally conclude that $W_H$ is a polynomial (rather than  algebraic function) of $\bX_{I\setminus J}$. That is, we may write $W_H=Q(\bX_{I\setminus J},\bX_J)$ as a function that is a polynomial function of the variables $\bX_{I \setminus J}$ and an algebraic function of the variables $\bX_J$. 
A similar argument shows that $\tilde{H}(Y) \in k(W_G)[Y,\bX_{I}]$, so that it also depends only polynomially on $\bX_{I \setminus J}$.  That is to say, we may write $\tilde{H}(Y) = H^\sharp(Y,\bX_{I\setminus J}, \bX_J)$ as a function that is a polynomial function of $Y$ and the variables $\bX_{I \setminus J}$ and an algebraic function of the variables $\bX_J$. 

Consequently, when specialized to any $\bfx_{J}\in k^{|J|}$, we see that
\[
H(Y,\bX_{I\setminus J},\bx_J)=(Y-Q(\bX_{I\setminus J},\bx_J))H^\sharp (Y,\bX_{I\setminus J},\bx_J)
\]
for    $Q(\bX_{I\setminus J},\bx_J)\in\overline{k}[\bX_{I\setminus J}]$ and $H^\sharp (Y,\bX_{I\setminus J},\bx_J)\in\overline{k}[Y,\bX_{I\setminus J}]$, which are polynomials in the variables $Y$ and $\bX_{I \setminus J}$ with coefficients in $\overline{k}$. In particular, for any $\bx_{J} \in k^{|J|}$,  $H(Y,\bX_{I \setminus J},\bx_J)$ has a linear factor in $Y$ over $\overline{k}$, and the proof is complete.
\end{proof} 

 With this lemma in hand, we can prove several relations within Theorem \ref{thm_gen_ppties}.
 
 ((I) $\Rightarrow$ (II))  Suppose $H(Y,\bX_I)$ is $n$-genuine. Fix any $i_0 \in I$ and set $I'=I \setminus \{i_0\}$.    Let $G(Y,\bX_{I'})\in k(\bX_{I'})[Y]$ be a polynomial such $ k(\bX_{I'})[Y]/G(Y,\bX_{I'})$ is the integral closure of $k(\bX_{I'})$ in $\mathcal{L}_H= k(\bX_I)[Y]/H(Y,\bX)$, that is,
      \[ (k(\bX_{I'})[Y]/G(Y,\bX_{I'})) \simeq(\Lcal_H \cap \overline{k(\bX_{I'})}).\]
Since $H$ is $n$-genuine, it has $\deg_{X_{i_0}} H \geq 1$, whereas $G$ does not depend on $X_{i_0}$. Thus the hypothesis that $\Lcal_H$ is an $n$-genuine field implies
\[
(k(\bX_{I'})[Y]/G(Y,\bX_{I'}))(X_{i_0})\subsetneq \Lcal_H.
\]
Combining these two facts shows that (II) holds.

((II) $\Rightarrow$ (I)) We will prove the contrapositive. Suppose that $\mathcal{L}_H$ is not $n$-genuine, so there exists some polynomial $G(Y,\bX_J) \in k[Y,\bX_J]$ with an index set $J \subsetneq I$ such that 
\[ (k(\bX_J)[Y]/G(Y,\bX_J))(\bX_{I\setminus J})=\mathcal{L}_H .\]
Fix any index $i_0 \in I \setminus J$; then by writing $G$ nominally as a polynomial in $Y$ and $X_i$ with $i \neq i_0$, without loss of generality the above property holds in particular for the set $J=I'$ with $I'=I \setminus \{i_0\}$.
That is,  
\[(k(\bX_{I'})[Y]/G(Y,\bX_{I'}))(X_{i_0})= \mathcal{L}_H.\]
Note that $k(\bX_{I'})[Y]/G(Y,\bX_{I'}) \subset \overline{k(\bX_{I'})}$, and certainly $k(\bX_{I'})[Y]/G(Y,\bX_{I'}) \subset \Lcal_H$,
so the above relation implies that $(\overline{k(\bX_{I'})} \cap \mathcal{L}_H)(X_{i_0}) = \mathcal{L}_H$. This argument confirms (II) $\Rightarrow$ (I).

((III) $\Rightarrow$ (I)) We will prove the contrapositive, namely that if $H(Y,\bX_I)$ is not $n$-genuine, then for some $i_0 \in I,$ $H(Y,X_{i_0},\bx_{I'})$ has a linear factor in $Y$ over $\overline{\Q}$ for all $\bx_{I'} \in k^{n-1}$. This follows from Lemma \ref{lemma_H_not_n_gen_linear_factor}: since $H(Y,\bX_I)$ is not   $n$-genuine,
there exists a non-empty subset $J\subsetneq I$ and polynomial  $G(Y,\bX_J)$  such that
\[ (k(\bX_J)[Y]/G(Y,\bX_J))(\bX_{I\setminus J})=\mathcal{L}_H .\]
By the lemma, it follows that for every $\bfx_{J}  \in k^{|J|}$,   the polynomial $H (Y, \bX_{I \setminus J}, \bx_J)$  has a linear factor in $Y$ over $\overline{\Q}$. Then the desired conclusion is certainly true for any index $i_0 \in I \setminus J$.

((III) for $i_0$ $\Leftrightarrow$ (IV) for $i_0$) For a fixed index $i_0$, the condition in (III) for the index $i_0$ is equivalent to the condition in (IV) for the index in $i_0$,  by (\ref{define_utility_of_P}) of Lemma \ref{lemma_actual_linear_factor}. Indeed, for each point $\bx_{I'} \in k^{n-1}$,  the property
\[ \text{$H(Y,X_{i_0},\bx_{I'})$  does not have a linear factor in $Y$ over $\overline{\Q}$ and $\deg H(Y,X_{i_0},\bx_{I'})=\deg_{Y,X_{i_0}} H$}\]
occurs
if and only if  $B_{\lin}(a_{\ell,m}(\bx_{I'})) \neq 0.$

((III) for $i_0$ $\Rightarrow$ (II) for $i_0$) We will prove the contrapositive, namely that if 
\beq\label{LHX_assp} 
(\overline{k(\bX_{I'})}\cap \Lcal_H)(X_{i_0}) =  \Lcal_H,
\eeq
then for all $\bx_{I'}\in  k^{n-1}$, $H(Y,X_{i_0},\bx_{I'})$ has a linear factor in $Y$ over $\overline{k}$. There exists a polynomial $G(Y,\bX_{I'})\in k[Y,\bX_{I'}]$, irreducible over $k[Y,\bX_{I'}]$, such that
$$\overline{k(\bX_{I'})}\cap \Lcal_H = k(\bX_{I'})[Y]/G(Y,\bX_{I'}).$$
Under the hypothesis (\ref{LHX_assp}),
\[\Lcal_H = (k(\bX_{I'})[Y]/G(Y,\bX_{I'}))(X_{i_0}),\]
so we may apply Lemma \ref{lemma_H_not_n_gen_linear_factor} with subset $J=I'$, and obtain the desired conclusion.

This completes the proof of   the theorem, except for the relation that  (II) for $i_0$ $\Rightarrow$ (III) for $i_0$; this will be obtained in Lemma \ref{lemma_II_implies_III_gen}, via the following lemma.

\begin{lem}\label{lemma_open_V_gen}
     Let $k/\Q$ be a finite extension. Let $I$ be an index set of cardinality $n\geq 1$. Let $H \in \calO_k[Y,\bX_I]$ be irreducible over $k(\bX_I)$, monic in $Y$, and   of total degree $D$, and let $\mathcal{L}_H = k(\bX_I)[Y]/H(Y,\bX)$. For any $i_0 \in I$, and $I':=I \setminus \{i_0\}$, 
let $G$ be a polynomial in $k(\bX_{I'})[Y]$, irreducible over $k(\bX_{I'})$, such that $k(\bX_{I'})[Y]/G(Y,\bX_{I'})$ is the integral closure of $k(\bX_{I'})$ in $\mathcal{L}_H$, that is,
\[ k(\bX_{I'})[Y]/G(Y,\bX_{I'}) \simeq (\overline{k(\bX_{I'})} \cap \mathcal{L}_H).\]
Then there exists an open set $V \subset k^{n-1}$ such that for every $\bx_{I'} \in V$, $k[Y]/G(Y,\bx_{I'})$ is integrally closed in $k(X_{i_0})[Y]/H(Y,X_{i_0},\bx_{I'})$.
\end{lem}
 
\begin{proof}
Consider
$$W_H = \bSpec(k(X_{i_0},\bX_{I'})[Y]/H(Y,X_{i_0},\bX_{I'})), \qquad W_{G} = \bSpec(k(\bX_{I'})[Y]/G(Y,\bX_{I'})).$$
 The inclusion of fields gives us a morphism $$\pi: W_H\rightarrow W_{G}.$$
Let $V''$ be the set of points $u\in W_G$ where the fiber $$\pi_u:W_{H,\pi^{-1}(u)}\rightarrow W_{G,u}$$
is geometrically integral. We first claim that this contains a nonempty open set.  
By  \cite[EGA IV part 3, Theorem 9.7.7]{EGAIV_part3},  $V''$ is locally constructible. Since $G$ is defined so that $k(\bX_{I'})[Y]/G(Y,\bX_{I'})$ is integrally closed in $ k(\bX_I)[Y]/H(Y,\bX_I)$,  the generic fiber is geometrically integral and therefore contained in $V''$. Hence, $V''$ contains a nonempty open subset $V'$ such that for all $u \in V'$, $K(W_{G,u})$ is integrally closed in $K(W_{H,\pi^{-1}(u)})$. (Here we use the  notation   that $K(W_{u})$ is the field of rational functions on $W_{u}$.) 

Finally, let $\rho:W_{G}\rightarrow \mathbb{A}_k^{| I'|}$ and $\rho \circ \pi:W_{H}\rightarrow \mathbb{A}_k^{|I'|}$  be the morphisms induced by the maps 
\[
k(\bfX_{I'})\hookrightarrow k(\bX_{I'})[Y]/G(Y,\bX_{I'})\hookrightarrow k(X_{i_0},\bX_{I'})[Y]/H(Y,X_{i_0},\bX_{I'}).
\]
Since $k(\bX_{I'})[Y]/G(Y,\bX_{I'})$ is a finite extension of $k(\bfX_{I'})$, it follows that $\rho$ is a dominant and finite morphism, in particular $\rho$ is surjective, and closed (closed since $\rho$ is finite, and surjective because it is closed and dominant). Hence  for the open set $V' \subset V'' \subset W_G$ constructed above, $\rho(V')$ contains an open subset of $\mathbb{A}_k^{|I'|}$, say $V$. Notice that for every $\bfx_{I'}\in V$, one has that 
\[
\pi_{\bfx_{I'}}:W_{H,(\rho \circ \pi)^{-1}(\bfx_{I'})}\rightarrow W_{G,\rho^{-1}(\bfx_{I'})},
\]
is geometrically integral, since $\rho^{-1}(\bfx_{I'})\subset V'$, and 
\[
W_{H,(\rho \circ \pi)^{-1}(\bfx_{I'})}=\bigcup_{u\in \rho^{-1}(\bfx_{I'})}W_{H,\pi^{-1}(u)}.
\]
That is to say,
  $K(W_{G,\rho^{-1}(\bx_{I'})})=k[Y]/G(Y,\bx_{I'})$ is integrally closed in  $K(W_{H,(\rho \circ \pi)^{-1}(\bx_{I'})})=k(X_{i_{0}})[Y]/(H(Y,X_{i_{0}},\bx_{I'}))$.
\end{proof}

The following lemma confirms that (II for $i_0$) $\Rightarrow$ (III for $i_0$)  and thereby finishes the proof of   Theorem \ref{thm_gen_ppties}.
\begin{lem}[II for $i_0$ $\Rightarrow$ III for $i_0$]\label{lemma_II_implies_III_gen}
   Let $k/\Q$ be a finite extension. Let $I$ be an index set of cardinality $n\geq 1$. Let $H \in \calO_k[Y,\bX_I]$ be irreducible over $k(\bX_I)$, monic in $Y$, and   of total degree $D$, and let $\mathcal{L}_H = k(\bX_I)[Y]/H(Y,\bX)$. 
   
   Suppose that for some $i_0 \in I$ and $I' := I \setminus \{i_0\}$, 
 \[ (  \overline{k(\bX_{I'})} \cap \mathcal{L}_H)(X_{i_0}) \subsetneq \mathcal{L}_H .\]   
   Then for this index $i_0$, there exists a point $\bx_{I'} \in k^{n-1}$ such that  $H(Y,X_{i_0},\bx_{I'})$, as a polynomial in $Y,X_{i_0}$, does not have a linear factor in $Y$ over $\overline{\Q}$,  and has  $\deg H(Y,X_{i_0},\bx_{I'})=\deg_{Y,X_{i_0}}H(Y,\bX)$.
\end{lem}
 
\begin{proof}
Under the hypotheses, let $G$ be a polynomial in $k(\bX_{I'})[Y]$, irreducible over $k(\bX_{I'})$, such that $k(\bX_{I'})[Y]/G(Y,\bX_{I'})$ is the integral closure of $k(\bX_{I'})$ in $\mathcal{L}_H$, that is,
\[ k(\bX_{I'})[Y]/G(Y,\bX_{I'}) \simeq (\overline{k(\bX_{I'})} \cap \mathcal{L}_H).\]
Without loss of generality, we may assume $G$ is monic in $Y$.
Then by hypothesis,
\beq\label{G_H_strictly}
(k(\bX_{I'})[Y]/G(Y,\bX_{I'}))(X_{i_0})\subsetneq k(\bX_I)[Y]/(H(Y,X_{i_{0}},\bX_{I'})).
\eeq
Since both $G$ and $H$ are irreducible over $k$, so that
\[
[(k(\bX_{I'})[Y]/G(Y,\bX_{I'}))(X_{i_0}):k(\bfX)]=\deg_Y G,\quad [\mathcal{L}_{H}:k(\bfX)]=\deg_Y H,
\]
this implies that
\beq\label{degG_less_degH} 
\deg_Y G(Y,\bX_{I'}) < \deg_Y H(Y,X_{i_0},\bX_{I'}).
\eeq

By Hilbert's Irreducibility Theorem (Lemma \ref{lemma_HIT}), there is a dense set  $U'\subset k^{|I'|} =\Abb^{|I'|}_k$ such that for all $\bx_{I'} \in U'$, $H(Y,X_{i_{0}},\bx_{I'})$ and $G(Y,\bx_{I'})$ are irreducible over $k$. Observe that there is a subset $U \subset U'$, also dense in $\Abb^{|I'|}_k$, such that for all $x_{I'} \in U$,  $\deg_Y H(Y,X_{i_{0}},\bx_{I'}) = \deg_Y H$, $\deg_{X_{i_0}} H(Y,X_{i_{0}},\bx_{I'}) = \deg_{X_{i_0}} H$ and $\deg_Y G(Y,\bx_{I'}) = \deg_Y G$. 
The degree of each polynomial with respect to $Y$ does not change upon specialization since $H$ and $G$ are monic in $Y$. For the degree of $X_{i_0}$,   if we let $V_0$ denote the subset of $U'$ where  $\deg_{X_{i_0}} H(Y,X_{i_{0}},\bx_{I'}) < \deg_{X_{i_0}} H$ for $\bx_{I'} \in V_0$, then $V_0$ is nowhere dense, and $U=U' \setminus V_0$ is the desired dense set. (For $V_0$ is contained in a finite union of sets, each of which is defined as the vanishing set of a polynomial (a nonzero polynomial in $\bx_{I'}$, that defines the coefficient of a certain monomial in ${X_{i_0}}$   in $H(Y,X_{i_{0}},\bx_{I'})$, so $V_0$ is a proper closed subset of lower dimension, and is nowhere dense.)   

Let $V \subset \Abb_k^{|I'|}$ denote the open set produced by Lemma \ref{lemma_open_V_gen}.
Since $V$ is open and $U$ is dense, we can find a point  $\bx_{I'}\in V \cap U$, and then   $H(Y,X_{i_{0}},\bx_{I'})$ and $G(Y,\bx_{I'})$ are irreducible over $k$, with $\deg H(Y,X_{i_{0}},\bx_{I'})=\deg_{Y,X_{i_0}} H(Y,\bX_I)$. Moreover   $ k[Y]/G(Y,\bx_{I'})$ is integrally closed in $ k(X_{i_{0}})[Y]/(H(Y,X_{i_{0}},\bx_{I'}))$, which we record as 
\beq\label{Gkbar_relation}
k(X_{i_{0}})[Y]/(H(Y,X_{i_{0}},\bx_{I'}))\cap \overline{k}=k[Y]/G(Y,\bx_{I'}) ,
\eeq
also using the fact that $ \overline{k[Y]/G(Y,\bx_{I'})} \simeq \overline{k}$.
 By our choice of $\bx_{I'}$, we know that 
\[\deg_Y H(Y,X_{i_0},\bx_{I'})= \deg_Y H(Y,X_{i_0},\bX_{I'}), \qquad \deg_Y G(Y,\bx_{I'})= \deg_Y G(Y,\bX_{I'}),\] 
and that both $H(Y,X_{i_0},\bx_{I'})$ and $G(Y,\bx_{I'})$ are irreducible over $k$. Thus, an application of (\ref{degG_less_degH}) shows that $\deg_Y G(Y,\bx_{I'})<\deg_Y H(Y,X_{i_0},\bx_{I'})$,  and from the irreducibility it follows that we maintain the strict inclusion of fields:
\begin{equation}\label{GH_closure_setup}
    k[Y]/(G(Y,\bx_{I'}))(X_{i_0})\subsetneq k(X_{i_0})[Y]/(H(Y,X_{i_0},\bx_{I'})).
\end{equation}

We claim this implies that $H(Y,X_{i_{0}},\bx_{I'})$ does not have a linear factor in $Y$ over $\overline{k}$. 
Suppose on the contrary that $H(Y,X_{i_{0}},\bx_{I'})$ has a linear factor in $Y$ over $\overline{k}$; then it splits completely by Lemma \ref{lemma_linear_factor_splits_completely}, so that
\[ H(Y,X_{i_0},\bx_{I'}) = \prod_j (Y-Q_j(X_{i_0}))\] 
for certain $Q_j(X_{i_0}) \in \overline{k}[X_{i_0}]$. 
But  this implies (as in Lemma   \ref{lemma_splits_field_identity}) that
\[
(k(X_{i_{0}})[Y]/(H(Y,X_{i_{0}},\bx_{I'}))\cap \overline{k})(X_{i_0})= k(X_{{i}_{0}})[Y]/(H(Y,X_{i_{0}},\bx_{I'})).
\]
By applying the identity (\ref{Gkbar_relation}) in the left-hand side, this is the statement 
\[
(k[Y]/G(Y,\bx_{I'}))(X_{i_0})= k(X_{{i}_{0}})[Y]/(H(Y,X_{i_{0}},\bx_{I'})),
\]
in contradiction to (\ref{GH_closure_setup}).
Thus $H(Y,X_{i_{0}},\bx_{I'})$ does not have a linear factor in $Y$ over $\overline{\Q}$, and the lemma is proved.
\end{proof}

\subsection{Natural consequence of being $n$-genuine}

 \begin{thm}\label{thm_gen_linear_factor_consequence_k}
Let $n \geq 2$. Let $k/\Q$ be a finite extension of degree $m$, with ring of integers $\mathcal{O}_k$. Let $F(Y,\bX) \in \Ocal_k[Y,X_1,\ldots,X_n]$ be an $n$-genuine polynomial of total degree $D$. Then for all $B \gg 1$, 
\[ \# \{\bx' \in \Ocal_k^{n-1}, \|\bx'\|\leq B : \text{$F(Y,X_1,\bx')$ splits completely over $\overline{\Q}$}\}\ll_{m,n,D} B^{n-2}(\log B)^{(n-2)(m-1)}.\]
Also, there exists a finite set $\mathcal{E}$ of exceptional prime ideals $\pfrak \in \Ocal_k$, with $|\Ecal| \ll_{m,n,D} \log \|F\|/ \log \log \|F\|$, such that for all $\pfrak \not\in \Ecal$,  
\[ \# \{\bx' \in k_\pfrak^{n-1}: \text{$F(Y,X_1,\bx')$ splits completely over $\overline{k_\pfrak}$}\}\ll_{n,D} |k_\pfrak|^{n-2}.\]
\end{thm} 
\begin{proof}[Proof of Theorem \ref{thm_gen_linear_factor_consequence}]
Theorem \ref{thm_gen_linear_factor_consequence} is simply the special case when $k=\Q$, $k_\pfrak=\F_p$, so that $|k_\pfrak| = |\F_p|=p$.
\end{proof}
\begin{proof}[Proof of Theorem \ref{thm_gen_linear_factor_consequence_k}]
Since $F$ is   $n$-genuine, by Theorem \ref{thm_gen_ppties} (IV), for each $i_0 \in I$, the polynomial $B_\lin(a_{\ell,j}(\bX_{I'})) \in \Z[\bX_{I'}]$ that detects whether $F(Y,X_{i_0},\bx_{I'})$ has a linear factor over $\overline{k} = \overline{\Q}$ is not identically zero. 
In particular this is true for $i_0=1$, and then we denote $\bX_{I'}$ by $\bX'$. Thus to prove the first claim, it suffices to observe that by Lemma \ref{lemma_linear_factor_splits_completely}, then Lemma \ref{lemma_actual_linear_factor}, and then the trivial bound in Lemma \ref{lemma_Schwartz_domain} followed by Lemma \ref{lemma_count_in_ring},
\begin{align*}
    \# \{\bx' & \in \Ocal_k^{n-1}, \|\bx'\|\leq B : \text{$F(Y,X_1,\bx')$ splits completely over $\overline{\Q}$}\}
   \\ 
   &=     \# \{\bx'  \in \Ocal_k^{n-1}, \|\bx'\|\leq B : \text{$F(Y,X_1,\bx')$ has a linear factor in $Y$ over $\overline{\Q}$}\}\\
   &\leq \#\{\bx' \in \Ocal_k^{n-1}, \|\bx'\|\leq B : B_\lin(a_{\ell,j}(\bX'))=0\} \ll_{m,n,D} B^{n-2} (\log B)^{(n-2)(m-1)}.
\end{align*}
For the second claim, we define $\Ecal$ to be the set of all prime ideals that divide the gcd, call it $g$, of the coefficients of $B_\lin(a_{\ell,j}(\bX')) \in \Z[\bX']$. Note that $g \leq \|B_\lin\|$ and by Lemma \ref{lemma_Noether},  $\log \|B_\lin\| \ll_{n,D} \log \|F\|$. Consequently by Lemma \ref{lemma_count_in_ring}, $|\Ecal| \ll_{m,n,D} \log \|F\|/ \log \log \|F\|$. Then for $\pfrak \not\in \Ecal$, $B_\lin(a_{\ell,j}(\bX'))$ is not identically zero over $k_\pfrak$. To prove the second claim, it suffices again to apply Lemma \ref{lemma_linear_factor_splits_completely}, followed by Lemma \ref{lemma_actual_linear_factor} and the trivial bound in Lemma \ref{lemma_Schwartz_domain}:
\begin{align*}
    \# \{\bx' & \in k_\pfrak^{n-1} : \text{$F(Y,X_1,\bx')$ splits completely over $\overline{k_\pfrak}$}\} \\
    & = \# \{\bx'  \in k_\pfrak^{n-1} : \text{$F(Y,X_1,\bx')$ has a linear factor in $Y$ over $\overline{k_\pfrak}$}\}
   \\ 
   &\leq \#\{\bx' \in k_{\pfrak}^{n-1} : B_\lin(a_{\ell,j}(\bX'))=0\} \leq \deg B_\lin |k_{\pfrak}|^{n-2} .
\end{align*}
Since $\deg B_\lin \ll_{n,D} 1$, this suffices.
\end{proof}
    
 \begin{rem}\label{remark_gen_consequence_weaker_hyp}
 As the method of proof showed, Theorem \ref{thm_gen_linear_factor_consequence_k} (and analogously Theorem \ref{thm_gen_linear_factor_consequence}) is still true under the following weaker hypothesis: that $F(Y,\bX)$ is irreducible over $k(\bX)$, and monic in $Y$, and that for $i_0=1$ and $I'=\{2,\ldots,n\}$,
 \[   (\overline{k(\bX_{I'})}\cap \mathcal{L}_F)(X_1) \subsetneq \mathcal{L}_F.\]
 For then by Theorem \ref{thm_gen_ppties} ((II) for $i_0=1$ $\Rightarrow$ (IV) for $i_0=1$), the polynomial 
         $  B_{\lin}(a_{\ell,m}(\bX_{I'})) \in \Z[\bX_{I'}]$ that detects whether $F(Y,X_1,\bx_{I'})$ has a linear factor over $\overline{\Q}$  is not identically zero, and the above proof can proceed.
 \end{rem}

 \section{Shifting polynomials to prove Theorem \ref{thm_appendix_shift_to_strongly_genuine}}\label{sec_shift_poly}

 The remaining task to prove Theorem \ref{thm_Cohen} is to prove Theorem \ref{thm_appendix_shift_to_strongly_genuine}. Recall from (\ref{app_shift_poly_dfn}) that for any  $\mathbf{a}\in \calO_K^{n-1}$ we define the shift (with respect to $X_1$) by
   \[
    F_{\bfa}(Y,X_{1},...,X_{n})=F(Y,X_{1},X_{2}+a_{2}X_{1},...,X_{n}+a_{n}X_{1}).
\]
    We wish to show that for some shift $F_\ba$ of the polynomial $F$, the associated minimal polynomial $M_{F_\ba}$ satisfies condition (II) of Theorem \ref{thm_strongly_gen_ppties} for the particular index $i_0=1$. (We have distinguished the variable $X_1$, for notational simplicity, but we could in fact work with any fixed index $i_0$ and achieve an analogous outcome.) 
We first show that it suffices to prove the following result about shifting polynomials.  

 \begin{prop}\label{prop_app_construct_shift}
 Let $L/K/\Q$ be finite extensions.
Let $R\in\calO_{L}[Y,X_{1},...,X_{n}]$   be a  polynomial of total degree $D$, irreducible over $L(X_1,\ldots,X_n)$, such that 
\[\mathcal{L}_R := L(X_1,\ldots,X_n)[Y]/R(Y,X_1,\ldots,X_n)\] 
is a strongly $1$-genuine extension of $L(X_1,\ldots,X_n)$.
Then there exists a choice of $\mathbf{a}\in\calO_K^{n-1}$ with  $\|\bfa\|\ll_{n, D,[L:\Q]} 1$ such that for the shifted polynomial $R_\ba$,
    \beq\label{app_alg_closed}
\mathcal{L}_{R_{\bf{a}}}\cap\overline{L(X_{2},...,X_{n})}=L(X_{2},...,X_{n}).
    \eeq
\end{prop}

\begin{proof}[Deduction of Theorem \ref{thm_appendix_shift_to_strongly_genuine}]
Suppose this proposition is true, and let us deduce Theorem \ref{thm_appendix_shift_to_strongly_genuine}. Under the hypotheses of Theorem \ref{thm_appendix_shift_to_strongly_genuine}, let $\Omega_F$ denote the splitting field of $F(Y,\bX)$ over $K(\bX)$, with minimal polynomial $M_F(Y,\bX)$ of $\Omega_F$ over $L_F(\bX)$, in which $L_F=\Omega_{F} \cap \overline{K}$. Since $\deg_Y F \geq 2$, note that   $\deg_Y M_F(Y,\bX) \geq 2$. Additionally, since $M_F(Y,\bX)$ is a minimal polynomial it must be irreducible over $L_F(\bX)$. Furthermore Corollary \ref{cor_min_poly_at_least_strongly_1_gen} shows that $M_F$ is strongly $1$-genuine over $L_F$, or equivalently $\Omega_F$ is a strongly $1$-genuine extension of $L_F(X_1,\ldots,X_n)$, by construction.  

Apply Proposition \ref{prop_app_construct_shift} with $L=L_F$, $K=K$, and the polynomial $R=M_F(Y,\bX) \in L(\bX)[Y]$.  The proposition produces a choice of $\bfa \in \calO_K^{n-1}$ such that the shifted polynomial, which we denote by 
\[(M_{F})_{\bfa}(Y,\bX) = M_F(Y,X_1,X_2+a_2X_1,\ldots,X_n+a_nX_1),\]
satisfies $\mathcal{L}_{(M_{F})_{\bfa}}\cap\overline{L(X_{2},...,X_{n})}=L(X_{2},...,X_{n})$. 
By the proposition, $\|\bfa\|\ll_{n, D,[L:\Q]} 1$, and $[L:\Q]\ll_{D,[K:\Q]}1$, so that $\|\bfa\|\ll_{n, D,[K:\Q]} 1$.
By  Lemma \ref{lemma_shift_before_or_later}, $(M_{F})_{\bfa}(Y,\bX) = M_{F_\ba}(Y,\bX)$.
Thus $\mathcal{L}_{(M_{F_{\bfa}})}\cap\overline{L(X_{2},...,X_{n})}=L(X_{2},...,X_{n})$, and since $\Lcal_{{(M_{F_\ba}})}=\Omega_\ba$ in the notation of Theorem \ref{thm_appendix_shift_to_strongly_genuine}, the proof of that theorem is complete.
\end{proof}

For orientation, here is an overview of the strategy to prove Proposition \ref{prop_app_construct_shift}. If 
\beq\label{R_desired} 
\mathcal{L}_{R}\cap\overline{L(X_{2},...,X_{n})}=L(X_{2},...,X_{n}),
\eeq
so that $L(X_{2},...,X_{n})$ is integrally closed in $\mathcal{L}_R$, then the proposition is already true, with the zero shift $\bfa=0$. If this does not hold, then since strongly $n$-genuine polynomials are generic (recall Remark \ref{remark_generic}), we hope that by shifting $R$ we can produce a polynomial for which the  relation (\ref{R_desired}) does hold. So suppose
\beq\label{R_initial} 
N_R := \mathcal{L}_{R}\cap\overline{L(X_{2},...,X_{n})}\supsetneq L(X_{2},...,X_{n}).
\eeq
Enumerate all the intermediate extensions $L(X_1,\ldots,X_n) \subsetneq K_j \subseteq \Lcal_R$ (including $\Lcal_R$) by $K_1,\ldots, K_e$, say. (Each of these fields is $\ell$-genuine for some $\ell \geq 1$, by Lemma \ref{lemma_at_least_strongly_1_gen}.)
If it were true that for all $j=1,\ldots, e$, $K_j \not\subset N_R(X_1)$ then we could conclude that $N_R(X_1)=L(X_1,\ldots,X_n)$, since the $K_j$ exhaust the nontrivial intermediate fields. This would imply that $N_R=L(X_2,\ldots,X_n)$, so that (\ref{R_desired}) holds.
Presently we are not in this case, but the strategy is to shift $R$ to $R_{\bfa}$, and again enumerate the intermediate fields $L(X_1,\ldots,X_n) \subsetneq K_{j,\bfa} \subseteq \Lcal_{R_{\bfa}}$,  now denoted by $K_{1,\bfa}, \ldots, K_{e,\bfa}$, and hope in particular to find a choice of shift $\bfa$ such that for all $j=1,\ldots, e$, $K_j \not\subset N_{R_{\bfa}} (X_1)$, thus forcing $N_{R_{\bfa}}=L(X_2,\ldots,X_n)$.

The property $K_{j,\bfa} \not\subset N_{R_{\bfa}} (X_1)$ will hold if $K_{j,\bfa} \not\subset (K_{j,\bfa} \cap \overline{L(X_{2},...,X_{n})})(X_1)$,  that is to say, 
\beq\label{K_j_onevar}
(K_{j,\bfa} \cap    \overline{L(X_{2},...,X_{n})})(X_1) \subsetneq K_{j,\bfa},
\eeq
or in other words, if $K_{j,\bfa}$ requires some nontrivial algebraic expression in $X_1$. (For indeed, the property $K_{j,\bfa}\subset N_{R_{\bfa}} (X_1)$ would imply $K_{j,\bfa} \subset \overline{L(X_2,\ldots,X_n)}(X_1)$, which would imply $K_{j,\bfa}\subset (K_{j,\bfa} \cap \overline{L(X_{2},...,X_{n})})(X_1)$.)
Since the original polynomial $R(Y,\bX)$ is strongly $1$-genuine, each of the intermediate fields $K_1,\ldots,K_e$ depends in a nondegenerate way on at least one variable, say $X_j$, and so by shifting $X_j \mapsto X_j+a_j X_1$ we aim to introduce nondegenerate dependence on $X_1$, to result in (\ref{K_j_onevar}). For  a given field, say $F$, (that is $\ell$-genuine for some $\ell \geq 1$), we perform such a shift in one variable to produce a property like (\ref{K_j_onevar}) in Lemma \ref{lemma_app_shifting}, and by shifting in multiple coordinates in Lemma \ref{lemma_shift_all_vars_R}.  Moreover, since being $n$-genuine is a generic property, we would expect most shifts $\bfa$ accomplish this. We  quantify this in Lemma \ref{lemma_shift_all_vars_R} by showing that for a given field $F$, there exists a polynomial $g^F$ such that for every $\bfa$ such that $g^F(\bfa) \neq 0$, we do obtain (\ref{K_j_onevar}). 
With this result in hand, we can apply it to each of $K_1,\ldots,K_e$ in turn, generate the corresponding polynomials $g^1,\ldots,g^e$, and find a shift $\bfa$ that is not a root of any of them. For this shift $\bfa$, we achieve (\ref{K_j_onevar}) simultaneously for all $j=1,\ldots, e$, and this is the key step to prove the proposition. We now turn to carrying out this strategy rigorously.

\subsection{Step 1: shifting one variable}

The main goal of this section is the following lemma:
 \begin{lem}\label{lemma_app_shifting}
 Let $L/\Q$ be a finite extension. Let $R\in\calO_{L}[Y,X_{1},...,X_{n}]$   be a polynomial of total degree $D$, irreducible over $L(X_1,\ldots,X_n)$.  
Assume that  $\mathcal{L}_R:=L(X_1,\ldots,X_n)[Y]/R$ is  an $\ell$-genuine extension  for some $\ell \geq 1$. 
   Then for every $i\in I$, for all but $O_{n,D}(1)$ choices of $\alpha_{i}\in L$, the following is true: if one considers the polynomial $R_{\alpha_{i},i}=R(Y,X_{1},...,X_{i}+\alpha_{i}X_{1},...,X_{n})$,  and the extension $\mathcal{L}_{R_{\alpha_{i},i}}:=L(X_1,\ldots,X_n)[Y]/R_{\alpha_{i},i}$ then the total degree of $R_{\alpha_{i},i}$ is $D$, and
 \[
 (\mathcal{L}_{R_{\alpha_{i},i}}\cap\overline{L(X_{2},...,X_{n})})(X_{1}) \subsetneq \mathcal{L}_{R_{\alpha_{i},i}}.
\]
\end{lem}
  Consider such a polynomial   $R(Y,X_1,\ldots,X_n)$  that is $\ell$-genuine. For some $I\subset\{1,...,n\}$ with  $|I|\geq \ell$   there exists a polynomial $H(Y,\bX_I) \in \Ocal_L[Y,\bX_I]$ such that $H(Y,\bX_I)$   is $|I|$-genuine (and in particular monic in $Y$) and
    \[
    \mathcal{L}_R=(L(\bX_I)[Y]/H)(\bX_{I^c}).
    \] 
 If we shift a variable $X_i$ with $i \in I$, $i \neq 1$,   by a multiple of $X_1$, it will be useful to show that property (IV) for $i_0=1$ of Theorem \ref{thm_gen_ppties} still holds for the shifted polynomial; this is the subject of the next  lemma, for which we establish the following notation.

 Fix   an index $i \in I$ with $i \neq 1$, and consider for any $\al_i \in L$ the polynomial $ H(Y,X_{i}+\alpha_{i} X_{1},\bX_{I\setminus\{i\}})$, in which this notation implies $X_{i}+\al_{i}X_1$ appears in the place of $X_{i}$. Note that this is a polynomial in $X_1$ and $\bX_I$.
   We  expand this polynomial in terms of $Y$ and $X_1$ as
    \beq\label{app_H_expansion_X1}
    H(Y,\alpha_{i}X_{1}+X_{i},\bX_{I\setminus\{i\}})=\sum_{\substack{\ell,m\\\ell+m\leq D}}b_{\ell,m}(\bX_{I'},\alpha_{i}) Y^{\ell}X_{1}^{m}.
    \eeq
    Here and throughout, we use the convention that $I'=I \setminus \{1\}$ if $1 \in I$ and  $I'=I$ if $1 \not\in I$.
   Consider also the detection polynomial $B_{\lin}(b_{\ell,m}(\bX_{I'},\alpha_{i}))$ provided by  Lemma \ref{lemma_actual_linear_factor}, that vanishes if and only if: as a polynomial in $Y,X_1$, the expression (\ref{app_H_expansion_X1}) has a linear factor in $Y$ over $\overline{\Q}$ (or has  degree  $< \deg_{Y,X_1}H(Y,\bX_{I})$).  Note  $B_{\lin}(b_{\ell,m}(\bX_{I'},\alpha_{i}))$ is a polynomial in   $\bX_{I'}$ and $\alpha_{i}$ with coefficients in $\Z$, and degree $\ll_{n,D} 1$.  
   
   \begin{lem}\label{lemma_P_shiftedH_not_zero}
  In the above setting, suppose $H(Y,\bX_I)$ is $|I|$-genuine. Suppose $|I| \geq 2$ or $1 \not\in I$, and  define $I'=I \setminus \{1\}$ if $1 \in I$ or $I'=I$ if $1\not\in I$. Fix any $i \in I'$ and expand the polynomial $H(Y,\alpha_{i}X_{1}+X_{i},\bX_{I\setminus\{i\}})$  as in (\ref{app_H_expansion_X1}).
   Aside from $O_{n,D}(1)$ possible values of $\al_{i} \in L$, the polynomial $B_{\lin}(b_{\ell,m}(\bX_{I'},\alpha_{i}))$ is not identically zero as a polynomial in $\bX_{I'}$.
   \end{lem}

       \begin{proof} 
First, since $H(Y,\bX_I)$ is $|I|$-genuine, by Theorem \ref{thm_gen_ppties} (I) $\Rightarrow$ (III) we see that for every $i_0 \in I$, there exists $\bx_{I\setminus \{i_0\}} \in L^{|I|-1}$ such that
   \begin{multline}\label{app_H_no_linear_factor}
    \text{$H(Y,X_{i_0},\bx_{I\setminus\{i_0\}})$ does not have a linear factor in $Y$ over $\overline{\Q}$, }\\\text{and   $\deg_{Y,X_{i_0}}H(Y,X_{i_0},\bx_{I\setminus\{i_0\}})$ $=\deg_{Y,X_{i_0}} H(Y,\bX_I)$}.
    \end{multline}
  (If $I=\{i_0\}$, this states that $H(Y,X_{i_0})$ does not have a linear factor in $Y$ over $\overline{\Q}$.)   
      
         We  distinguish two cases, depending whether $1 \in I$ or $1 \not\in I$. 
  In the case
when $1\in I$ (and so $|I| \geq 2$), fix any $i \in I'=I\setminus \{1\}$. Also fix a choice $\bx_{I'}$   as provided by (\ref{app_H_no_linear_factor}) applied with $i_0=1$, and for this choice consider  $B_{\lin}(b_{\ell,m}(\bx_{I'},\alpha_{i}))$ as a polynomial in $\alpha_{i}$. For our choice of $\bx_{I'}$,  we claim when $\al_i=0$, 
$B_{\lin}(b_{\ell,m}(\bx_{I'},0)) \neq 0$. Indeed upon setting $\al_i=0$ in (\ref{app_H_expansion_X1}), by definition,
\[\left. H(Y,X_{i},\bX_{I\setminus\{i\}})\right|_{\bX_{I'}=\bx_{I'}}
= H(Y,X_1,\bx_{I'}).
\]
By our choice of $\bx_{I'}$, the right-hand side satisfies (\ref{app_H_no_linear_factor}), so the left-hand side does too and the claim is proved. Thus as a polynomial in $\al_{i}$,   $B_{\lin}(b_{\ell,m}(\bx_{I'},\alpha_{i})) $ is not identically zero, and consequently $B_{\lin}(b_{\ell,m}(\bx_{I'},\alpha_{i}))\neq 0$ for all  but possibly $O_{n,D}(1)$ exceptional values of $\alpha_{i}$. Now consider any value $\al_{i}$ such that $B_{\lin}(b_{\ell,m}(\bX_{I'},\alpha_{i}))\equiv 0$ as a polynomial in $\bX_{I'}$; then certainly $B_{\lin}(b_{\ell,m}(\bx_{I'},\alpha_{i}))=0$. Thus aside from all  but possibly $O_{n,D}(1)$ exceptional values of $\alpha_{i}$, $B_{\lin}(b_{\ell,m}(\bX_{I'},\alpha_{i}))\not\equiv 0$, as desired.

       In the second case, when $1\not\in I$, let $I'=I$ and fix any  $i \in I'$.
        Suppose $\al_{i} \neq 0$. 
        Then taking a choice of $\bx_{ I\setminus\{i\}}$ with the property (\ref{app_H_no_linear_factor}) (applied with the choice $i_0=i$), $H(Y,X_{i},\bx_{I\setminus\{i\}})$ does not have a linear factor in $Y$ over $\overline{\Q}$ (and has degree $=\deg_{Y,X_i}H$);  therefore, the same properties hold for $H(Y,\al_{i} X_i,\bx_{I\setminus\{i\}})$ since $\al_i\neq 0$. Since $1 \not\in I$, this is equivalent to the statement, with $\al_{i} X_1$ now in the place of $\al_iX_{i}$,  that  $H(Y,\al_{i} X_{1},\bx_{I\setminus\{i\}})$ does not have a linear factor in $Y$ over $\overline{\Q}$ (and has total degree (as a function of $Y,X_1$) equal to $\deg_{Y,X_i} H$). On the other hand, in the notation of (\ref{app_H_expansion_X1}) (recalling $I' = I$ in the present case),  
        \[
            H(Y,\al_{i}X_{1},\bX_{I\setminus\{i\}})=H(Y,\al_{i}X_{1}+0,\bX_{I\setminus\{i\}})=\sum_{\substack{\ell,m\\\ell+m\leq D}}\left. b_{\ell,m}(\bX_{I'},\alpha_{i})\right|_{X_{i}=0} Y^{\ell}X_{1}^{m}.
            \] 
          From these two facts we may conclude that for any nonzero $\al_{i}$, $B_{\lin}(b_{\ell,m}(\bX_{I'},\alpha_{i}))$ evaluates to a nonzero value when specialized to  $X_{i}=0$ and $\bX_{I \setminus \{i\}} = \bx_{I \setminus \{i\}}.$ Consequently, for any nonzero $\al_{i}$ we conclude  $B_{\lin}(b_{\ell,m}(\bX_{I'},\alpha_{i})) \not\equiv 0$ as a polynomial in $\bX_{I'}$. 
   
        \end{proof}

\begin{proof}[Proof of Lemma \ref{lemma_app_shifting}]

    Let the polynomial $R$  be given as in the lemma, with an $|I|$-genuine polynomial $H$ such that
    \beq\label{R_H_first_lemma}
    \mathcal{L}_R=(L(\bX_I)[Y]/H(Y,\bX_I))(\bX_{I^c}).
    \eeq
    We first claim that for every $i\in I$, for all but $O_{n,D}(1)$ values of $\alpha_{i}$, the polynomial $R_{\alpha_{i},i}=R(Y,X_{1},...,X_{i}+\alpha_{i}X_{1},...,X_{n})$ has total degree $D$. We can write the polynomial $R_{\alpha_{i},i}$ as
    \begin{equation}
    R_{\alpha_{i},i}=\sum_{\bfj,\ell}p_{\bfj,\ell}(\alpha_{i})\bfX^{\bfj}Y^{\ell},
    \label{eq : nolossdegree}
    \end{equation}
    where for every $\bfj=(j_{1},...,j_{n})$ and for every $\ell$,  $p_{\bfj,\ell}(A_{i})$ is a polynomial of degree $\leq D$, whose coefficients depend on the coefficients of $R$. Note that $p_{\bfj,\ell}(A_i)\equiv 0$ if $\ell +j_{1}+\cdots +j_{n} >D$ so the  decomposition (\ref{eq : nolossdegree}) has $O_{n,D}(1)$ summands. Each (nonzero) polynomial $p_{\bfj,\ell}(A_i)$ vanishes for at most $D$ values of $\alpha_{i}$. Thus, for all but $O_{n,D}(1)$ values of $\alpha_{i}$, $p_{\bfj,\ell}(\alpha_{i})\neq 0$, for all $p_{\bfj,\ell}$ appearing in $(\ref{eq : nolossdegree})$, in which case the total degree of $ R_{\alpha_{i},i}$ is $D$, and the claim is proved.
    
    We now prove that for all but $O_{n,D}(1)$ choices of $\alpha_{i}\in L$,
     \[(\mathcal{L}_{R_{\alpha_{i},i}}\cap\overline{L(X_{2},...,X_{n})})(X_{1}) \subsetneq \mathcal{L}_{R_{\alpha_{i},i}}.
\]
 First, we deal separately with the case in which  $1 \in I$ and the distinguished index considered in the lemma is $i=1$. If $I=\{1\}$ then we already have that $(\mathcal{L}_{H}\cap\overline{L(X_{2},...,X_{n})})(X_{1})\subsetneq  \mathcal{L}_{H}$.
 This then implies that 
 $(\mathcal{L}_{R}\cap\overline{L(X_{2},...,X_{n})})(X_{1})\subsetneq  \mathcal{L}_{R}$
 via (\ref{R_H_first_lemma}).
 Moreover  the desired conclusion of the lemma is true for all $\al_{1}$ such that  $\alpha_{1}+1\neq 0$ (we can see this by a change of variable $X_{1}'=(1+\alpha_{1})X_{1}$). An analogous argument also works if $|I| \geq 2$ and $i=1 \in I$.
      Henceforward we may suppose that either $1\not\in I $ or $|I|\geq 2$, and in the latter case we only need to prove the lemma in the case when the distinguished index $i \neq 1$.

     Let $I'=I\setminus \{1\}$ if $1 \in I$ and $I'=I$ if $ 1\not\in I$.
  Now we fix $i \in I'$ and apply Lemma \ref{lemma_P_shiftedH_not_zero}:
thus aside from $O_{n,D}(1)$ possible values of $\al_{i} \in L$, the polynomial $B_{\lin}(b_{\ell,m}(\bX_{I'},\alpha_{i}))\not\equiv 0$ as a polynomial in $\bX_{I'}$.
   For each  $\al_{i}$ with the property $B_{\lin}(b_{\ell,m}(\bX_{I'},\alpha_{i}))\not\equiv 0$, there is an open set $U \in L^{|I'|}$ such that for all  $\bfx_{I'} \in U$, $B_{\lin}(b_{\ell,m}(\bx_{I'},\alpha_{i}))\neq 0$. Note that $B_{\lin}(b_{\ell,m}(\bx_{I'},\alpha_{i}))\neq 0$  occurs if and only if $H(Y,\alpha_{i}X_{1}+x_{i},\bfx_{ I'\setminus\{i\}})$ has no linear factor in $Y$ over $\overline{\Q} = \overline{L}$ and has total degree $=\deg_{Y,X_1} H(Y,X_1,X_2,\ldots,X_n)$.
   
   On the other hand, note that since by hypothesis $H(Y,\bX_I)$ is irreducible over $L(X_1,\ldots,X_n)$ then so is $H(Y,\al_iX_1+X_i,\bX_{I\setminus \{i\}})$, since it is obtained by a linear transformation. 
   Thus
   by the Hilbert Irreducibility Theorem (Lemma \ref{lemma_HIT}), there is a dense set $V \subset L^{|I'|}$ such that for all $\bx_{I'} \in V$, $H(Y,\alpha_{i}X_{1}+x_{i},\bfx_{ I'\setminus\{i\}})$ is irreducible over $L$. Since $U$ is open and $V$ is dense, we may choose $\bx_{I'} \in U \cap V$. 
   For this choice, by Theorem \ref{thm_gen_ppties} (III for $i_0=1$) $\Rightarrow$ (II for $i_0=1$) applied over the field $L$, 
   we learn that
    \begin{multline*}
    (L(\bfX)[Y]/H(Y,\alpha_{i}X_{1}+X_{i},\bX_{I\setminus\{i\}})\cap\overline{L(X_{2},...,X_{n})})(X_{1})\\
    \subsetneq L(\bfX)[Y]/H(Y,\alpha_{i}X_{1}+X_{i},\bX_{ I\setminus\{i\}}).
    \end{multline*}
    The final step is to note that 
    \beq\label{R_H_change_together}
    \mathcal{L}_{R_{\alpha_{i},i}} := L(\bX)[Y]/R_{\alpha_i,i} = L(\bX)[Y]/H(Y,\alpha_iX_1+X_i,\bX_{I\setminus\{i\}}),
    \eeq
    so that the previous identity is the conclusion of   Lemma \ref{lemma_app_shifting}, as desired.

   \end{proof}

\subsection{Step 2: shifting multiple variables}
  
For our second step we  iterate the shifting process $X_i \mapsto a_iX_1+X_i$ for each index $i\neq 1$, $i \in I$. The main goal of this section is the following lemma:
 \begin{lem}\label{lemma_shift_all_vars_R}
 Let $L/K/\Q$ be finite extensions.
Let $R\in\calO_{L}[Y,X_{1},...,X_{n}]$   be a polynomial of total degree $D$ such that $\mathcal{L}_R = L(X_1,\ldots,X_n)[Y]/R$ is an $\ell$-genuine extension of $L(X_1,\ldots,X_n)$ for some $\ell \geq 1$.    For any $\bfa\in\calO_K^{n-1}$ define
the polynomial 
 \[
    R_{\bfa}(Y,X_{1},...,X_{n}):=R(Y,X_{1},X_{2}+a_{2}X_{1},...,X_{n}+a_{n}X_{1})
    \]
      and correspondingly set 
    \[ \mathcal{L}_{R_\mathbf{a}} := L(X_1,\ldots,X_n)[Y]/R_{\mathbf{a}}.\]
  There exists a nonzero polynomial $g^R\in\calO_K[A_{2},...,A_{n}]$, such that for every $\bfa\in\calO_K^{n-1}$ with $g^R(\bfa)\neq 0$, 
 \[
 (\mathcal{L}_{R_{\bfa}}\cap\overline{L(X_{2},...,X_{n})})(X_{1}) \subsetneq \mathcal{L}_{R_{\bfa}}.
\]
Moreover $\deg g^R  \ll_{n,D, [L:K]} 1$.

\end{lem}
\begin{rem} Note that here we require the shifts $\ba$ to lie in the lower ring of integers $\Ocal_K^{n-1}$. This is because later we will apply this to $R=M_F$ the minimal polynomial for some $F \in \Ocal_K$, where $M_F \in \Ocal_{L_F}$. Then we will use the relation $(M_F)_\ba = M_{F_\ba}$ from Lemma \ref{lemma_shift_before_or_later}, and we want $F_\ba$ also to have coefficients in $\Ocal_K$.
\end{rem}
 
     To prepare for the proof, we suppose that $R$ is given, so that  $\mathcal{L}_R$ is an $\ell$-genuine extension of $L(X_1,\ldots,X_n)$ for some $\ell \geq 1$. Let $I\subset\{1,...,n\}$ be a set with  $|I|\geq \ell$ such that there exists a polynomial $H(Y,\bX_I) \in \Ocal_L[Y,\bX_I]$ of total degree $d$, irreducible over $L(\bX_I)$ and $|I|$-genuine such that
    \beq\label{R_defines_H}
    \mathcal{L}_R=(L(\bX_I)[Y]/H(Y,\bX_I))(\bX_{I^c}).
    \eeq
    In the proof of the lemma, we will immediately reduce to the case where $1 \not\in I$ or $|I| \geq 2$.    We establish some notational conventions so that we can treat both the cases  together.  If $1 \not\in I$, then $H(Y,\bX_I)$ does not depend on $X_1$; we will set $I'=I$ and our principal study will be the polynomial $H(Y,X_i + A_iX_1)_{i \in I'}$, which is a polynomial in $Y,X_1, \bX_I$. If $1 \in I$, we will set $I' = I \setminus \{1\}$ and our principal study will be   the polynomial $H(Y,X_1,X_i + A_iX_1)_{i \in I'}$, which is again a polynomial in $Y,X_1, \bX_I$. With an abuse of notation, we will unify both of these cases by studying a polynomial we denote as $H(Y,X_1,X_i + A_iX_1)_{i \in I'}$ in either case.
    
  With the convention $I'=I\setminus\{1\}$ if $1 \in I$ and $I'=I$ if $1 \not\in I$, expand as a polynomial in $Y$ and $X_1$ the polynomial $H$ after shifting each $X_i$ with $i \neq 1$:
    \beq\label{app_H_expansion_X1.1}
    H(Y,X_{1},X_{i}+A_{i}X_{1})_{i \in I'}=\sum_{\substack{\ell,m\\\ell+m\leq D}}b_{\ell,m}(\bfX_{I'},\bfA_{I'})Y^{\ell}X_{1}^{m}.
    \eeq
   Consider the detection polynomial $B_{\lin}(b_{\ell,m}(\bfX_{I'},\bfA_{I'}))$ that detects if this polynomial, as a polynomial in $Y,X_1$, has a linear factor in $Y$ over $\overline{\Q}$ (or degree $< \deg_{Y,X_1} H$), as provided by Lemma \ref{lemma_actual_linear_factor}. Note that $B_{\lin}(b_{\ell,m}(\bfX_{I'},\bfA_{I'}))$  is a polynomial in $(\bfX_{I'},\bfA_{I'})$, with coefficients in $\Z$ and degree $\ll_{n,D} 1$. 
   Then upon specializing $\bfA_{I'} = \bfa_{I'} \in L^{|I'|}$,  $B_{\lin}(b_{\ell,m}(\bfX_{I'},\bfa_{I'}))=0$ if and only if $ H(Y,X_{1},X_{i}+a_{i}X_{1})_{i \in I'}$ has a linear factor in $Y$ over $\overline{\Q}$ (or the degree drops, as a polynomial in $Y$ and $X_1$). We claim:
   \begin{lem}\label{lemma_P_nonzero_NRAI}
    In the above setting, suppose $H(Y,\bX_I)$ is $|I|$-genuine. Suppose $|I| \geq 2$ or $1 \not\in I$, and define $I'=I \setminus \{1\}$ if $1 \in I$ or $I'=I$ if $1\not\in I$. Expand the polynomial $H(Y,X_1, X_i+A_iX_1)_{i\in I'}$  as in (\ref{app_H_expansion_X1.1}). 
      As a polynomial in $(\bfX_{I'},\bfA_{I'})$,  $B_{\lin}(b_{\ell,m}(\bfX_{I'},\bfA_{I'}))\not\equiv 0$.
   \end{lem}
   \begin{proof}

   Fix any index $i_k \in I'$, and fix a choice of $a_i \in L$ for each $i \in I' \setminus \{i_k\},$ and then consider the polynomial $H(Y,X_{1},X_{i}+a_{i}X_{1},X_{i_{k}})_{i\in I'\setminus\{i_{k}\}}$. This polynomial defines  an $\ell$-genuine extension for some $\ell \geq 1$ since it certainly has nontrivial degree with respect to $X_{i_k}$, for example.  Moreover, it is irreducible over $L(\bX_I)$, since by hypothesis  $L(\bfX_{I})/H(Y,\bX_I)$ is a field; hence   $H(Y,X_{1},X_{i}+a_{i}X_{1},X_{i_{k}})_{i\in I'\setminus\{i_{k}\}}$ is irreducible since it is obtained by a linear transformation. Now, apply Lemma \ref{lemma_app_shifting} with $R=H(Y,X_{1},X_{i}+a_{i}X_{1},X_{i_{k}})_{i\in I'\setminus\{i_{k}\}}$ in the variables $Y,X_1,\bX_{I'}$, and with the distinguished index $i_k$. The outcome is that  for all but $O_{n,D}(1)$ choices of $a_{i_{k}}\in L$, the 
  tuple $\bfa_{I'} := (a_i)_{i \in I'} = (a_i, a_{i_k})_{i \in I' \setminus \{i_k\}}$ and the definition 
  \[  H_{\bfa_{I'}}(Y,X_1,\bX_{I'}) := H(Y,X_{1},X_{i}+a_{i}X_{1},X_{i_{k}}+a_{i_k}X_1)_{i\in I'\setminus\{i_{k}\}} ,\]
 we have $\deg H_{\bfa_{I'}}=\deg H$, and moreover
  the field 
  \[ \mathcal{L}_{H_{\bfa_{I'}}} := L(X_1,\bX_{I'})[Y]/ (H_{\bfa_{I'}}(Y,X_1,\bX_{I'}))\]
  has the property
   \beq\label{NRAI}
(\mathcal{L}_{H_{\bfa_{I'}}}\cap\overline{L(\bX_{I'})})(X_{1})\subsetneq \mathcal{L}_{H_{\bfa_{I'}}}.
\eeq
Now by an application of Lemma \ref{lemma_II_implies_III_gen} (Theorem \ref{thm_gen_ppties} (II) for $i_0=1$ $\Rightarrow$ (III) for $i_0=1$), this implies that there exists a point $\bx_{I'} \in L^{|I'|}$ such that $H_{\bfa_{I'}}(Y,X_1,\bx_{I'})$, as a polynomial in $Y,X_1$, does not have a linear factor in $Y$ over $\overline{\Q}$, and has $\deg H_{\bfa_{I'}}(Y,X_1,\bx_{I'})=\deg_{Y,X_1} H_{\ba_{I'}}(Y,X_1,\bX_{I'})$. Hence for this $\bfa_{I'}$ and $\bx_{I'}$,
$B_{\lin}(b_{\ell,m}(\bfx_{I'},\bfa_{I'}))\neq 0$. Consequently, 
    $B_{\lin}(b_{\ell,m}(\bfX_{I'},\bfA_{I'}))\not\equiv 0$ as a polynomial in $\bX_{I'}$, and the lemma is proved.
  \end{proof}

\begin{proof}[Proof of Lemma \ref{lemma_shift_all_vars_R}]
   We suppose that $R$ is given, so that  $\mathcal{L}_R$ is an $\ell$-genuine extension of $L(X_1,\ldots,X_n)$ for some $\ell \geq 1$. As in (\ref{R_defines_H}), let $I\subset\{1,...,n\}$ be a set with $|I|=\ell$ such that there exists an $|I|$-genuine polynomial $H(Y,\bX_I) \in \Ocal_L[Y,\bX_I]$ of total degree $D$  such that
    \[
    \mathcal{L}_R=(L(\bX_I)[Y]/H(Y,\bX_I))(\bX_{I^c}).
    \]

        If $I=\{1\}$, then observe that $(\mathcal{L}_{H} \cap \overline{L(X_2,\ldots,X_n)})(X_1) \subsetneq  \mathcal{L}_{H}$. Since $I=\{1\}$, for each index $j \neq 1$, then $\deg_{X_j}H=0$ so that shifts $X_j \mapsto X_j+a_jX_1$ do not change $H$. Hence $\mathcal{L}_{R_{\bfa}} \cap \overline{L(X_2,\ldots,X_n)})(X_1) \subsetneq  \mathcal{L}_{R_{\bfa}}$ for all $\bfa \in \Z^{n-1}$, so we may simply define the polynomial $g^R(\bfa)=c_0$ for some $c_0 \neq 0$. Henceforward we may assume $1 \not\in I$ or $|I| \geq 2$.

  Let $I' = I$ if $1 \not\in I$ and $I'= I\setminus \{1\}$ if $1 \in I$. Recall from the convention above that in both cases, we denote the polynomial we study by $H(Y,X_1,X_i + A_iX_1)_{i \in I'}$.
Note that it suffices to construct a nonzero polynomial  $g_{I'}$ with coefficients in $\calO_K$, in $|I'|$ variables, and of degree $\deg g_{I'} \ll_{n,D,[L:K]}1$ with the following property: for all $\bfa_{I'}\in \calO_K^{|I'|}$ with $g_{I'}(\bfa_{I'}) \neq 0$, the polynomial 
  \[  H_{\bfa_{I'}}(Y,X_1,\bX_{I'}) := H(Y,X_{1},X_{i}+a_{i}X_{1})_{i\in I'} ,\]
has associated field 
  \[ \mathcal{L}_{H_{\bfa_{I'}}} := L(X_1,\bX_{I'})[Y]/ (H_{\bfa_{I'}}(Y,X_1,\bX_{I'}))\]
with the property
   \beq\label{NRAI'}
(\mathcal{L}_{H_{\bfa_{I'}}}\cap\overline{L(\bX_{I'})})(X_{1})\subsetneq \mathcal{L}_{H_{\bfa_{I'}}}.
\eeq
 From this it would follow that
 \[
    R_{\bfa_{I'}}(Y,X_{1},...,X_{n}):=R(Y,X_{1},X_{i}+a_{i}X_{1},X_{j})_{i\in I',j\not\in I'\cup\{1\}},
    \]
is such that 
\[
 (\mathcal{L}_{R_{\bfa_{I'}}}\cap\overline{k(X_{2},...,X_{n})})(X_{1})\subsetneq \mathcal{L}_{R_{\bfa_{I'}}}.
\]
  The final step is then to observe that for each index $j \not\in I' \cup \{1\}$, then $\deg_{X_{j}} H=0$, so that shifts of the form $X_j \mapsto X_{j}+a_{j}X_{1}$ do not change  $H$.   Thus we may consider $g_{I'}$ as a polynomial in $\calO_K[A_2,\ldots,A_n]$ (i.e. in $n-1$ variables), which we call $g^R$, and the desired conclusion of the lemma holds for all $\bfa \in \calO_K^{n-1}$ such that $g^R(\bfa) \neq 0$.

 To construct the desired polynomial $g_{I'}$ in variables $A_i$   for $i \in I'$ that leads to (\ref{NRAI'}), we will apply Lemma \ref{lemma_P_nonzero_NRAI}.
  Expand
   \[
   B_{\lin}(b_{\ell,m}(\bfX_{I'},\bfA_{I'}))=\sum_{|\bfk|\leq D'} g_{\bfk}(\bfA_{I'})\bfX_{I'}^{\bfk},
   \]
 in which $D' \ll_{n, \deg R}1$, and $\deg g_{\bfk} \ll_{n,\deg R} 1$ for each $\bfk$.
By Lemma \ref{lemma_P_nonzero_NRAI}, $B_{\lin}(b_{\ell,m}(\bfX_{I'},\bfA_{I'})) \not\con 0$ so we can find $\bfk$ such that $g_{\bfk}(\bfA_{I'})\not\equiv 0$. Note that while $B_\lin$ has coefficients in $\Z$ (as a function of $b_{\ell,m}(\bX_{I'},\bfA_{I'})$), in the expansion above, \emph{a priori} $g_{\bfk}(\bfA_{I'})\in \calO_L[\bfA_{I'}]$. To construct our desired polynomial $g_{I'}\in \calO_K[\bfA_{I'}]$, define
$$
g_{I'}(\bfA_{I'}) := \prod_{\sigma\in \Aut(L/K)} \sigma(g_{\bfk}(\bfA_{I'})),
$$
where $\sigma\in \Aut(L/K)$ act on the coefficients of $g_{\bfk}$; observe that all the coefficients of $g_{I'}$ lie in $\Ocal_K$ as desired. Also, $\deg {g_{I'}} \ll_{n,\deg R,[L:K]} 1$.  It remains to check that when $g_{I'}(\bfa_{I'})\neq 0$, then $B_{\lin}(b_{\ell,m}(\bX_{I'},\bfa_{I'}))\not\equiv 0$. This follows from the fact that $g_{I'}(\bfa_{I'})\neq 0$ implies that $g_{\bfk}(\bfa_{I'})\neq 0$ and hence $B_{\lin}(b_{\ell,m}(\bX_{I'},\bfa_{I'}))\not\equiv 0$. That is, when $g_{I'}(\bfa_{I'})\neq 0$, then the property in (IV) of Theorem \ref{thm_gen_ppties} holds for the index $i_0=1$.  By Theorem \ref{thm_gen_ppties} ((IV) for $i_0=1$ $\Rightarrow$ (II) for $i_0=1$), it follows that when $g_{I'}(\bfa_{I'})\neq 0$, then
\[ (\mathcal{L}_{H_{\bfa_{I'}}}\cap\overline{L(X_{2},...,X_{n})})(X_{1})\subsetneq \mathcal{L}_{H_{\bfa_{I'}}},
\]
verifying (\ref{NRAI'}).
This suffices to complete the proof of the lemma.
\end{proof}

\subsection{Step 3: Proof of Proposition \ref{prop_app_construct_shift}}

By convention, let $\mathcal{L}_R$ denote $\mathcal{L}_{R_\mathbf{a}}$ when $\mathbf{a} = \boldsymbol{0}.$ 
  If $R$ has the property that 
  \[
\mathcal{L}_{R}\cap\overline{L(X_{2},...,X_{n})}=L(X_{2},...,X_{n}),
    \]
then the conclusion of the proposition is true for $\bfa=0$, and we are finished.
Thus we reduce consideration to the case in which
$
\mathcal{L}_{R}\cap\overline{L(X_{2},...,X_{n})} \supsetneq L(X_{2},...,X_{n}).
    $
    In this case, let $K_{1},...,K_{e}$ denote all the nontrivial extensions of $L(X_1,\ldots,X_n)$ contained in $\mathcal{L}_R$ (including $\mathcal{L}_R$ itself).  Under the hypotheses of the proposition, an application of Lemma \ref{lemma_at_least_strongly_1_gen} shows that for each $j=1,\ldots, e$, $K_j$ is an $\ell_j$-genuine extension of $L(X_1,\ldots,X_n)$ for some $\ell_j \geq 1$.  For each $j$, let $P_j(Y,X_1,\ldots,X_n)$ be a polynomial, irreducible over $L(X_1,\ldots,X_n)$, and monic in $Y$, such that $K_j = L(X_1,\ldots,X_n)[Y]/P_j$. Define the field $K_{j,\bfa} = L(X_1,\ldots,X_n)[Y]/(P_j)_{\bfa}$ generated by the shifted polynomial $(P_j)_{\bfa}$, for each $\bfa \in \calO_K^{n-1}$.  We claim that for a fixed $\bfa$, the fields $K_{1,\bfa},...,K_{e,\bfa}$  also  enumerate all the nontrivial intermediate extensions of $L(X_{1},...,X_{n})$ contained in $\mathcal{L}_{R_{\bfa}}$. To see this, first observe that, for every polynomial $G$ and every $\bfa$,
    \[
    \begin{split}
    (G_{-\bfa})_{\bfa}&=G_{-\bfa}(X_{1},X_{2}+a_{2}X_{1},...,X_{n}+a_{n}X_{1})\\&=G_{\bfa}(X_{1},(X_{2}-a_{2}X_{1})+a_{2}X_{1},...,(X_{n}-a_{n}X_{1})+a_{n}X_{1})\\&=G.
    \end{split}
    \]
    Now let $L(\bfX)\subsetneq \mathcal{K} \subset\mathcal{L}_{R_{\bfa}}$ be a nontrivial intermediate extension and let $G(Y,\bfX)$ be monic in $Y$ and irreducible over $L$ such that $\mathcal{K}=L(X_1,\ldots,X_n)[Y]/G$; in the nomenclature above, $\mathcal{K}_{-\bfa} = L(\bX)[Y]/G_{-\bfa}$. Then $k(\bfX)\subsetneq \mathcal{K}_{-\bfa}\subset\mathcal{L}_{R}$, hence $\mathcal{K}_{-\bfa}=K_{j}$ for some $j\in\{1,...,e\}$ which implies that $K_{j,\bfa}=(\mathcal{K}_{-\bfa})_{\bfa}=L(X_1,\ldots,X_n)[Y]/(G_{-\bfa})_{\bfa}=L(X_1,\ldots,X_n)[Y]/G=\mathcal{K}$.  This verifies the claim.

 We claim that Lemma \ref{lemma_shift_all_vars_R}  proves the existence of a vector $\bfa\in\calO_K^{n-1}$ such that for every $j=1,\ldots,e$,
   \beq\label{Kja}
     K_{j,\bfa}\supsetneq ( K_{j,\bfa}\cap \overline{L(X_2,...,X_n)})(X_1) .
     \eeq
    Indeed, in the notation above, for each $j$ we can apply that lemma with $R$ chosen to be the polynomial $P_j$ such that $K_j =L(\bX)[Y]/P_j$, which as remarked above is $\ell_j$-genuine over $L$ for some $\ell_j \geq 1$.
  Then upon taking the polynomials $g^{P_1},\ldots,g^{P_e}$ provided by that lemma, it suffices to choose $\bfa \in \Ocal_K^{n-1}$ that is a root of none of these polynomials.

    By Lemma \ref{lemma_count_in_ring}, for any $M\geq 1$
    \begin{equation*}
        \#\{\bx\in \calO_K^{n-1}: \|\bx\|\leq M\} \asymp_m M^{n-1},
    \end{equation*}
   in which $m=[K:\Q]\ll_D 1$.
    On the other hand, by the trivial bound in Lemma \ref{lemma_Schwartz_domain}, for a fixed $j$ and any $M \gg 1$,  
    \[\{\bx\in  \calO_K^{n-1}:\|\bx\|\leq M, g^{P_j}(\bx)=0\} \ll_{n,m} (\deg g^{P_j})M^{n-2}.\] 
  For each $j$, $\deg g^{P_j}\ll_{n,D, [L:\Q]} 1$. 
  Thus in total the union over $j=1,\ldots,e$ of the roots of $g^{P_j}$ 
contains at most $\ll_{n,D,[L:\Q]} M^{n-2}$ points. By applying this for any $M \gg_{n,D, [L:\Q]} 1$, 
there must be some element $\bx\in  \calO_K^{n-1}$ with $\|\bx\| \leq M$ for which none of the $g^{P_j}$ vanishes; hence, taking $M\asymp_{n,D,[L:\Q]} 1$,  
 there exists a suitable $\bfa$ with $\|\bfa\|\ll_{n,D,[L:\Q]} 1$.

 As a consequence of (\ref{Kja}), we claim that for every $j=1,\ldots,e$,
     \beq\label{Kj_k_conclusion}
     K_{j,\bfa}\not\subset ( \mathcal{L}_{R_\ba}\cap \overline{L(X_2,...,X_n)})(X_1) .
     \eeq
    Indeed, if not, then $ K_{j,\bfa}\subset  \overline{L(X_2,...,X_n)}(X_{1})$, and hence \
     \[
     K_{j,\bfa}=( K_{j,\bfa}\cap \overline{L(X_2,...,X_n)})(X_1),
     \]
     which contradicts (\ref{Kja}).
     Now
    denote $N_{R_\ba}:=\mathcal{L}_{R_\ba} \cap \overline{L(X_2,...,X_n)}$ so that the right-hand side in (\ref{Kj_k_conclusion}) is $N_{R_\ba}(X_1)$.
     Since  $K_{1,\bfa},...,K_{e,\bfa}$ enumerate  all   nontrivial intermediate extensions of $L(X_1,\ldots,X_n)$ contained in $\mathcal{L}_{R_{\ba}}$, yet $K_{j,\ba} \not\subset N_{R_\ba}(X_1)$ for every $j=1,\ldots,e$,  it follows that $N_{R_{\ba}}(X_{1})=L(X_{1},...,X_{n})$, i.e. $\mathcal{L}_{R_\ba} \cap \overline{L(X_2,...,X_n)}=N_{R_\ba}=L(X_{2},...,X_{n})$.
   This completes the proof of Proposition \ref{prop_app_construct_shift}, and hence also the proof of Theorem \ref{thm_appendix_shift_to_strongly_genuine}.

\begin{rem}
It is reasonable to ask whether this strategy of shifting $F(Y,\bX)$ by a choice of ``short'' $\bfa \in \Z^{n-1}$ with   $\|\bfa\| \ll_{n,D} 1$ could allow the methods of  \cite{BPW25x} to apply to a broader class of polynomials than presently achieved in that paper. In order to apply the methods leading to \cite[Theorem 6.4]{BPW25x} and hence to \cite[Theorems 1.1-1.3]{BPW25x}, it suffices for $F$ to be strongly $(1,n)$-allowable; this is the requirement that $F(Y,L(\bX))$ is strongly $n$-genuine for all $L \in \GL_n(\Q)$. Proposition  \ref{prop_app_construct_shift} can shift a polynomial $R(Y,\bX)$ by a vector $\ba\in \Z^{n-1}$ with $\|\ba\| \ll_{n,D,[L:\Q]} 1$ to a polynomial satisfying (\ref{app_alg_closed}) (weaker than being strongly $n$-genuine, by Theorem \ref{thm_strongly_gen_ppties} (I) $\Rightarrow$ (II)), but it is not clear that such a ``short'' shift can produce a strongly $(1,n)$-allowable polynomial.
On the other hand, the methods leading to \cite[Theorem 6.4]{BPW25x} only require a weaker   property that for every $L$ belonging to a finite set $\mathscr{L}$ of linear transformations, $F_L(Y,\bX):=F(Y,L(\bX))$  satisfies $\mathcal{L}_{F_L}\cap \overline{\Q(X_2,\ldots,X_n)}= \Q(X_2,\ldots,X_n)$ for all $L \in \mathscr{L}$; the linear transformations in $\mathscr{L}$ may themselves have large norms.  
\end{rem}

 \subsection*{Acknowledgements}
 The authors thank S. D. Cohen for his encouragement to pursue these investigations, and S. Chow, R. Cluckers, M. Parades and R. Sasyk for helpful remarks and references. 
L.P.  has been partially supported during portions of this research by NSF DMS-2200470, a Joan and Joseph Birman Fellowship, a Simons Fellowship, and a Guggenheim Fellowship, and thanks the Hausdorff Center for Mathematics for hosting several research periods as a Bonn Research Chair; the Mittag-Leffler Institute for hosting a research period in 2024; Rhodes House and IMJ-PRG (hosted by R. de la Bret\`{e}che) in 2025. 
K.W. visited Duke in 2023 with funding from NSF RTG-2231514, which supports the Number Theory group at Duke University. K.W. is partially supported by NSF under DGE-2039656 and DMS-2502864.

\bibliographystyle{alpha}

\bibliography{NoThBib}

\end{document}